\documentclass[12pt]{amsart}

\usepackage{geometry}
\usepackage{amsmath,amssymb,amsfonts}
\usepackage{kpfonts}
\usepackage{mathtools}
\usepackage{bbm}
\usepackage{xcolor}
\usepackage{hyperref}
\usepackage{booktabs}
\usepackage[normalem]{ulem}

\newtheorem{theorem}{Theorem}

\numberwithin{equation}{section}
\numberwithin{theorem}{section}

\newtheorem{lemma}[theorem]{Lemma}

\newtheorem{corollary}[theorem]{Corollary}
\newtheorem{rem}[theorem]{Remark}
\newtheorem{prop}[theorem]{Proposition}
\newtheorem{ex}[theorem]{Example}

\renewcommand{\Re}{{\rm Re}}
\renewcommand{\Im}{{\rm Im}}
\newcommand{\zeros}{{\sf Zeros}\,}
\newcommand{\cl}{{\rm cl}\,}
\newcommand{\Cov}{{\rm Cov}\,}

\newcommand{\ii}{{\rm i}}
\newcommand{\ee}{{\rm e}}
\newcommand{\1}{\mathbbm{1}}

\renewcommand{\d}{{\rm d}}

\newcommand{\E}{\mathbb{E}}

\newcommand{\erf}{\operatorname{erf}}
\newcommand{\erfc}{\operatorname{erfc}}

\newcommand{\od}{\overset{{\rm d}}{=}}
\newcommand{\todistr}{\overset{{\rm d}}{\longrightarrow}}
\newcommand{\tofdd}{\overset{{\rm f.d.d.}}{\longrightarrow}}
\newcommand{\toas}{\overset{{\rm a.s.}}{\longrightarrow}}

\newcommand*\oline[1]{%
  \vbox{%
    \hrule height 0.5pt
    \kern0.25ex
    \hbox{%
      \kern-0.05em
      \ifmmode#1\else\ensuremath{#1}\fi
      \kern-0.05em
    }
  }
}

\makeatletter
\@namedef{subjclassname@2020}{\textup{2020} Mathematics Subject Classification}
\makeatother

\usepackage{nicefrac}
\usepackage{ulem}

\begin{document}

\title[On zeros of random polynomials with regularly varying coefficients]{On the distribution patterns of zeros for random polynomials with regularly varying coefficients}

\author{Zakhar Kabluchko}

\address{Zakhar Kabluchko: Institute of Mathematical Stochastics, Department of Mathematics and Computer Science, University of M\"{u}nster, Orl\'{e}ans-Ring 10, D-48149 M\"{u}nster, Germany}
\email{zakhar.kabluchko@uni-muenster.de}

\author{Boris Khoruzhenko}
\address{Boris Khoruzhenko: School of Mathematical Sciences, Queen Mary University of London, Mile End Road, London E14NS, United Kingdom}
\email{b.khoruzhenko@qmul.ac.uk}

\author{Alexander Marynych}
\address{Alexander Marynych: Faculty of Computer Science and Cybernetics, Taras Shevchenko National University of Kyiv, Kyiv 01601, Ukraine; School of Mathematical Sciences, Queen Mary University of London, Mile End Road, London E14NS, United Kingdom}
\email{marynych@knu.ua}


\begin{abstract}
This paper investigates asymptotic distribution of complex zeros of random polynomials $P_n(z):=\sum_{k=0}^{n}b(k)\xi_k z^k$, as $n\to\infty$, where $b$ is a regularly varying function at infinity with index $\alpha\in \mathbb{R}$ and $(\xi_k)_{k\geq 0}$ is a sequence of independent copies of a complex-valued random variable $\xi$. The limiting distribution of zeros both inside and outside the unit disk is determined assuming $\mathbb{E}[\log^{+}|\xi|]<\infty$. Under the additional assumptions $\mathbb{E}[\xi]=0$ and $\mathbb{E}[|\xi|^2]<\infty$, local universality results for zeros near the boundary of the unit disk are established. Notably, it is shown that the point process of zeros undergoes a transition from liquid-like to crystalline phases as $\alpha$ crosses the critical value $\alpha_c = -\nicefrac{1}{2}$ from right to left. In the liquid phase ($\alpha > \alpha_c$), the limiting point process of zeros is universal. In the crystalline phase, it is universal if and only if $\alpha = \alpha_c$ and $\sum_k b^2(k) = +\infty$ (the weak crystalline phase), and non-universal when $\sum_k b^2(k) < +\infty$ (the strong crystalline phase). The zeros of the so-called random self-inversive polynomials on the unit circle exhibit a similar phase transition.

\end{abstract}

\maketitle

\section{Introduction}

\subsection{Random polynomials with regularly varying coefficients.}
Let $(\xi_k)_{k\geq 0}$ be an infinite sequence of independent copies of a
complex-valued\footnote{The case $\mathbb{P}\{\xi\in\mathbb{R}\}=1$ is not excluded.} random variable $\xi$ with $\mathbb{P}\{\xi=0\}<1$. Further, let $b:[0,+\infty)\to \mathbb{R}$ be a measurable function which varies regularly at $+\infty$ with index $\alpha\in\mathbb{R}$, that is $b(x)=x^{\alpha}\ell(x)$, $x>0$, where $\ell$ satisfies
\begin{equation}\label{eq:slow_variation_def}
\lim_{x\to+\infty}\frac{\ell(\lambda x)}{\ell(x)}=1,\quad\text{for every}\quad \lambda>0.
\end{equation}
Any eventually positive measurable function $\ell$ fulfilling~\eqref{eq:slow_variation_def} is called slowly varying and typical examples include functions converging to a non-zero finite limit as $x\to+\infty$; $\log^{\beta} \!x$, for any $\beta\in\mathbb{R}$; $\exp({\log^{\gamma}\!x})$, for any $\gamma\in (0,1)$; and any iterations of the logarithm, like $\log\log x$, $\log\log \log x$, and so on. For a comprehensive treatment of slowly and regularly varying functions, we refer the reader to~\cite{BGT}.

This paper is concerned with random polynomials
\begin{equation}\label{eq:p_n_def}
P_n(z):=\sum_{k=0}^{n}b(k)\xi_k z^k= b(0)\xi_0+\sum_{k=1}^{n}k^{\alpha}\ell(k)\xi_k z^k,\quad n\in\mathbb{N},
\end{equation}
and their zeros in the limit of high degree $n\to\infty$.
It is known that under a mild moment assumption (see~\eqref{eq:xi_log_moment} below) the normalized zero counting measure of $P_n$ converges to the uniform measure on the unit circle~\cite{Hughes+Nikeghbali:2008,Ibragimov+Zaporozhets:2013,Kabluchko+Zaporozhets:2014}, and in this note we investigate finer details of the asymptotic distribution of complex zeros of $P_n$ and how these depend on the behavior of its  random coefficients $b(k)\xi_k$.

\subsection{Motivation and background.}
Our investigation is motivated by the influential work of Bogomolny, Bohigas and Leboeuf on zeros of random polynomials \cite{Bogomolny+Bohigas+Leboeuf:1992}, see also a more detailed account of this work in \cite{Bogomolny+Bohigas+Leboeuf:1996}. It argues that random polynomials of high degree with zeros close to the unit circle appear naturally in the semiclassical approximation of the spectra of quantum chaotic systems and investigates conditions to be imposed on the coefficients of random polynomials ensuring that their zeros are close to or lie on the unit circle.  One of the conclusions in~\cite{Bogomolny+Bohigas+Leboeuf:1992} was that as the second moments of the coefficients increase, the zeros, which initially form a crystal-like circular lattice of an appropriate radius (Figure~\ref{Fig:1}, left), progressively spread into a liquid-like configuration (Figure~\ref{Fig:1}, right). This conclusion was based on heuristic analysis of zeros of polynomials with Gaussian coefficients, and it was conjectured in \cite{Bogomolny+Bohigas+Leboeuf:1992} that the transition between the crystalline and liquid-like patterns of distribution of zeros was a general phenomenon for random polynomials.

In the present work, we investigate this conjecture quantitatively and prove it within the broad framework of the family of random polynomials $P_n$ defined in~\eqref{eq:p_n_def}. Although it may appear more natural to consider random polynomials of the form $P_n(z) = \sum_{k=1}^n k^\alpha \xi_k z^k$ (without a slowly varying factor $\ell$), we show that at the critical threshold $\alpha = -\nicefrac{1}{2}$ the situation becomes more delicate. Inclusion of the slowly varying factor $\ell$ provides additional insight into crossover from the crystalline  to the liquid-like phases and highlights the existence of two distinct sub-phases (one is universal and the other is not) of the crystalline phase.

Our approach can be summarized as follows. We first select an appropriate observation window containing on average finitely many zeros~--~either a macroscopic window (for studying zeros in compacts inside or outside the unit disk) or a microscopic window (for studying zeros near the boundary of the unit disk). Within this window, we show that the sequence of (possibly rescaled) polynomials $P_n$ converges locally uniformly to a limiting random analytic process. Then, by Hurwitz’s theorem (recalled below), the zero set of $P_n$ converges, as a point process, to the zero set of this limiting process. Similar approach has been applied to random trigonometric polynomials in~\cite{Iksanov+Kabluchko+Marynych:2016}, real zeros of Kac polynomials (the random polynomials $P_n$ with $b(k)\equiv 1$) close to $1$ in~\cite{Michelen:2021}, complex zeros of the partition function of the (generalized) Random Energy Model in~\cite{Kab+Klim:2014,Kab+Klim:GREM:2014} and for random Dirichlet series in~\cite{Buraczewski+Dong+Iksanov+Marynych:2023}.  A description of the general framework and further examples can be found in~\cite{Shirai:2012}.

\begin{figure}[ht]
\begin{center}
\includegraphics[width = 0.38\textwidth, trim = 60 30 60 30, clip, angle = 0]{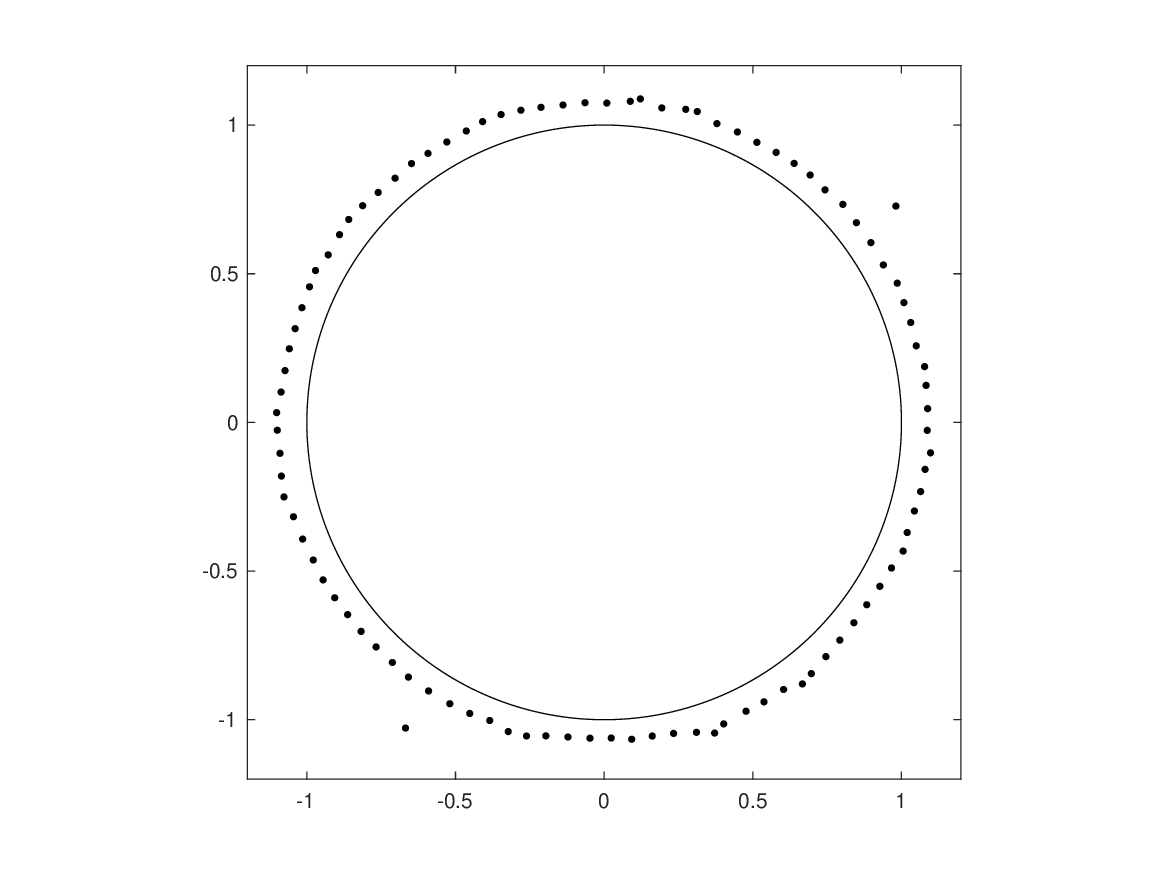}
\includegraphics[width = 0.38\textwidth, trim = 60 30 60 30, clip, angle = 0]{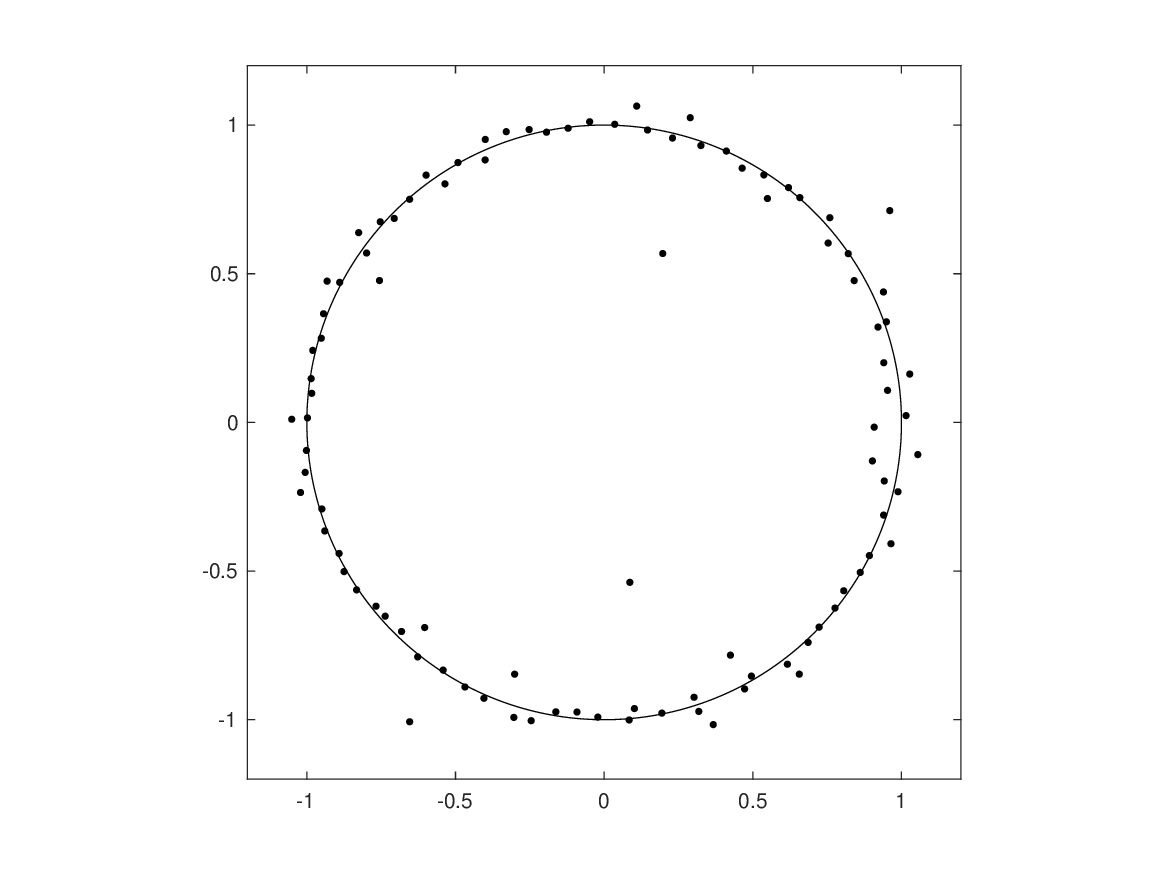}
\end{center}
\caption{Zeros of the random polynomial $P_n(z)= \sum_{k=0}^{n} (k+1)^{\alpha} \xi_k z^k$ of degree $n=100$ with $\alpha = -2$ (left, crystalline phase)
and $\alpha = 0$ (right, liquid phase).  The random variables $\xi_k$ are i.i.d.\ standard complex normals. The zeros are represented by dots and the unit circle is shown as the solid line.}
\label{Fig:1}
\end{figure}

\begin{figure}[ht]
\begin{center}
\includegraphics[width=0.6\textwidth ]{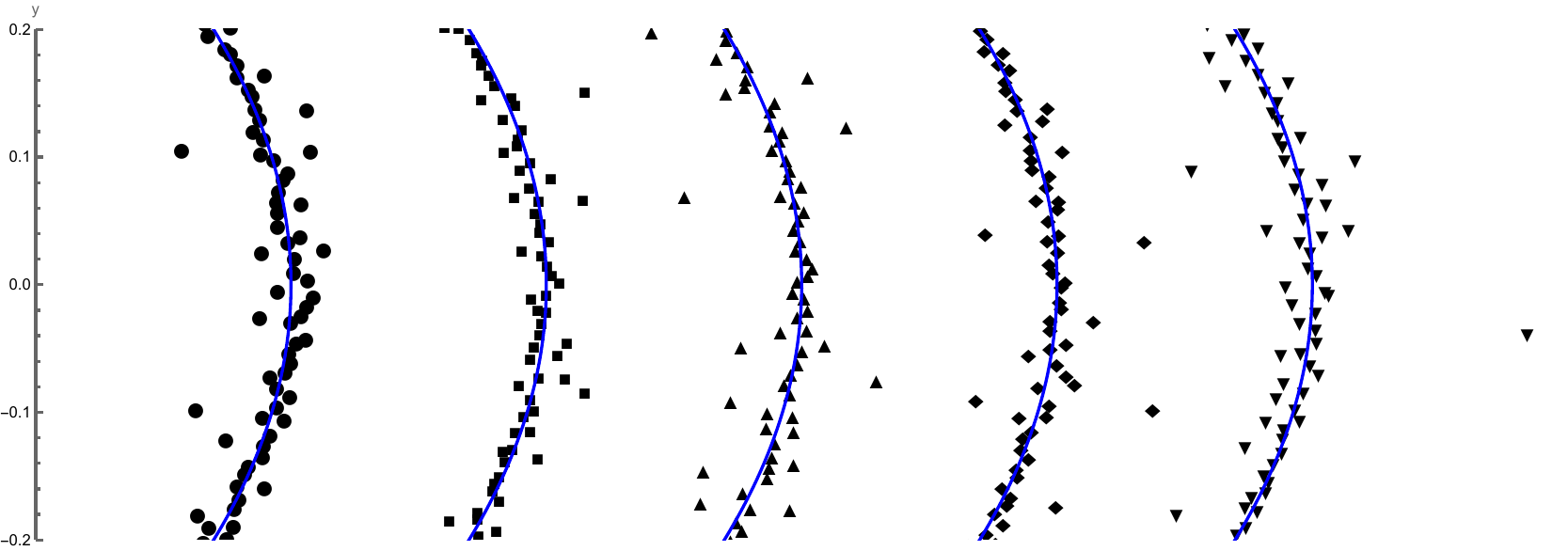}
\includegraphics[width=0.6\textwidth ]{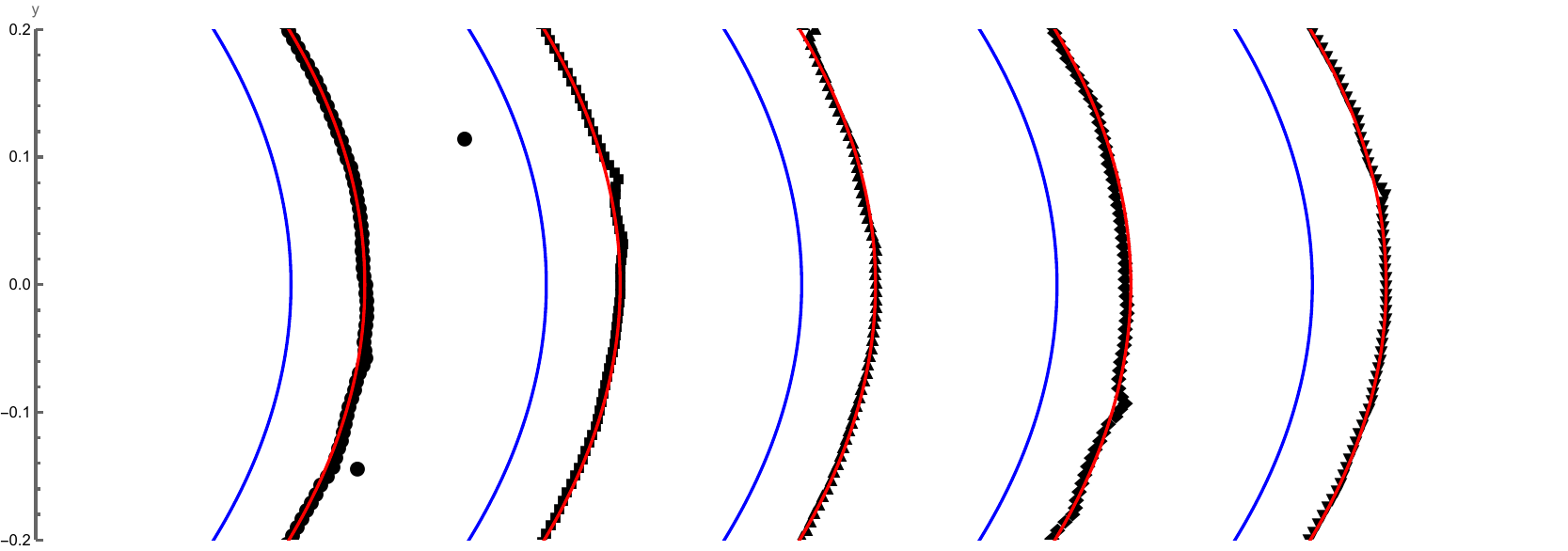}
\end{center}
\caption{Each row shows zeros of $5$  independent realizations of the random polynomial $P_n(z) = \sum_{k=0}^{n} (k+1)^{\alpha} \xi_k z^k$ with $\alpha = 0$ (top row, liquid phase) and $\alpha = -3$ (bottom row, crystalline phase). The random variables $\xi_k$ are i.i.d.\ complex standard  normal. The degree is $n=1000$. The unit circle is shown in blue. In the bottom row, the circle of radius $r_n$ is shown in red. Only zeros contained in the window  $[0.9,1.1] \times [-0.2,0.2]$ are shown. }
\label{Fig:2}
\end{figure}

\subsection{Main findings.}
Our key finding is the existence of a transition 
in the \emph{local} pattern of complex zeros of $P_n$ (in the limit $n\to\infty$) from an asymptotically equidistant one-dimensional lattice  along small arcs of a centered circle whose radius tends to $1$ to a  
genuinely  two-dimensional point process in the $1/n$-neighborhood of the unit circle, see Figures~\ref{Fig:1} and \ref{Fig:2} for illustration. Specifically, under the assumption that $\xi$ has zero mean and finite second moment, we prove,  that:
\begin{enumerate}
\item[(i)] If  $\alpha>-\nicefrac{1}{2}$, then for any finite family of non-empty compact sets $K_j\subset \mathbb{C}$ and any angles $\psi_j \in [0,2\pi)$, $j=1,\ldots,m$, the zeros of $P_n$ in the scaling windows
$$
\ee^{\frac{K_j}{n}+\ii \psi}=\{\ee^{\frac{u}{n}+\ii \psi_j}: u\in K_j\},\quad j=1,\ldots,m,
$$
attached to the unit circle at points $\ee^{\ii \psi_j}$  are asymptotically distributed as zeros of  certain non-trivial \emph{Gaussian} entire random functions, see Theorems~\ref{thm:boundary:alpha>-1/2} ($m=1$) and~\ref{thm:boundary:alpha>-1/2_joint} ($m\geq 2$).
Following the terminology of  \cite{Bogomolny+Bohigas+Leboeuf:1992} we call this regime {\bf the liquid phase}.

\item[(ii)] If $\alpha\leq-\nicefrac{1}{2}$, then the zeros of $P_n$ in the scaling windows $r_n\ee^{\frac{K_j}{n}+\ii \psi_j}$, attached to the centered circle of radius $r_n$, asymptotically form one-dimensional lattices with equal spacings between the neighbouring zeros. As $n \to \infty$, the radius $r_n$ approaches $1$ at a rate slower than $n^{-1}$ and the lattices are randomly shifted away from the circle by distances $n^{-1}\!\omega_j$, see Theorems~\ref{thm:boundary:alpha<-1/2} and~\ref{thm:critical}.  We call this regime {\bf the crystalline phase}.  If $\sum_{k\geq 0}b^2(k)<\infty$, then the shifts $\omega_j$ depend on the whole collection of the variables $(\xi_k)_{k\geq 0}$ ({\bf the strong crystalline phase}). Otherwise, $\omega_j$'s are determined by the diffusive limit of $(\xi_k)_{k\ge 0}$, with the contribution of each individual term in the collection being negligible ({\bf the weak crystalline phase}).
\end{enumerate}

For the reader's convenience we summarize our main findings in Table~\ref{tab:1}. Although our results hold for any finite collection of arguments $\psi_1, \ldots, \psi_m \in [0, 2\pi)$, for brevity we present them only for a fixed $\psi \in [0, 2\pi)$. The zeros of the random entire functions $u\mapsto P_n\left(r_n\ee^{\frac{u}{n}+i\psi}\right)$ converge, in the sense of weak convergence of point processes in $\mathbb{C}$, to the zeros of the entire functions in the penultimate column. In the liquid phase the radius $r_n=1$ for all $n$, and the scaling windows are centered on the unit circle. In the crystalline phase $r_n$ approaches $1$ from above as $n\to\infty$. We give the leading-order term of $\log r_n$ in the table. Exact formulae for $r_n$ in the crystalline phases are given in terms of the secondary branch of the Lambert-$W$ function and can be found in equations \eqref{eq:z_n_alpha<=-1/2}, \eqref{eq:z_n_alpha<-1/2} and \eqref{eq:z_n_alpha=-1/2}.

\begin{table}[ht]
\label{table:1}
  \centering
  \caption{The limiting forms of $\frac{1}{c_n}P_n(r_n\ee^{\frac{u}{n}+\ii \psi})$ as $n\to\infty$ for the random polynomials $P_n$ of degree $n$ with i.i.d.\ $\xi_k$ satisfying \eqref{eq:moment_assumptions_on_xi1} -- \eqref{eq:moment_assumptions_on_xi2}. In the penultimate column of the table, $G_{\psi}(u)$ is a centered Gaussian entire function with covariance structure~\eqref{eq:covariance:Ga}-\eqref{eq:covariance:Gb}, $N_1$ and $\hat N_1$ are independent centered complex Gaussians with the common distribution given by equations~\eqref{eq:limit_gaussian_cov_alt}; $N_1$ is also independent of a random process $P_{\infty}$ defined by formula~\eqref{eq:p_infty_def} below.}
    \begin{tabular}{cccccc}
        \toprule
           $\alpha$                & $L(n)=\int_1^n \frac{\ell^2(t)}{t}dt$             & $c_n$             & $\log r_n$            & Limiting Form       & Phase           \\[2ex]
        \midrule
          $\alpha < - \nicefrac{1}{2}$   & ---         & $1$          & $\sim -(\alpha+\nicefrac{1}{2})\frac{\log n}{n}$     & $P_{\infty} (\ee^{\ii\psi}) + \ee^{u} N_1$   & Str.~cryst.         \\[2ex]
           $\alpha = - \nicefrac{1}{2}$ & $ L(+\infty) < +\infty$       & $1$          &    $\sim \frac{\log{(1/\ell(n))}}{n}$        &   $P_{\infty} (\ee^{\ii\psi}) + \ee^{u} N_1$    & Str.~cryst.      \\[2ex]
          $\alpha = - \nicefrac{1}{2}$     & $L(+\infty) = +\infty$        & $ L^{\nicefrac{1}{2}}(n)$        & $\sim\frac{{\log (L(n)/\ell^2(n))}}{2n}$        & $\hat N_1 +\ee^{u} N_1$   & Weak~cryst.       \\[2ex]
         $\alpha > - \nicefrac{1}{2}$      & ---          & $n^{\alpha+\nicefrac{1}{2}}\ell(n) $          & $0$     & $G_{\psi}(u)$     & Liquid        \\[2ex]
                 \bottomrule
    \end{tabular}
  \label{tab:1}
\end{table}

\subsection{Universality of limiting point processes of zeros and crossover from liquid to crystalline phases.}
Fix $\psi\in [0,2\pi)$ and a non-empty compact set $K$ and consider the scaling window
$\ee^{\frac{K}{n}+\ii \psi}$ attached to the unit circle at $\ee^{\ii \psi}$.
In the liquid phase ($\alpha > -\nicefrac{1}{2}$) the limiting distribution of zeros of $P_n$, $n\to\infty$, in this scaling window is universal. Indeed, the Gaussian analytic function $G_{\psi}$ that determines the scaling limit of zeros depends only on the index of regular variation $\alpha$ and on the variances of $\Re \xi$ and $\Im \xi$. Similarly, the distribution of zeros in the weak crystalline phase is universal too. Indeed, in this case the limiting distribution of zeros of $P_n$ in the scaling window  $r_n \ee^{\frac{K}{n}+\ii \psi}$ attached to the centered circle of radius $r_n\to1$ is determined by the Gaussian analytic function $u\mapsto \hat N_1+\ee^{u}N_1$ which only depends on the variances of $\Re \xi$ and $\Im \xi$. If $\sum_{k\ge 0} b^2(k)<\infty$, then the scaling limit of the distribution of zeros of $P_n$ in the crystalline phase is no longer universal. It depends on the distribution of $\xi$, on the angle $\psi$ (and also on the value of $\alpha$), see the first two rows in Table~\ref{tab:1}. Further distinctions between weak and strong crystalline phases, aside from the issue of universality, are discussed in Remarks~\ref{rem:zeros_in_annulus_weak_cryst_phase} and~\ref{rem:zeros_in_annulus_strong_cryst_phase} below.

Interestingly, transition from the liquid to the weak crystalline phase as well as one from the strong crystalline to the weak crystalline phases are continuous. To make this statement precise, put
\begin{equation}\label{eq:s_gamma_def}
S(\gamma; \ell):=\sum_{k\geq 1}k^{\gamma}\ell^2(k)\in (0,+\infty],\quad \gamma\in\mathbb{R}.
\end{equation}
Note that $S(\gamma; \ell)<+\infty$ if $\gamma<-1$, $S(\gamma; \ell )=+\infty$ if $\gamma>-1$ for every slowly varying function $\ell$, whereas $S(-1; \ell)$ might be finite or equal to $+\infty$ depending on $\ell$. Note that $S(2\alpha; \ell)=\sum_{k\geq 1}b^2(k)$ and  $S(-1; \ell)<+\infty$ if and only if $L(+\infty)<+\infty$ in Table~\ref{tab:1}.

The transition from the liquid to the weak crystalline phase occurs when $\alpha\to -\nicefrac{1}{2}+0$, whereas the transition from the strong to the weak crystalline phase occurs when the sequence $(\alpha_m, \ell_m)_{m\geq 1}$ defining the coefficients of our polynomials, satisfy
\begin{equation}\label{eq:b_m_crossover_strong_weak}
S(2\alpha_m; \ell_m)<+\infty,\text{ for every }m\geq 1,
\quad
\text{and}\quad \lim_{m\to\infty}S(2\alpha_m; \ell_m)=+\infty.
\end{equation}
Both aforementioned transitions are continuous, as will be explained below.

We prove in Section~\ref{sec:crossover} that in the limit when $\alpha$ tends to $-\nicefrac{1}{2}$ from above, the zero set of $G_{\psi}$ (liquid phase), which does not depend on $\ell$, converges in distribution after an appropriate centering to the zero set of  $u\mapsto \hat N_1+\ee^{u}N_1$ (weak crystalline phase), see Theorem~\ref{thm:crossover} and Corollary~\ref{cor:crossover}. In the case of real and imaginary parts of $\xi$ having equal variances, the marginal distributions of $G_{\psi}$ do not depend on $\psi$ and are isotropic Gaussian. In this case one can also prove convergence of the first intensities of the zero set of $G_{\psi}$ to the first intensities of the zero set of $u\mapsto \hat N_1+\ee^{u}N_1$, see formula~\eqref{eq:convergence_of_intensities_right}. This convergence is visualised in Figure~\ref{Fig:3}.

Similarly, it can be shown that when the sequence $(\alpha_m, \ell_m)_{m\geq 1}$, is chosen such that~\eqref{eq:b_m_crossover_strong_weak} holds true, then under mild regularity assumptions, the sequence of point processes of zeros of random analytic functions $u\mapsto P_{m,\infty}(\ee^{\ii \psi}) + \ee^{u}N_1=\sum_{k\geq 0}b_m(k)\xi_k \ee^{\ii \psi k}+\ee^{u N_1}$ (strong crystalline phase) converges after an appropriate centering to the zero set of the Gaussian analytic function $u\mapsto \hat N_1+\ee^{u}N_1$ (weak crystalline phase) when $m\to\infty$. See Proposition~\ref{prop:left-hand-side-crossover} for the precise statement.

\begin{figure}[ht]
\begin{center}
\includegraphics[width = 0.6\textwidth]{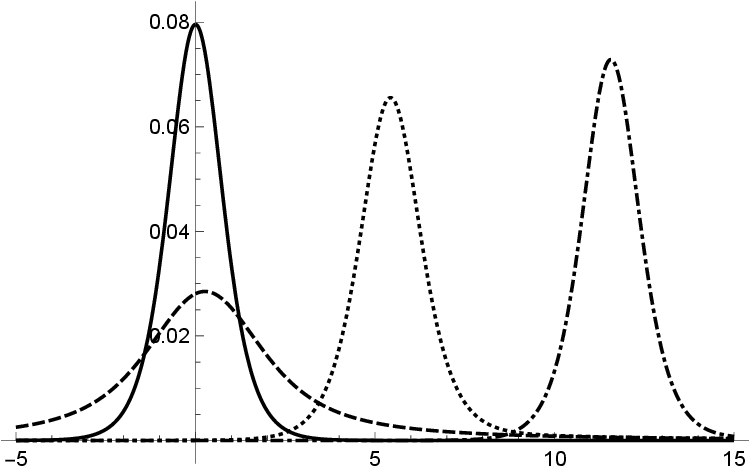}
\end{center}
\caption{Graphs of the intensity $\rho_1 (\alpha, \cdot )$ of zeros of $G_{\psi}$ restricted to the real axis for  $\alpha = -0.1$ (dashed line), $\alpha= -0.5 +10^{-4}$ (dotted line) and  $\alpha=- 0.5 +10^{-9}$ (dash-dotted line).  The solid line is the graph of the intensity of zeros of the Gaussian analytic function $u\mapsto \hat N_1+\ee^{u}N_1$ restricted to the real axis and is the limiting form of $\rho_1 (\alpha, \cdot )$, as $\alpha\to-\nicefrac{1}{2}+0$, after an appropriate shift. The intensity $\rho_1 (\alpha, \cdot )$ is defined in Eqs.~\eqref{eq:rho}--\eqref{eq:Phi_def} and $G_{\psi}$ is assumed to have isotropic Gaussian marginals.}
\label{Fig:3}
\end{figure}

\subsection{Two examples of random polynomials with regularly varying variance profile.}
As an example of random polynomials having complex zeros near the unit circle in the liquid phase we mention a family of the so-called hyperbolic random polynomials
\begin{align*}
P_n^{(2\alpha)}(z):= \sum_{k=0}^n b_{2\alpha}(k)  \, \xi_k z^k, \quad b_{2\alpha}(k) = \frac{\sqrt{(2\alpha+1)(2\alpha+2) \cdots (2\alpha+k)}}{\sqrt{k!}}, \quad   \alpha >-\nicefrac{1}{2}.
\end{align*}
Here, $x\mapsto b_{2\alpha}(x)$ is regularly varying at $+\infty$ with index $\alpha$. This family includes the Kac polynomials $P_n^{(0)}$ as the special case $\alpha=0$. When $\xi$ has the standard complex normal distribution, the zero set of the limiting Gaussian analytic function
\begin{align*}
P_{\infty}^{(2\alpha)}(z):=\lim_{n\to\infty} P_{n}^{(2\alpha)}(z), \quad z \in \mathbb{D}_1:=\{ z\in \mathbb{C}: |z|<1\},
\end{align*}
has  intensity $z\mapsto (2\alpha+1)\pi^{-1}(1-|z|^2)^{-2}$ which is invariant under the M\"obius transformations of $\mathbb{D}_1$, see Remark 2.4.5 in~\cite{HKPV}. This explains the word hyperbolic in the name of the families $P_{\infty}^{(2\alpha)}$ and $P_{n}^{(2\alpha)}$.
The Kac polynomials $P_n^{(0)}$ with Gaussian coefficients is a very special case. In this case the zero set of $P_{\infty}^{(0)}$
is a determinantal point process whose properties are well understood \cite{HKPV,Peres+Virag:2005}. It turns out, see  Theorem~\ref{thm:1} and its Corollary~\ref{cor:outside_unit_disk} below, that this determinantal process appears to be universal in the context  of Gaussian polynomials with regularly varying coefficients. For every $\alpha\in \mathbb{R}$ the zeros of random polynomials $P_n$ \eqref{eq:p_n_def} converge, outside the unit disk $|z|>1$, to the point process obtained by inversion of the zero set of $P_{\infty}^{(0)}$, and if $\xi$ follows the standard complex normal distribution, then the latter is determinantal.

The simplest example of $P_n$ at the critical threshold $\alpha=-\nicefrac{1}{2}$ is given by
\begin{equation}\label{eq:critical}
P^{{\rm crit}}_{n}(z)=\sum_{k=1}^n \frac{\xi_k}{\sqrt{k}} \,  z^k.
\end{equation}
According to our trichotomy, zeros of these polynomials are in the weak crystalline phase. If $\mathbb{E}[\xi]=0$ and $\sigma^2:=\mathbb{E}[|\xi|^2]\in (0,+\infty)$, then $P^{{\rm crit}}_n$ converge in $\mathbb{D}_1$, as $n\to\infty$, to a random analytic function
$$
P^{{\rm crit}}_{\infty}(z)=\sum_{k=1}^\infty \frac{\xi_k}{\sqrt{k}} \,  z^k, \qquad z\in \mathbb{D}_1,
$$
with the logarithmic covariance  $\mathbb{E}[P^{{\rm crit}}_{\infty}(z) \oline{P^{{\rm crit}}_{\infty}(w)}] =\sigma^2 \log (1- z\overline{w})$, $z,w\in\mathbb{D}_1$. Interestingly, when $\xi$ follows the standard complex normal distribution, the Gaussian analytic function $P^{\mathrm{crit}}_{\infty}$ appears in several other contexts. First, $P^{\mathrm{crit}}_{\infty}$ arises as the distributional limit (in large dimensions) of the logarithm of the characteristic polynomial $\det({\bf Id}_n - z U_n)$, where $U_n$ is a Haar-distributed unitary matrix of dimension $n$. This result follows easily from the findings in~\cite{Diaconis+Shahshahani:1994} but we provide a brief proof in Proposition~\ref{prop:char_poly_unitary} below for completeness.
Second, the radial limit of $P_{\infty}(r\ee^{\ii\theta})$ as $r \to 1-$ defines a (generalized) logarithmically correlated Gaussian process in the angular variable  $
\theta$~\cite{Hughes+Keating+OConnell:2001} which can be used to study the freezing transition for the free energy in the log-correlated Random Energy Model \cite{Fyodorov+Bouchaud:2008}, and, in a more general context, to construct an example of Gaussian multiplicative chaos, see, for example, 
~\cite[pp.~2248-2249]{Chhaibi+Madaule+Najnudel:2018} and reference therein.
The real zeros of $P^{{\rm crit}}_n$ for real-valued $\xi$ have been analyzed in~\cite{Krishnapur+Lundberg+Nguyen:2022},
where it was found that $P^{{\rm crit}}_n$ is a critical threshold in the family of random polynomials with power-law growing coefficients separating the logarithmic growth of the expected number of real zeros in the interval $(-1,1)$, as $n\to\infty$, from no growth at all. For further details we refer the interested reader to the original paper~\cite{Krishnapur+Lundberg+Nguyen:2022}.

\subsection{Self-inversive random polynomials with regularly varying coefficients}
Our findings, see Table~\ref{tab:1}, can be applied to investigate zeros of the so-called self-inversive random polynomials. Any self-inversive polynomial $\mathfrak{P}$ satisfies the relation
\begin{align}\label{eq:SI}
\mathfrak{P}(z)=z^n\oline{\mathfrak{P}(1/ \,\oline{z})}, \quad n = \deg \mathfrak{P},
\end{align}
and its zeros off the unit circle come in pairs $(z, {1}/\,{\oline{z}})$. In the study of random polynomials, the self-inversive property~\eqref{eq:SI} was employed in~\cite{Bogomolny+Bohigas+Leboeuf:1992} as a simple condition to ensure that at least a finite fraction of the zeros lies on the unit circle. This provides partial justification for using random polynomials to model the spectra of quantum chaotic systems, analogous to how random matrix theory is used to model the spectra of heavy nuclei. It was claimed in~\cite{Bogomolny+Bohigas+Leboeuf:1992} and demonstrated in~\cite{Bogomolny+Bohigas+Leboeuf:1996} that the random self-inversive polynomials
\begin{multline*}
\mathfrak{P}_{m} (z) = 1+\sigma \sum_{k=1}^{2m} N_k z^k + z^{2m+1}, \text{where $N_1,\ldots,N_m$ are independent}\\ \text{standard complex normal random variables and } N_{2m+1-k}=\oline{N}_k,\quad m\in\mathbb{N},
\end{multline*}
in the regime when $\sigma=\sigma_m$ depends on $m$ and
\begin{align}\label{eq:scalinglimit}
\lim_{m\to\infty} \sigma_m  \sqrt{2m+1}  = \epsilon \in (0,\infty)
\end{align}
have, on average, a finite fraction of zeros lying on the unit circle. Motivated by the results on self-inversive polynomials obtained in~ \cite{Bogomolny+Bohigas+Leboeuf:1992, Bogomolny+Bohigas+Leboeuf:1996} we investigate the family of self-inversive random polynomials
\begin{multline}\label{eq:K_m_polys_def}
K_m(z):=1+\sum_{k=1}^{m}b(k)\xi_k z^k + \sum_{k=1}^{m}b(k)\oline{\xi_k} z^{2m-k+1} + z^{2m+1}\\
=1+P_m(z)+z^{2m+1}\oline{P_m(\oline{z}^{-1})}+z^{2m+1},\quad m\in\mathbb{N},
\end{multline}
with regularly varying $b$. Whilst the focus of~\cite{Bogomolny+Bohigas+Leboeuf:1992,Bogomolny+Bohigas+Leboeuf:1996} was on investigating the regime~\eqref{eq:scalinglimit}, we investigate the dependence of the distribution patterns of zeros on the parameter $\alpha$. We find that for every $\alpha > -\nicefrac{1}{2}$ the zeros of  $K_m$ on the unit circle in a scaling window of length $1/m$ centered at $z=1$ are asymptotically distributed as zeros of a certain Gaussian analytic function in the limit $m\to\infty$. And if $\alpha \le  -\nicefrac{1}{2}$ the zeros of $K_m$ lying on the unit circle in the same window asymptotically form a  lattice along the unit circle with equal spacing between the neighbouring zeros  and a random angular shift.

\subsection{Relevant work on random polynomials.}
The polynomials $P_n$ with regularly varying coefficients studied in this work constitute a variant of random polynomials with coefficients exhibiting polynomial growth, as analyzed in~\cite{Do+Nguen+Vu:2018,Do:2021,Do+Nguyen:2025}; see Eq.~(2.5) in~\cite{Do+Nguen+Vu:2018} for the definition of this class. However, neither class is a subset of the other, and in addition, unlike the present paper, which focuses on complex zeros, the work presented in the cited papers focuses on real zeros and assumes $\alpha> -\nicefrac{1}{2}$.

The asymptotic properties of complex zeros of random polynomials have been much studied,  see~\cite{Nguyen+Vu:2022} and~\cite{Pritsker+Yeager:2015} and the references therein. We mention here only work that is most relevant in the context of our investigation. Uniform concentration of zeros of $P_n$ near the unit circle in the limit $n\to\infty$ was established in ~\cite{Hughes+Nikeghbali:2008,Ibragimov+Zaporozhets:2013,Kabluchko+Zaporozhets:2014}. The intensity of zeros in the annulus of width $O(1/n)$ around the unit circle, which is a refinement of the concentration of zeros, was obtained for the Kac polynomials in~\cite{Ibragimov+Zeitouni:1997,Shepp+Vanderbei:1995}. This result holds on the global scale in the sense that the angular coordinate in the annulus is not fixed. A local extension of this result (in the window of size $O(1/n)$ attached to a fixed point $\ee^{\ii \psi}$) for random polynomials $P_n$ with index of regular variation $\alpha > -\nicefrac{1}{2}$ can be inferred from our Theorem~\ref{thm:boundary:alpha>-1/2}, see Example~\ref{ex:edelman_kostlan}, and in particular equation~\eqref{eq:rho} there. The latter equation suggests that the global profile of the intensity of zeros of $P_n$, $\alpha > -\nicefrac{1}{2}$, should be given by Eq.~ \eqref{eq:intenstiy_rho_1_int}. This is indeed so if the coefficients of $P_n$ have the isotropic Gaussian distribution, see Example~\ref{ex:edelman_kostlan}. We note that Eq.~\eqref{eq:intenstiy_rho_1_int}  was also verified in~\cite{Bleher+Ridzal:2002} for Gaussian hyperbolic polynomials $P_n^{(2\alpha)}$.

When the distribution of the coefficients of $P_n$ is isotropic Gaussian, analysis of zeros is facilitated by the Kac-Rice and Edelman-Kostlan formulae, see~\cite{HKPV}. For instance, the limiting joint intensities of zeros for the Gaussian hyperbolic polynomials $P_n^{(2\alpha)}$ as $n\to\infty$ on compact sets inside and outside the unit disk $\mathbb{D}_1$ were obtained in~\cite{Bleher+Ridzal:2002}. The authors observe that outside $\mathbb{D}_1$ the limiting joint intensities do not depend on the value of $\alpha$. Our Theorem \ref{thm:1} explains why this generally holds true for random polynomials $P_n$ with index of regular variation  $\alpha > - \nicefrac{1}{2}$. The scaling limit of joint intensities of zeros of Kac polynomials with isotropic Gaussian coefficients in the window of size $O(1/n)$ attached to  fixed points at the boundary of $\mathbb{D}_1$ was obtained in \cite{Shiffman+Zelditch:2003}  along with an extension of this result to Kac-like polynomials with the monomials $z^k$, $k\ge 0$, replaced by the orthonormal polynomials on the boundary of a bounded simply connected analytic planar domain. The universality of local correlations of zeros at the boundary of $\mathbb{D}_1$ was established for the (non-Gaussian) Kac polynomials in~\cite{Tao+Vu:2014} and later extended to random polynomials with coefficients exhibiting polynomial growth in \cite{Do+Nguen+Vu:2018}, see also \cite{Nguyen+Vu:2022} for general universality framework for random polynomials.

Finally, persistence probabilities for real zeros of random polynomials with regularly varying variance profile were studied in~\cite{Dembo+Mukherjee:2015}, and the real zeros of Kac-like polynomials in the scaling window of length $O(1/n)$ centered at the accumulation points $\pm1$ on the unit circle were investigated in ~\cite{Aldous+Fyodorov:2004} and~\cite{Michelen:2021}. Additional results on real zeros can be found in~\cite{Flasche+Kabluhcko:2020} and in references therein.

\subsection{Organization of paper.}
The paper is organized as follows. In Section~\ref{sec:zeros_away_from_the_boundary} we prove, for completeness, some simple facts about the asymptotic distribution of zeros away from the boundary of the unit disk. The rest of the paper is devoted to the behavior of zeros near the boundary of the unit disk. The liquid phase ($\alpha>-\nicefrac{1}{2}$) is treated in Section~\ref{sec:zeros_alpha>-1/2}, and the crystalline phases are treated in Sections~\ref{sec:zeros_alpha<=-1/2} and~\ref{sec:zeros_alpha=-1/2}. The results obtained there also include the description of the joint law for point processes of zeros in scaling windows attached to the \emph{different} points on the unit circle. Section~\ref{sec:crossover} is devoted to analyzing the limiting crossover from the liquid phase to the strong crystalline phase, and from the strong crystalline phase to the weak crystalline phase. In Section~\ref{sec:si_polys} we apply our results to the analysis of the self-inversive random polynomials for which a proportion of zeros lies exactly on the unit circle. In Section~\ref{section:5} we analyse the asymptotic behaviour of the average of the proportion of zeros of self-inversive polynomials with regularly varying coefficients lying on the unit circle in the particular case of isotropic Gaussian coefficients. Finally, two Appendices contains some auxiliary propositions.

\subsection{Notation.} Throughout the paper we use the following notation. The convergence in distribution of random elements taking  values in a functional space is denoted by $\Longrightarrow$. The convergence in the sense of finite-dimensional distributions of random processes is denoted by $\overset{{\rm f.d.d.}}{\longrightarrow}$. The convergence in distribution of random variables and vectors is denoted by $\overset{{\rm d}}{\to}$. Finally, the a.s.~convergence is denoted by $\overset{{\rm a.s.}}{\to}$. For the asymptotic equivalence we use the notation $\sim$. Thus, $f(x)\sim g(x)$, $x\to A$, means that $\lim_{x\to A}f(x)/g(x)=1$.
When the function $\ell$ in the notation $S(\gamma; \ell)$ is clear from context, we simply write $S(\gamma)$.

Throughout the paper  we assume that $b(x)>0$ for all $x>0$, but all the results of this paper hold for any eventually positive regularly varying function $b$.

\medskip

\textbf{Acknowledgements.} The authors are grateful to Yan Fyodorov for bringing paper \cite{Bleher+Ridzal:2002} to our attention, and for comments on the manuscript that helped to improve its presentation.

\section{Zeros inside and outside the unit disk}\label{sec:zeros_away_from_the_boundary}

For any open connected domain $\mathcal{D}\subseteq \mathbb{C}$, the polynomial $P_n$ can be regarded as a random element taking values in $\mathcal{A}(\mathcal{D})$, the vector space of analytic functions on $\mathcal{D}$. A natural topology on $\mathcal{A}(\mathcal{D})$ is the topology of locally uniform convergence. It is well known that this topology is metrizable via the metric
$$
\|f-g\|_{\mathcal{A}(\mathcal{D})}:=\sum_{j\geq 1}2^{-j}\min\big(\sup_{z\in K_j}|f(z)-g(z)|,1\big),
$$
where $(K_j)_{j\geq 1}$ is any increasing sequence of compact sets such that $\mathcal{D}=\cup_{j\geq 1} K_j$. Under this metric, $\mathcal{A}(\mathcal{D})$ is a complete separable metric space. We refer to \cite[\S2]{Shirai:2012} for the basic properties of the space $\mathcal{A}(\mathcal{D})$.

Observe that under the moment assumption
\begin{equation}\label{eq:xi_log_moment}
\mathbb{E}[\log^{+}|\xi|]<\infty,
\end{equation}
where $\log^{+}(x)=\max(\log x,0)=\log \max(x,1)$, the sequence $(P_n)$ converges almost surely (a.s.) on $\mathcal{A}(\mathbb{D}_1)$ to a random analytic function
\begin{equation}\label{eq:p_infty_def}
P_{\infty}(z)=\sum_{k\geq 0}b(k)\xi_k z^k, \qquad z\in \mathbb{D}_1=\{|z|<1\}.
\end{equation}
This is a consequence of the Cauchy root test in combination with the following two facts. Firstly, by the Borel-Cantelli lemma,~\eqref{eq:xi_log_moment} implies
\begin{equation}\label{eq:slln_xi}
\lim_{k\to\infty}\frac{\log^{+}|\xi_k|}{k}=0\quad\text{a.s.}
\end{equation}
and, secondly,
$$
\lim_{k\to\infty}\frac{\log b(k)}{k}=0,
$$
by well-known properties of regularly varying functions.

\begin{prop}[Convergence inside the unit disk: polynomials]
\label{prop:1}
Assume~\eqref{eq:xi_log_moment}. Then
$$
(P_n(z))_{z\in\mathbb{D}_1}~\toas~(P_{\infty}(z))_{z\in\mathbb{D}_1},\quad n\to\infty,
$$
on the space $\mathcal{A}(\mathbb{D}_1)$.
\end{prop}

For an open connected set (domain) $\mathcal{D}\subseteq\mathbb{C}$, let $M_p(\mathcal{D})$ be the space of locally finite integer-valued measures on $\mathcal{D}$ endowed with the topology of vague convergence. Further, let $\zeros_{\mathcal{D}}:\mathcal{A}(\mathcal{D})\backslash\{0\}\to M_p(\mathcal{D})$ be the mapping which assigns to $f\in\mathcal{A}(\mathcal{D})$, $f\not\equiv 0$, the point process of complex zeros of $f$ in $\mathcal{D}$ with multiplicities. By Hurwitz's theorem the mapping $\zeros_{\mathcal{D}}$ is continuous on $\mathcal{A}(\mathcal{D})\setminus\{0\}$, see~\cite[Lemma 2.2]{Shirai:2012}.

\begin{corollary}[Convergence inside the unit disk: zeros]
\label{cor:2.1}
Assume~\eqref{eq:xi_log_moment}. Then
$$
\zeros_{\mathbb{D}_1}(P_n(z))~\toas~\zeros_{\mathbb{D}_1}(P_{\infty}(z)),\quad n\to\infty,
$$
on the space $M_p(\mathbb{D}_1)$.
\end{corollary}

Outside the closed unit disk $\cl(\mathbb{D}_1)$, the series $P_{\infty}$ is always a.s.~divergent unless $\mathbb{P}\{\xi=0\}=1$. However, a simple inversion argument allows us to conclude the following
\begin{theorem}[Convergence outside the unit disk: polynomials]
\label{thm:1}
Assume~\eqref{eq:xi_log_moment}. Then the sequence of polynomials
$$
Q_n(z):=b_n^{-1}z^nP_n(z^{-1}),\quad n\in\mathbb{N},\quad z\in\mathbb{C},
$$
converges in distribution on $\mathcal{A}(\mathbb{D}_1)$ to a random analytic function
$$
Q_{\infty}(z):=\sum_{k\geq 0}\xi_k z^k,\quad z\in\mathbb{D}_1.
$$
\end{theorem}
\begin{proof}
Observe that, for every $n\in\mathbb{N}$, the polynomial $Q_n$ has the same distribution as
$$
\hat{Q}_n(z):=\sum_{k=0}^{n}\frac{b_{n-k}}{b_n}\xi_k z^k,\quad z\in\mathbb{C}.
$$
Fix $\varepsilon\in (0,1)$ and decompose $\hat{Q}_n$ as
$$
\hat{Q}_n(z)=\sum_{k=0}^{\lfloor (1-\varepsilon) n\rfloor}\left(\frac{b_{n-k}}{b_n}-\left(\frac{n-k}{n}\right)^{\alpha}\right)\xi_k z^k+\sum_{k=0}^{\lfloor (1-\varepsilon) n\rfloor}\left(\frac{n-k}{n}\right)^{\alpha}\xi_k z^k+\sum_{k=\lfloor (1-\varepsilon) n\rfloor+1}^{n}\frac{b_{n-k}}{b_n}\xi_k z^k.
$$
The first term converges a.s.~to $0$ on $\mathcal{A}(\mathbb{D}_1)$ by the uniform convergence theorem for regularly varying functions, see~\cite[Theorem 1.5.2]{BGT}. The second sum converges a.s.~to $Q_{\infty}$ on $\mathcal{A}(\mathbb{D}_1)$ by the dominated convergence theorem since $((n-k)/n)^{\alpha}\leq \max(1,\varepsilon^{\alpha})$. Let $K$ be a fixed compact subset of $\mathbb{D}_1$ and $\rho:=\sup_{z\in K}|z|\in [0,1)$. Then
$$
\sup_{z\in K}\left|\sum_{k=\lfloor (1-\varepsilon) n\rfloor+1}^{n}\frac{b_{n-k}}{b_n}\xi_k z^k\right|\leq \rho^{(1-\varepsilon)n}\max_{0\leq k\leq n}|\xi_k|\frac{1}{b_n}\sum_{k=0}^{\lceil \varepsilon n\rceil}b_{k}.
$$
By the regular variation, there exists $A>0$ such that $b_n^{-1}\sum_{k=0}^{\lceil \varepsilon n\rceil}b_{k}=O(n^{A})$ as $n\to\infty$. Hence for any $\delta\in (0,1-\rho)$, the right-hand side of the last formula is bounded by ${\rm const}\cdot ((\rho+\delta)^{(1-\varepsilon)n}\max_{0\leq k\leq n}|\xi_k|)$, which converges to $0$ almost surely, as $n\to\infty$, in view of~\eqref{eq:slln_xi}. Thus,
$$
(\hat{Q}_n(z))_{z\in\mathbb{D}_1}~\toas~(Q_{\infty}(z))_{z\in\mathbb{D}_1},\quad n\to\infty,
$$
whence
$$
(Q_n(z))_{z\in\mathbb{D}_1}~\Longrightarrow~(Q_{\infty}(z))_{z\in\mathbb{D}_1},\quad n\to\infty
$$
on the space $\mathcal{A}(\mathbb{D}_1)$. The proof is complete.
\end{proof}

Hurwitz's theorem immediately implies
\begin{corollary}[Convergence outside the unit disk: zeros]\label{cor:outside_unit_disk}
Assume~\eqref{eq:xi_log_moment}. Then
$$
\zeros_{\mathbb{C}\setminus \cl(\mathbb{D}_1)}(P_n(z))~\Longrightarrow~\zeros_{\mathbb{C}\setminus \cl(\mathbb{D}_1)}(Q_{\infty}(1/z)),\quad n\to\infty.
$$
\end{corollary}

\begin{ex}
Assume that $\xi$ has the standard complex Gaussian distribution. Then, $(Q_{\infty}(z))_{z\in\mathbb{D}_1}$ is a standard hyperbolic Gaussian analytic function, see~\cite[Chapter 5]{HKPV}. In particular, its point process of zeros is determinantal, see Theorem 5.1.1 in the same reference. Corollary~\ref{cor:outside_unit_disk} implies that the point process of zeros of $P_n$ outside $\cl(\mathbb{D}_1)$ converges, as $n\to\infty$, to the image of this determinantal process under the inversion $z\mapsto z^{-1}$.
\end{ex}

\section{Zeros near the the boundary of the unit disk}\label{sec:zeros_near_the_boundary}

We shall switch our attention now to the local behavior of $P_n$ and their zeros near the boundary $\partial \mathbb{D}_1$ of the unit disk, which is more intriguing. In this setting we shall work under the more restrictive moment assumptions on $\xi$. Namely, we shall assume that
\begin{equation}\label{eq:moment_assumptions_on_xi1}
\mathbb{E}[\xi]=0,\quad \sigma_1^2:=\mathbb{E}[(\Re \xi)^2]<\infty,\quad \sigma_2^2:=\mathbb{E}[(\Im \xi)^2]<\infty\quad\text{satisfy}\quad \sigma^2:=\sigma_1^2+\sigma_2^2>0.
\end{equation}
Without loss of generality we can also assume that
\begin{equation}\label{eq:moment_assumptions_on_xi2}
\Cov[\Re\xi,\Im\xi]=0
\end{equation}
because the general case can be reduced to~\eqref{eq:moment_assumptions_on_xi2} by replacing $(\xi_k)_{k\geq 0}$ with $(z^{\ast}\xi_k)_{k\geq 0}$ where the constant $z^{\ast}\in\mathbb{C}$ is chosen such that $\Cov[\Re(z^{\ast}\xi),\Im(z^{\ast}\xi)]=0$. This amounts to replacing $P_n(z)$ with $z^{\ast}P_n(z)$.

Under the above assumptions the series $P_{\infty}(z)$ can be convergent or divergent at $z\in\partial\mathbb{D}_1$ depending on $b$. For example, if $\alpha<-\nicefrac{1}{2}$ then the random series
$$
P_{\infty}(\ee^{\ii \phi})=\sum_{k\geq 0}b(k)\xi_k\ee^{\ii \phi k}
$$
converges absolutely by the Kolmogorov two series theorem, see~\cite[Chapter IV, \S 2]{Shiryaev}, because
$\sum_{k\geq 0}b^2(k)<\infty$ under the assumption $\alpha<-\nicefrac{1}{2}$. In this case, $\phi\mapsto P_{\infty}(\ee^{\ii \phi})$ is an a.s.~continuous $2\pi$-periodic function, see~\cite[Theorem 2]{Hunt:1951}. On the other hand, by the same theorem, if $\alpha>-\nicefrac{1}{2}$ and (for example) the support of $\xi$ is bounded, then the series $P_{\infty}(z)$ is almost surely divergent, for every fixed $z\in\partial\mathbb{D}_1$. In the boundary case $\alpha=-\nicefrac{1}{2}$ convergence/divergence is determined by the slowly varying factor $\ell$. Moreover, even when convergence holds, the function  $\phi\mapsto P_{\infty}(\ee^{\ii \phi})$ is not necessarily continuous and, in general, is defined only almost everywhere on $[0,2\pi)$ by the Paley-Zygmund theorem, see~\cite[Eq.~(2)]{Hunt:1951}.

\subsection{Case \texorpdfstring{$\alpha>-\nicefrac{1}{2}$}{alpha>-1/2}: the liquid phase}\label{sec:zeros_alpha>-1/2} Fix $0\leq \psi<2\pi$ and consider the sequence of transformed polynomials
$$
P_n\left(\ee^{\ii \psi+u/n}\right),\quad u\in\mathbb{C},\quad n\in\mathbb{N}.
$$
\begin{theorem}\label{thm:boundary:alpha>-1/2}
Assume~\eqref{eq:moment_assumptions_on_xi1},~\eqref{eq:moment_assumptions_on_xi2} and $\alpha>-\nicefrac{1}{2}$. Then, for every fixed $0\leq \psi<2\pi$,
$$
\left(\frac{P_n\left(\ee^{\ii \psi+u/n}\right)}{b(n)\sqrt{n}}\right)_{u\in\mathbb{C}}~\Longrightarrow~(G_{\psi}(u))_{u\in\mathbb{C}},\quad n\to\infty,
$$
on the space $\mathcal{A}(\mathbb{C})$, where $G_{\psi}=(G_{\psi}(u))_{u\in\mathbb{C}}$ is a centered Gaussian analytic function\footnote{A Gaussian analytic function is a random element taking values in $\mathcal A(\mathcal D)$ whose finite dimensional distributions are complex Gaussian.} with the following covariance structure: for $u_1,u_2\in\mathbb{C}$,
\begin{align} \label{eq:covariance:Ga}
\mathbb{E}[G_{\psi}(u_1)\overline{G_{\psi}(u_2)}]&=\sigma^2\int_{0}^{1}x^{2\alpha}\ee^{x(u_1+\overline{u_2})}{\rm d}x,\\
\label{eq:covariance:Gb}
\mathbb{E}[G_{\psi}(u_1)G_{\psi}(u_2)]&=
\begin{cases}
(\sigma_1^2-\sigma_2^2)\int_{0}^{1}x^{2\alpha}\ee^{x(u_1+u_2)}{\rm d}x,&\psi\in\{0,\pi\},\\
0,&\psi\in (0,\pi)\cup(\pi,2\pi).
\end{cases}
\end{align}
\end{theorem}

\medskip

\begin{rem}\label{rem:integral_representation}
The process $u\mapsto G_{\psi}(u)$ possesses a neat integral representation. Indeed, let $B_1$ and $B_2$ be two independent standard real Brownian motions and $(B(t))_{t\in [0,1]}$ a $\mathbb{C}$-valued Brownian motion  given by
$$
B(t):=
\begin{cases}
\sigma_1 B_1(t)+\ii \sigma_2 B_2(t),&\psi\in\{0,\pi\},\\
\nicefrac{1}{\sqrt{2}}(\sigma B_1(t)+\ii \sigma B_2(t)),&\psi\in(0,\pi)\cup(\pi,2\pi).\\
\end{cases}
$$
By comparing the covariances we conclude that the following equality in distribution holds true:
\begin{equation}\label{eq:integral_representation_BM}
(G_{\psi}(u))_{u\in\mathbb{C}}\overset{{\rm d}}{=}\left(\int_0^1 t^{\alpha}\ee^{u t}{\rm d}B(t)\right)_{u\in\mathbb{C}},
\end{equation}
where the integral is understood in the Skorokhod sense or, alternatively, as a pathwise integral obtained via integration by parts
$$
\int_0^1 t^{\alpha}\ee^{u t}{\rm d}B(t)=B(t)t^{\alpha}\ee^{u t}\Big|_0^1-\int_0^1 B(t){\rm d}(t^{\alpha}\ee^{u t})=B(1)\ee^{u}-\int_0^1 B(t)(\alpha t^{\alpha-1}\ee^{ut }+ ut^{\alpha}\ee^{ut}){\rm d}t.
$$
Here we have used that $t^{\beta}B(t)\toas 0$, as $t\to 0+$, for any $\beta>-\nicefrac{1}{2}$. Observe that the latter limit relation also secures the Riemann integrability of $t\mapsto B(t)t^{\alpha-1}$ on $[0,1]$ when $\alpha>-\nicefrac{1}{2}$. Interestingly, Gaussian processes with an integral representation very similar to that in~\eqref{eq:integral_representation_BM} appear in the study of zeros of random Dirichlet series; see~\cite[Section 2.3]{Buraczewski+Dong+Iksanov+Marynych:2023}. The only difference is that, in the case of random Dirichlet series, the domain of integration $[0,1]$ is replaced by $[0, \infty)$.
\end{rem}

\begin{rem}
Observe that the normalization $b(n)\sqrt{n}$ in Theorem~\ref{thm:boundary:alpha>-1/2} satisfies
$$
b(n)\sqrt{n}~\sim~\left((1+2\alpha)\sum_{k=0}^{n}b^2(k)\right)^{\nicefrac{1}{2}}~\sim~\sqrt{(1+2\alpha)}(L(n))^{\nicefrac{1}{2}},
$$
where $L$ is defined in formula~\eqref{eq:L_definition} below. This follows from the direct half of Karamata’s theorem, see~\cite[Proposition 1.5.8]{BGT}. Consequently, the term $c_n = n^{\alpha+\nicefrac{1}{2}}\ell(n) = b(n)\sqrt{n}$ in the last row of Table~\ref{tab:1} may be replaced by $\bigl((1+2\alpha)L(n)\bigr)^{\nicefrac{1}{2}}$. 
\end{rem}

\begin{proof}
Since $P_n\left(\ee^{\ii \psi+u/n}\right)$ is the sum of independent centered random variables with the finite second moments, the proof amounts to checking: (a) convergence of the covariances; (b) the Lindeberg-Feller condition; (c) tightness in $\mathcal{A}(\mathbb{C})$. Throughout the proof let $K$ be a fixed compact set in $\mathbb{C}$.

\noindent
{\sc Convergence of covariances.} Observe that
$$
\mathbb{E}\left[\frac{P_n\left(\ee^{\ii \psi+u_1/n}\right)}{b(n)\sqrt{n}}\times \frac{P_n\left(\ee^{\ii \psi+u_2/n}\right)}{b(n)\sqrt{n}}\right]=\frac{\mathbb{E}[\xi^2]}{nb^2(n)}\sum_{k=0}^{n}b^2(k)\ee^{k(u_1+u_2)/n}\ee^{2\ii \psi k}
$$
and
\begin{equation}\label{eq:covariance}
\mathbb{E}\left[\frac{P_n\left(\ee^{\ii \psi+u_1/n}\right)}{b(n)\sqrt{n}}\times \frac{\overline{P_n\left(\ee^{\ii \psi+u_2/n}\right)}}{b(n)\sqrt{n}}\right]=\frac{\mathbb{E}[|\xi|^2]}{nb^2(n)}\sum_{k=0}^{n}b^2(k)\ee^{k(u_1+\overline{u_2})/n}.
\end{equation}
For the convergence of covariances, it suffices to check that, for every $w\in\mathbb{C}$,
\begin{equation}\label{eq:boundary:alpha>-1/2_main_covariance_convergence}
\lim_{n\to\infty}\frac{1}{nb^2(n)}\sum_{k=0}^{n}b^2(k)\ee^{k w/n}\ee^{2\ii \psi k}=
\begin{cases}
\int_{0}^{1}x^{2\alpha}\ee^{x w }{\rm d}x,&\psi\in\{0,\pi\},\\
0,&\psi\in (0,\pi)\cup(\pi,2\pi).
\end{cases}
\end{equation}
Fix $\varepsilon\in(0,1)$ and write
\begin{multline*}
\frac{1}{nb^2(n)}\sum_{k=0}^{n}b^2(k)\ee^{k w/n}\ee^{2\ii \psi k}=\frac{1}{n}\sum_{k=\lfloor n\varepsilon\rfloor}^{n}\frac{b^2(k)}{b^2(n)}\ee^{kw/n}\ee^{2\ii \psi k}
+\frac{1}{n}\sum_{k=0}^{\lfloor n\varepsilon\rfloor-1}\frac{b^2(k)}{b^2(n)}\ee^{kw/n}\ee^{2\ii \psi k}\\
=:Z_1(\varepsilon,n)+Z_2(\varepsilon,n).
\end{multline*}
First, we verify that
\begin{equation}\label{eq:covariances_proof1}
\lim_{\varepsilon\to 0+}\limsup_{n\to\infty}|Z_2(\varepsilon,n)|=0.
\end{equation}
Indeed,
\begin{align*}
|Z_2(\varepsilon,n)|\leq \ee^{\Re (w)}\frac{1}{nb^2(n)}\sum_{k=0}^{\lfloor n\varepsilon\rfloor}b^2(k).
\end{align*}
Using that $\alpha>-\nicefrac{1}{2}$, we infer $\sum_{k=0}^{n}b^2(k)\sim (1+2\alpha)^{-1}nb^2(n)$, as $n\to\infty$, see~\cite[Proposition 1.5.8]{BGT}. Hence
$$
\limsup_{n\to\infty}|Z_2(\varepsilon,n)|\leq \frac{\ee^{\Re (w)}}{1+2\alpha}\varepsilon^{1+2\alpha}
$$
and~\eqref{eq:covariances_proof1} follows.

The term $Z_1(\varepsilon,n)$ can be further decomposed as follows
\begin{multline*}
\frac{1}{n}\sum_{k=\lfloor n\varepsilon\rfloor}^{n}\frac{b^2(k)}{b^2(n)}\ee^{kw/n}\ee^{2\ii \psi k}=
\frac{1}{n}\sum_{k=\lfloor n\varepsilon\rfloor}^{n}\left(\frac{b^2(k)}{b^2(n)}-\left(\frac{k}{n}\right)^{2\alpha}\right)\ee^{kw/n}\ee^{2\ii \psi k}+\frac{1}{n}\sum_{k=\lfloor n\varepsilon\rfloor}^{n}\left(\frac{k}{n}\right)^{2\alpha}\ee^{kw/n}\ee^{2\ii \psi k}\\
=:Z_{11}(\varepsilon,n)+Z_{12}(\varepsilon,n).
\end{multline*}
By the uniform convergence theorem for regularly varying functions, see~\cite[Theorem 1.5.2]{BGT},
$$
\lim_{n\to\infty}\sup_{\lfloor n\varepsilon\rfloor\leq k\leq n}\left|\frac{b^2(k)}{b^2(n)}-\frac{k^{2\alpha}}{n^{2\alpha}}\right|=0.
$$
Thus, for every $\delta>0$, there exists $n_0=n_0(\delta)\in\mathbb{N}$ such that for all $n\geq n_0$,
$$
|Z_{11}(\varepsilon,n)|\leq \frac{\delta}{n}\sum_{k=\lfloor n\varepsilon\rfloor}^{n}\ee^{k\Re(w)/n}~\to~\delta\int_{\varepsilon}^{1}\ee^{\Re(w)x}{\rm d}x,\quad n\to\infty.
$$
Since $\delta>0$ is arbitrary
$$
\limsup_{n\to\infty}|Z_{11}(\varepsilon,n)|=0.
$$
Summarizing, we have checked that~\eqref{eq:boundary:alpha>-1/2_main_covariance_convergence} is a consequence of
$$
\lim_{\varepsilon\to 0+}\lim_{n\to\infty}\frac{1}{n}\sum_{k=\lfloor n\varepsilon\rfloor}^{n}\left(\frac{k}{n}\right)^{2\alpha}\ee^{kw/n}\ee^{2\ii \psi k}=\begin{cases}
\int_{0}^{1}x^{2\alpha}\ee^{x w }{\rm d}x,&\psi\in \{0,\pi\},\\
0,&\psi\in(0,\pi)\cup(\pi,2\pi).
\end{cases}
$$
The latter formula is secured by Proposition~\ref{prop:riemann_lebesgue} in the Appendix.

\noindent
{\sc The Lindeberg-Feller condition.} Fix $\delta>0$ and put $V_{n,k}(u,\psi):=\frac{b(k)\xi_k\ee^{i\psi k +u k/n}}{\sqrt{n}b_n}$, $k=0,\ldots,n$. We need to verify that, for every fixed $u\in K$,
\begin{equation}\label{eq:covariances_proof2}
\lim_{n\to\infty}\sum_{k=0}^{n}\mathbb{E}[|V_{n,k}|^2\1_{\{|V_{n,k}|\geq \delta\}}]=0.
\end{equation}
Using that the modulus of the exponent in $V_{n,k}$ is upper bounded by $\ee^{\Re(u)}$, we can estimate
\begin{align*}
\sum_{k=0}^{n}\mathbb{E}[|V_{n,k}|^2\1_{\{|V_{n,k}|\geq \delta\}}]\leq \frac{C}{n}\sum_{k=0}^{n}\frac{b^2(k)}{b^2(n)}\mathbb{E}[|\xi|^2\1_{\{|\xi|\geq \delta_1\sqrt{n}b(n)/b(k)\}}],
\end{align*}
for some $\delta_1\in (0,\delta)$. Fix again some $\varepsilon\in (0,1)$ and note that the uniform convergence theorem for regularly varying functions yields that for arbitrarily large $M>0$,
\begin{multline*}
\lim_{n\to\infty}\frac{1}{n}\sum_{k=\lfloor \varepsilon n\rfloor}^{n}\frac{b^2(k)}{b^2(n)}\mathbb{E}[|\xi|^2\1_{\{|\xi|\geq \delta_1\sqrt{n}b(n)/b(k)\}}]\leq
\lim_{n\to\infty}\frac{1}{n}\sum_{k=\lfloor \varepsilon n\rfloor}^{n}\frac{b^2(k)}{b^2(n)}\mathbb{E}[|\xi|^2\1_{\{|\xi|\geq \delta_1 M b(n)/b(k)\}}]\\
\leq \int_{\varepsilon}^{1}x^{2\alpha}\mathbb{E}[|\xi|^2\1_{\{|\xi|\geq \delta_1 M \varepsilon^{-2\alpha}\}}]{\rm d}x.
\end{multline*}
By the dominated convergence theorem, the right-hand side tends to zero as $M\to+\infty$ because $\mathbb{E}[|\xi|^2]<\infty$. On the other hand,
$$
\frac{1}{n}\sum_{k=0}^{\lfloor \varepsilon n\rfloor}\frac{b^2(k)}{b^2(n)}\mathbb{E}[|\xi|^2\1_{\{|\xi|\geq \delta_1\sqrt{n}b(n)/b(k)\}}]\leq \frac{\sigma^2}{n}\sum_{k=0}^{\lfloor \varepsilon n\rfloor}\frac{b^2(k)}{b^2(n)}.
$$
As explained above in the analysis of $Z_2(\varepsilon,n)$, the right-hand side converges to zero after sending first $n\to\infty$ and then $\varepsilon\to 0+$. Thus,~\eqref{eq:covariances_proof2} is proved.

\noindent
{\sc Tightness}. Using~\cite[Lemma 4.2]{Kab+Klim:2014} we infer that it suffices to show that for every fixed $R>0$,
$$
\sup_{n\in\mathbb{N}}\sup_{|u|\leq R}\frac{\mathbb{E}|P_n(\ee^{\ii \psi + u/n})|^2}{nb^2 (n)}<\infty.
$$
Using~\eqref{eq:covariance} with $u_1=u_2$ we conclude
$$
\sup_{|u|\leq R}\frac{\mathbb{E}|P_n(\ee^{\ii \psi + u/n})|^2}{nb^2 (n)}\leq \ee^{2R}\frac{\sigma^2}{nb^2(n)}\sum_{k=0}^{n}b^2(k)~\to~\frac{\ee^{2R}\sigma^2}{1+2\alpha},\quad n\to\infty.
$$
The proof is complete.
\end{proof}

\begin{corollary}[Local convergence of zeros near the unit circle, the liquid phase $\alpha>-\nicefrac{1}{2}$]
Assume~\eqref{eq:moment_assumptions_on_xi1},~\eqref{eq:moment_assumptions_on_xi2} and $\alpha>-\nicefrac{1}{2}$. Then, for every fixed $0\leq \psi<2\pi$,
$$
\zeros_{\mathbb{C}}(P_n(\ee^{\ii \psi+(\cdot)/n})~\Longrightarrow~(\zeros_{\mathbb{C}}(G_{\psi}(\cdot)),\quad n\to\infty,
$$
on the space $M_p(\mathbb{C})$, where $G_{\psi}$ is a random Gaussian analytic function defined in Theorem~\ref{thm:boundary:alpha>-1/2}.
\end{corollary}

\begin{ex}\label{ex:edelman_kostlan}
It is apparent from Eqs.~\eqref{eq:covariance:Ga} -- \eqref{eq:covariance:Gb} that the marginal distributions of the Gaussian analytic function $G_{\psi}$ are isotropic for every $\psi\in [0,2\pi)$  if $\sigma_1=\sigma_2$ and for $\psi\in (0,\pi)\cup (\pi,2\pi)$ otherwise. Assuming $G_{\psi}$ has isotropic distribution, the Edelman-Kostlan formula, see \cite[Eq.~(2.4.8)]{HKPV}, yields the first intensity $\rho_1(\alpha, \cdot)$ of zeros of $G_\psi$ in the form
\begin{equation}\label{eq:rho}
\rho_1(\alpha; u)=\frac{1}{\pi}\frac{\partial}{\partial u}\frac{\partial}{\partial \overline{u}}\log \Cov[G_{\psi}(u),G_{\psi}(u)]=\frac{1}{\pi}\left(\frac{\Phi_{2\alpha+2}(u+\overline{u})}{\Phi_{2\alpha}(u+\overline{u})}-\frac{\Phi^2_{2\alpha+1}(u+\overline{u})}{\Phi^2_{2\alpha}(u+\overline{u})}\right),
\end{equation}
where
\begin{equation}\label{eq:Phi_def}
\Phi_{\beta}(u):=\int_0^1 x^{\beta}\ee^{u x}{\rm d}x,\quad \beta>-1,\quad u\in\mathbb{C}.
\end{equation}
The right-hand side of equation \eqref{eq:rho} is the intensity of the limiting local point process of zeros of the random polynomials $P_n$  in the scaling window $z=\ee^{u/n+\ii \psi}$ attached to point $\ee^{\ii \psi}$. One observes that the intensity is constant in the direction tangential to the unit circle (this direction corresponds to purely imaginary values of $u$). Thus,
equation~\eqref{eq:rho} combined with Theorem~\ref{thm:boundary:alpha>-1/2} suggests that the expected number of zeros of $P_n$ in the annulus
$$\mathbb{A}(\ee^{s_1/n},\ee^{s_2/n})=\{z\in \mathbb{C}: \ee^{s_1/n} \leq |z| \leq \ee^{s_2/n}\}, \quad s_1 < s_2,
$$
is asymptotically equivalent to
$2\pi n \int_{s_1}^{s_2} \rho_1 (\alpha; r) {\rm d} r $, that is,
\begin{equation}\label{eq:intenstiy_rho_1_int}
\lim_{n\to\infty} \frac{1}{n} \mathbb{E} \big[\zeros_{\mathbb{A}(\ee^{s_1/n},\ee^{s_2/n})} (P_n)\big] = \frac{\Phi_{2\alpha+1}(2s_2)}{\Phi_{2\alpha}(2s_2)}-\frac{\Phi_{2\alpha+1}(2s_1)}{\Phi_{2\alpha}(2s_1)}.
\end{equation}
If $\xi$ has the standard complex normal distribution, this can be verified using the Edelman-Kostlan formula.  Indeed, let $p_n (\alpha; r)$ stand for the radial intensity \footnote{This means that  the expected number of zeros in a domain  $\{z\in \mathbb{C}: |z| \in (r_1,r_2),\; \arg z \in (\psi_1, \psi_2)  \}$ is given by the integral $(\psi_2 - \psi_1) \int_{r_1}^{r_2} p_n (\alpha; r)\d r $.} of $\zeros_{\mathbb{C}} (P_n)$.  Then, by the Edelman-Kostlan formula,
\begin{align}\label{eq:RadialIntensity}
p_n (\alpha; r) =  \frac{1}{\pi r} \frac{S_n(2+2\alpha; \ell; r^2) S_n(2\alpha; \ell;  r^2) - S_n^2(1+2\alpha; \ell;  r^2)}{S_n^2(2\alpha; \ell;  r^2)},\quad r>0,
\end{align}
where $S_n(\gamma;\ell;q) := \sum_{k=1}^n k^{\gamma} \ell^2(k) q^{k}
$, and
in turn, by \eqref{eq:boundary:alpha>-1/2_main_covariance_convergence},
\begin{align}\label{eq:RadialIntensity1}
\lim_{n\to\infty} \frac{1}{n^2} p_n (\alpha; \ee^{s/n}) =  \frac{1}{\pi} \left(
\frac{\Phi_{2\alpha+2}(2s)}{\Phi_{2\alpha}(2s)}-\frac{\Phi^2_{2\alpha+1}(2s)}{\Phi^2_{2\alpha}(2s)}
\right),\quad s\in\mathbb{R}.
\end{align}
If $\alpha=0$, then the right-hand side simplifies to
\begin{equation}\label{eq:ibragimov+zeitouni:1997}
\lim_{n\to\infty} \frac{1}{n^2} p_n (0; \ee^{s/n})=
\frac{1}{4\pi s^2}\left(1-\left(\frac{s}{\sinh s} \right)^2
\right),\quad s\in\mathbb{R}.
\end{equation}
Formula~\eqref{eq:ibragimov+zeitouni:1997} was derived in~\cite{Ibragimov+Zeitouni:1997} for the Kac polynomials ($\alpha = 0$) with the random coefficients satisfying some regularity conditions. 
It is natural to expect that the generalisation of \eqref{eq:ibragimov+zeitouni:1997} to $b(k)=k^{\alpha}\ell (k)$, $\alpha >-\nicefrac{1}{2}$, see Eq.~\eqref{eq:intenstiy_rho_1_int}, which we established for the isotropic Gaussian $\xi_k$ also holds for all probability distributions with finite second moment, although being mostly focused on the \emph{local} behavior of random zeros we do not pursue this avenue in this paper.
\end{ex}

Using the same method we can prove a slightly more general version of Theorem~\ref{thm:boundary:alpha>-1/2} with the joint convergence for different phases $\psi_1,\ldots,\psi_m$.

\begin{theorem}\label{thm:boundary:alpha>-1/2_joint}
Assume~\eqref{eq:moment_assumptions_on_xi1},~\eqref{eq:moment_assumptions_on_xi2} and $\alpha>-\nicefrac{1}{2}$. Then, for every fixed $m\in\mathbb{N}$ and $0\leq \psi_1<\psi_2<\cdots<\psi_m<2\pi$,
$$
\left(\frac{P_n\left(\ee^{\ii \psi_1+u/n}\right)}{b(n)\sqrt{n}},\ldots,\frac{P_n\left(\ee^{\ii \psi_m+u/n}\right)}{b(n)\sqrt{n}}\right)_{u\in\mathbb{C}}~\Longrightarrow~(G_{\psi_1}(u),\ldots,G_{\psi_m}(u))_{u\in\mathbb{C}},\quad n\to\infty,
$$
on the space $\mathcal{A}^m(\mathbb{C})$, where the cross-covariances are given by
\begin{align*}
\mathbb{E}[G_{\psi_i}(u_1)\overline{G_{\psi_{j}}(u_2)}]&=0,\quad u_1,u_2\in\mathbb{C},\\
\mathbb{E}[G_{\psi_i}(u_1)G_{\psi_j}(u_2)]&=
\begin{cases}
(\sigma_1^2-\sigma_2^2)\int_{0}^{1}x^{2\alpha}\ee^{x(u_1+u_2)}{\rm d}x,&\text{if }\psi_i+\psi_j=2\pi,\\
0,&\text{if }\psi_i+\psi_j\neq 2\pi,
\end{cases}
\end{align*}
for $i\neq j$.
\end{theorem}
The proof of Theorem~\ref{thm:boundary:alpha>-1/2_joint} relies on the formulae
$$
\frac{1}{n b^2(n)}\mathbb{E}[P_n(\ee^{\ii \psi_i+u_1/n})\times \overline{P_n(\ee^{\ii \psi_i+u_2/n})}]=\frac{\sigma^2}{n}\sum_{k=0}^{n}\frac{b^2(k)}{b^2(n)}\ee^{k(u_1+\overline{u_2})/n}\ee^{\ii k (\psi_i-\psi_j)}
$$
and
$$
\frac{1}{n b^2(n)}\mathbb{E}[P_n(\ee^{\ii \psi_i+u_1/n})\times P_n(\ee^{\ii \psi_i+u_2/n})]=\frac{(\sigma_1^2-\sigma_2^2)}{n}\sum_{k=0}^{n}\frac{b^2(k)}{b^2(n)}\ee^{k(u_1+u_2)/n}\ee^{\ii k (\psi_i+\psi_j)}
$$
in conjunction with relation~\eqref{eq:boundary:alpha>-1/2_main_covariance_convergence}. We omit the details here and instead provide a complete explanation in the next subsection, which addresses a similar situation for the case $\alpha\leq -\nicefrac{1}{2}$.

\subsection{Case \texorpdfstring{$\alpha\leq-\nicefrac{1}{2}$}{alpha <= - 1/2} and \texorpdfstring{$S(2\alpha;\ell)<\infty$}{S(2alpha;l)<infinity}: the strong crystalline phase}\label{sec:zeros_alpha<=-1/2}
Throughout this section we assume that
\begin{equation}\label{eq:convergence_b^2}
S(2\alpha)<+\infty~\Longleftrightarrow~\sum_{k\geq 0}b^2{(k)}<\infty,
\end{equation}
where we recall our convention that $S(2\alpha)=S(2\alpha;\ell)$. This is always true if $\alpha<-\nicefrac{1}{2}$ and can be true or false if $\alpha=-\nicefrac{1}{2}$ depending on the slowly varying factor $\ell$. The behavior of $P_n$ (and their zeros) near the boundary of $\mathbb{D}_1$ is rather different whenever~\eqref{eq:convergence_b^2} holds true because $P_n$ converges to $P_{\infty}$ on the boundary, as was explained above.

It turns out that the zeros of $P_n$ are concentrated near a centered circle of radius $r_n$ which we are now going to define.
Fix $\psi\in [0,2\pi)$. For  $u\in\mathbb{C}$ and all sufficiently large $n\in\mathbb{N}$ put
\begin{align}\label{eq:z_n_alpha<=-1/2}
z_n(\psi,u)=\exp\left\{\frac{a_n}{2n}+\frac{u}{n}+{\rm i}\psi\right\}=z_n(0,0)\ee^{u/n}\ee^{\ii \psi},
\end{align}
where $a_n=-W_{-1}(-nb^2(n))$ and $W_{-1}$ is the secondary branch of the Lambert $W$-function. Then, $r_n = z_n(0,0)$ and the map $u\mapsto z_n(\psi, u)$ defines ``local coordinates'' in a scaling window centered at $\ee^{\ii \psi} r_n$.

According to Lemma~\ref{lem:W_-1} in the Appendix $a_n$ is well-defined and positive for all sufficiently large $n\in\mathbb{N}$ and satisfies
\begin{equation}\label{eq:a_n_choice}
a_n\exp(-a_n)=n b^2(n).
\end{equation}
Moreover, $z_n(0,0)\geq 1$ and
\begin{equation}\label{eq:z_n_main_relation}
z_n^{2n}(0,0)=\exp(a_n)=\frac{a_n}{nb^2(n)}.
\end{equation}
Lemma~\ref{lem:W_-1} also contains some further properties of the sequence $(a_n)$ that will be used later on. In particular, $a_n\to+\infty$, which implies $n(r_n-1) \to + \infty$.

\begin{theorem}\label{thm:boundary:alpha<-1/2}
Assume~\eqref{eq:moment_assumptions_on_xi1},~\eqref{eq:moment_assumptions_on_xi2}, $\alpha\leq -\nicefrac{1}{2}$ and, if $\alpha=-\nicefrac{1}{2}$ also~\eqref{eq:convergence_b^2}. Then, for any $0\leq \psi_1<\cdots<\psi_m<2\pi$, the sequence $(P_n(z_n(\psi_1,u)),\ldots,P_n(z_n(\psi_m,u)))_{u\in\mathbb{C}}$, $n\in\mathbb{N}$, converges in distribution on $\mathcal{A}^m(\mathbb{C})$ to
$$
(P_{\infty}(\ee^{\ii\psi_1})+\ee^{u}N_1,\ldots,P_{\infty}(\ee^{\ii\psi_m})+\ee^{u}N_m)_{u\in\mathbb{C}},
$$
where $(N_1,\ldots,N_m)$ is a  complex Gaussian random vector with zero mean and
\begin{equation}\label{eq:limit_gaussian_cov_1}
\mathbb{E}[N_i\overline{N_j}]=\sigma^2\1_{\{i=j\}},\quad 1\leq i,j\leq m,
\end{equation}
and
\begin{equation}\label{eq:limit_gaussian_cov_2}
\mathbb{E}[N_iN_j]=(\sigma_1^2-\sigma_2^2)\1_{\{\phi_i+\phi_j=2\pi\}},\quad 1\leq i,j\leq m.
\end{equation}
Moreover, the Gaussian vector $(N_1,\ldots,N_m)$ is independent of $P_{\infty}$.

If $\alpha<-\nicefrac{1}{2}$, then the sequence $a_n=-W_{-1}(-nb^2(n))$ in the definition of $z_n(\psi,u)$ can be replaced by
\begin{align}\label{eq:z_n_alpha<-1/2}
-2\log b(n)-\log n +\log \log n+\log (-2\alpha-1).
\end{align}
\end{theorem}

Theorem~\ref{thm:boundary:alpha<-1/2} implies the following
\begin{corollary}[Local convergence of zeros near the unit circle, strong crystalline phase $\alpha\leq -\nicefrac{1}{2}$ and $\sum b^2(k)<\infty$]\label{cor:boundary:alpha<-1/2}
Under the assumptions of Theorem~\ref{thm:boundary:alpha<-1/2}, the sequence
$$
(\zeros_{\mathbb{C}}(P_n(z_n(\psi_1,\cdot))),\ldots,\zeros_{\mathbb{C}}(P_n(z_n(\psi_m,\cdot))),\quad n\in\mathbb{N},
$$
regarded as a sequence of random elements of $(M_p(\mathbb{C}))^m$, converges in distribution to
\begin{multline*}
\left(\zeros_{\mathbb{C}}(P_{\infty}(\ee^{\ii \psi_1})+\ee^{(\cdot)}N_1),\ldots,\zeros_{\mathbb{C}}(P_{\infty}(\ee^{\ii \psi_m})+\ee^{(\cdot)}N_m)\right)=\\
\left(\log(-N_1^{-1}P_{\infty}(\ee^{\ii \psi_1}))+2\pi {\rm i}\mathbb{Z},\ldots,\log(-N_m^{-1}P_{\infty}(\ee^{\ii \psi_m}))+2\pi {\rm i}\mathbb{Z}\right),
\end{multline*}
where $\log$ denotes any branch (say, principal) of the complex logarithm.
\end{corollary}

\begin{proof}[Proof of Theorem~\ref{thm:boundary:alpha<-1/2}]
By the Cram\'{e}r-Wold device it suffices to prove that, for all $\alpha_1,\ldots,\alpha_m\in\mathbb{C}$,
\begin{equation}\label{eq:prop:boundary:alpha<-1/2_main_CW}
\left(\sum_{j=1}^{m}\alpha_j P_n(z_n(\psi_j,u))\right)_{u\in\mathbb{C}}~\Longrightarrow~\left(\sum_{j=1}^{m}\alpha_j (P_{\infty}(\ee^{\ii \psi_j})+\ee^{u}N_j)\right)_{u\in\mathbb{C}},\quad n\to\infty,
\end{equation}
on $\mathcal{A}(\mathbb{C})$. As in the proof of Theorem~\ref{thm:boundary:alpha>-1/2} we shall first prove the convergence of finite-dimensional distributions and then check the tightness.

Fix $M\in\mathbb{N}$ and decompose (for a sufficiently large $n$) the left-hand side of~\eqref{eq:prop:boundary:alpha<-1/2_main_CW} as follows
\begin{multline*}
\sum_{k=0}^{\lfloor \log n \rfloor}b(k)\xi_k \left(\sum_{j=1}^{m}\alpha_j z_n^k(\psi_j,u)\right)+\sum_{j=1}^{m}\alpha_j \sum_{k=\lfloor \log n \rfloor+1}^{n-M\lfloor n a_n^{-1}\rfloor}b(k)\xi_k z_n^k(\psi_j,u)\\
+\sum_{k=n-M\lfloor n a_n^{-1}\rfloor+1}^{n}b(k)\xi_k \sum_{j=1}^{m}\alpha_j z_n^k(\psi_j,u)=:Y^{(1)}_n(u)+\sum_{j=1}^{m}\alpha_j Y^{(2)}_{n,j}(u,M)+Y^{(3)}_n(u,M).
\end{multline*}

Part (ii) of Lemma~\ref{lem:W_-1} implies that for every compact set $K\subset\mathbb{C}$ and $j=1,\ldots,m$,
$$
\sup_{u\in K}\max_{0\leq k\leq \lfloor \log n\rfloor}|z_n^k(\psi_j,u)-\ee^{\ii \psi_j k}|=
\sup_{u\in K}\max_{0\leq k\leq \lfloor \log n\rfloor}|z_n^k(0,u)-1|=O\left(\frac{\log^2 n}{n}\right),\quad n\to\infty.
$$
Therefore,
$$
Y^{(1)}_n(u)=\sum_{j=1}^{m}\alpha_j \left(\sum_{k=0}^{\lfloor \log n \rfloor}b(k)\xi_k \ee^{\ii \psi_j k}\right)+R_n(u),
$$
where, for some $C>0$,
$$
\mathbb{E}[\sup_{u\in K} |R_n|]\leq C\frac{\log^2 n}{n}\sum_{k=0}^{\lfloor \log n \rfloor}b(k)~\to~0,\quad n\to\infty.
$$
Summarizing, we have shown that
\begin{equation}\label{eq:Y_n_1_func}
(Y_n^{(1)}(u))_{u\in\mathbb{C}}~\Longrightarrow~\left(\sum_{j=1}^{m}\alpha_j P_{\infty}(\ee^{\ii \psi_j})\right)_{u\in\mathbb{C}},\quad n\to\infty,
\end{equation}
on $\mathcal{A}(\mathbb{C})$.

Our next goal is to show that, for every $\varepsilon>0$ and $u\in\mathbb{C}$,
\begin{equation}\label{eq:Y2_to_zero}
\lim_{M\to\infty}\limsup_{n\to\infty}\mathbb{P}\{|Y_{n,j}^{(2)}(u,M)|>\varepsilon\}=0,\quad j=1,\ldots,m.
\end{equation}
By Markov's inequality it suffices to check that
$$
\lim_{M\to\infty}\limsup_{n\to\infty}\mathbb{E}[|Y_{n,j}(u,M)|^2]=0,\quad j=1,\ldots,m.
$$
Observe that
$$
\mathbb{E}[|Y_{n,j}^{(2)}(u,M)|^2]=\sigma^2\sum_{k=\lfloor \log n\rfloor+1}^{n-M\lfloor n a_n^{-1}\rfloor}b^2(k)z_n^{2k}(0,0)\ee^{2\Re(u)k/n}.
$$
Hence,~\eqref{eq:Y2_to_zero} follows if we can check that
\begin{equation}\label{eq:Y2_to_zero_1}
\lim_{M\to\infty}\limsup_{n\to\infty}\sum_{k=\lfloor \log n\rfloor+1}^{n-M\lfloor n a_n^{-1}\rfloor}b^2(k)z_n^{2k}(0,0)=0.
\end{equation}
Let $A_1>0$ be a fixed positive constant to be specified later. Note that for $k\leq A_1 n a_n^{-1}$, the quantity $z_n^{2k}(0,0)$ is bounded. Since also
$$
\sum_{k\geq 0}b^2(k)<\infty,
$$
the limit relation~\eqref{eq:Y2_to_zero_1} is equivalent to
\begin{equation}\label{eq:Y2_to_zero_2}
\lim_{M\to\infty}\limsup_{n\to\infty}\sum_{k=\lfloor A_1 n a_n^{-1}\rfloor+1}^{n-M\lfloor n a_n^{-1}\rfloor}b^2(k)z_n^{2k}(0,0)=0.
\end{equation}
Using~\eqref{eq:z_n_main_relation} we can write
$$
\sum_{k=\lfloor A_1 n a_n^{-1}\rfloor+1}^{n-M\lfloor n a_n^{-1}\rfloor}b^2(k)z_n^{2k}(0,0)=\frac{a_n}{n}
\sum_{k=\lfloor A_1 n a_n^{-1}\rfloor+1}^{n-M\lfloor n a_n^{-1}\rfloor}\left(\frac{b^2(k)}{b^2(n)}\right)z_n^{2k-2n}(0,0).
$$
By the Potter's bound for regularly varying functions, see~\cite[Theorem 1.5.6]{BGT}, for every $\delta>0$, $A_2>1$ and all sufficiently large $n$,
$$
\frac{b^2(k)}{b^2(n)}\leq A_2\left(\frac{k}{n}\right)^{2\alpha-\delta},\quad \lfloor A_1 n a_n^{-1}\rfloor+1\leq k\leq n-M\lfloor n a_n^{-1}\rfloor.
$$
Therefore, we are left with showing that
\begin{equation}\label{eq:Y2_to_zero_3}
\lim_{M\to\infty}\limsup_{n\to\infty}\frac{a_n}{n}
\sum_{k=\lfloor A_1 n a_n^{-1}\rfloor+1}^{n-M\lfloor n a_n^{-1}\rfloor}\left(\frac{k}{n}\right)^{2\alpha-\delta}z_n^{2k-2n}(0,0)=0.
\end{equation}
Changing the index of summation we infer
\begin{equation}
\frac{a_n}{n}\sum_{k=\lfloor A_1 n a_n^{-1}\rfloor+1}^{n-M\lfloor n a_n^{-1}\rfloor}\left(\frac{k}{n}\right)^{2\alpha-\delta}z_n^{2k-2n}(0,0)
=\frac{a_n}{n}\sum_{k=M\lfloor n a_n^{-1}\rfloor}^{n-\lfloor A_1 n a_n^{-1}\rfloor-1}\left(1-\frac{k}{n}\right)^{2\alpha-\delta}\exp\left\{-\frac{a_n k}{n}\right\}\label{eq:Y2_to_zero_4}.
\end{equation}
Pick now $A_1>2|2\alpha-\delta|$ and write
$$
\max_{M\lfloor n a_n^{-1}\rfloor\leq k\leq n-\lfloor A_1 n a_n^{-1}\rfloor-1}\left(1-\frac{k}{n}\right)^{2\alpha-\delta}\exp\left\{-\frac{ka_n}{2n}\right\}\leq \max_{0\leq x\leq 1-A_1 a_n^{-1}}(1-x)^{2\alpha-\delta}\ee^{-x a_n/2}.
$$
Using the fact that the derivative of $x\mapsto (1-x)^{2\alpha-\delta}\ee^{-x a_n/2}$ is strictly negative on $[0,1-A_1 a_n^{-1}]$, the above maximum is attained at $x=0$ and is equal to $1$. Thereupon,~\eqref{eq:Y2_to_zero_4} yields
\begin{align*}
\frac{a_n}{n}\sum_{k=\lfloor A_1 n a_n^{-1}\rfloor+1}^{n-M\lfloor n a_n^{-1}\rfloor}\left(\frac{k}{n}\right)^{2\alpha-\delta}z_n^{2k-2n}(0,0)\leq \frac{a_n}{n}
\sum_{k=M\lfloor n a_n^{-1}\rfloor}^{\infty}\exp\left\{-\frac{k a_n}{2n}\right\}\leq \frac{a_n}{n}\frac{\ee^{-M/4}}{1-\ee^{-a_n/(2n)}}.
\end{align*}
The latter implies that $\limsup_{n\to\infty}$ in~\eqref{eq:Y2_to_zero_3} is less or equal $2\ee^{-M/4}$. Sending $M\to\infty$ finishes the proof of~\eqref{eq:Y2_to_zero_3} and hence of~\eqref{eq:Y2_to_zero}.

It remains to treat $Y_n^{(3)}(u,M)$. To this end we need a lemma.
\begin{lemma}\label{lem:Y_n_3_variance}
Assume that $\alpha\leq -\nicefrac{1}{2}$ and, if $\alpha=-\nicefrac{1}{2}$, that~\eqref{eq:convergence_b^2} holds. Then, for every $M>0$ and $w\in\mathbb{C}$,
\begin{equation}\label{lem:claim1}
\lim_{n\to\infty}\sum_{k=n-M\lfloor n a_n^{-1}\rfloor}^{n}b^2(k)z_n^{2k}(0,0)\ee^{w k/n}=\ee^{w}(1-\ee^{-M}).
\end{equation}
Moreover, for every fixed $\theta\notin 2\pi\mathbb{Z}$,
\begin{equation}\label{lem:claim2}
\lim_{n\to\infty}\sum_{k=n-M\lfloor n a_n^{-1}\rfloor}^{n}b^2(k)z_n^{2k}(0,0)\ee^{w k/n}\ee^{\ii k \theta}=0.
\end{equation}
\end{lemma}
\begin{proof}
Observe that
\begin{equation}\label{eq:lem_proof_uniformity}
\ee^{w k/n}b^2(k)~\sim \ee^{w}b^{2}(n),\quad n\to\infty,
\end{equation}
uniformly in $k\in\{n-M\lfloor n a_n^{-1}\rfloor,\ldots,n\}$, by the uniform convergence theorem for regularly varying functions, see~\cite[Theorem 1.5.2]{BGT}. Thus, using also~\eqref{eq:z_n_main_relation}, we infer
\begin{align}
\sum_{k=n-M\lfloor n a_n^{-1}\rfloor}^{n}b^2(k)z_n^{2k}(0,0)\ee^{w k/n}&\sim \ee^{w} b^{2}(n)\sum_{k=n-M\lfloor n a_n^{-1}\rfloor}^{n}z_n^{2k}(0,0)\notag\\
&=\ee^{w} b^{2}(n)\sum_{k=0}^{M\lfloor n a_n^{-1}\rfloor}z_n^{2(n-k)}(0,0)\notag\\
&=\ee^{w}\frac{a_n}{n}\sum_{k=0}^{M\lfloor n a_n^{-1}\rfloor}z_n^{-2k}(0,0)\notag\\
&=\ee^{w}\frac{a_n}{n}\frac{1-z_n^{-2(M\lfloor n a_n^{-1}\rfloor+1)}(0,0)}{1-z_n^{-2}(0,0)}.\label{eq:lem_proof4}
\end{align}
Since
$$
1-z_n^{-2}(0,0)~\sim~\frac{a_n}{n},\quad n\to\infty,
$$
and also
\begin{equation}\label{eq:lem_proof5}
\lim_{n\to\infty}z_n^{-2(M\lfloor n a_n^{-1}\rfloor+1)}(0,0)=\ee^{-M},
\end{equation}
we conclude that~\eqref{lem:claim1} holds true. For the proof of~\eqref{lem:claim2} note that
\begin{align*}
&\hspace{-2cm}\sum_{k=n-M\lfloor n a_n^{-1}\rfloor}^{n}b^2(k)z_n^{2k}(0,0)\ee^{w k/n}\ee^{\ii k \theta}\\
&=\frac{a_n}{n} \sum_{k=n-M\lfloor n a_n^{-1}\rfloor}^{n}\left(\frac{b^2(k)}{b^2(n)}\ee^{w k/n}-\ee^{w}\right)z_n^{-2(n-k)}(0,0)\ee^{\ii k \theta}\\
&+\ee^{w}\frac{a_n}n\sum_{k=n-M\lfloor n a_n^{-1}\rfloor}^{n}z_n^{-2(n-k)}(0,0)\ee^{\ii k \theta}=:R_1(n)+R_2(n).
\end{align*}
From~\eqref{eq:lem_proof_uniformity} we infer that, for arbitrarily small $\varepsilon>0$ and all sufficiently large $n$,
$$
|R_1(n)|\leq \frac{\varepsilon a_n}{n} \sum_{k=n-M\lfloor n a_n^{-1}\rfloor}^{n}z_n^{-2(n-k)}(0,0)
=\frac{\varepsilon a_n}{n} \sum_{k=0}^{M\lfloor n a_n^{-1}\rfloor}z_n^{-2k}(0,0).
$$
Using~\eqref{eq:lem_proof4} and~\eqref{eq:lem_proof5}, we infer
$$
\limsup_{n\to\infty}|R_1(n)|\leq \varepsilon (1-\ee^{-M}).
$$
Since $\varepsilon>0$ is arbitrary, the above limit superior is equal to $0$. The sum in $R_2(n)$ is a geometric progression and can be evaluated directly:
$$
R_2(n)=\ee^{w}\ee^{\ii n \theta}\frac{a_n}{n}\sum_{k=0}^{M\lfloor n a_n^{-1}\rfloor}z_n^{-2k}(0,0)\ee^{-\ii k \theta}=\ee^{w}\ee^{\ii n \theta}\frac{a_n}n\frac{1-\ee^{-\ii\theta(M\lfloor n a_n^{-1}\rfloor+1)}z_n^{-2(M\lfloor n a_n^{-1}\rfloor+1)}(0,0)}{1-\ee^{-\ii \theta}z_n^{-2}(0,0)}.
$$
The last ratio is bounded because the denominator converges to $1-\ee^{-\ii \theta}\neq 0$ whereas the numerator is bounded by~\eqref{eq:lem_proof5}. Hence $|R_2(n)|=O(n^{-1}a_n)$, as $n\to\infty$, and~\eqref{lem:claim2} follows, see part (ii) of Lemma~\ref{lem:W_-1}.
\end{proof}

We now return to the proof of Theorem~\ref{thm:boundary:alpha<-1/2} and proceed with the analysis of $Y_n^{(3)}(u,M)$. Observe that $Y_n^{(3)}(u,M)$ is the sum of independent random variables with finite second moments. Further, for every $u_1,u_2\in\mathbb{C}$,
\begin{multline*}
\mathbb{E}[Y_n^{(3)}(u_1,M)Y_n^{(3)}(u_2,M)]\\
=(\sigma_1^2-\sigma_2^2)\sum_{j=1}^{m}\sum_{i=1}^{m}\alpha_j\alpha_i\left(\sum_{k=n-M\lfloor n/\log n\rfloor}^{n}b^2(k)z_n^{2k}(0,0)\ee^{k(u_1+u_2)/n}\ee^{\ii (\psi_j+\psi_i)k}\right)
\end{multline*}
and
\begin{multline*}
\mathbb{E}[Y_n^{(3)}(u_1,M)\overline{Y_n^{(3)}(u_2,M)}]=\sigma^2\sum_{j=1}^{m}\sum_{i=1}^{m}\alpha_j\overline{\alpha_i}\left(\sum_{k=n-M\lfloor n/\log n\rfloor}^{n}b^2(k)z_n^{2k}(0,0)\ee^{k(u_1+\overline{u_2})/n}\ee^{\ii (\psi_j-\psi_i)k}\right).
\end{multline*}
Lemma~\ref{lem:Y_n_3_variance} implies
$$
\lim_{n\to\infty}\mathbb{E}[Y_n^{(3)}(u_1,M)Y_n^{(3)}(u_2,M)]=(\sigma_1^2-\sigma_2^2)\sum_{j=1}^{m}\sum_{i=1}^{m}\alpha_j\alpha_i\ee^{u_1+u_2}\left(1-\ee^{-M}\right)\1_{\{\phi_i+\phi_j=2\pi\}}.
$$
and
$$
\lim_{n\to\infty}\mathbb{E}[Y_n^{(3)}(u_1,M)\overline{Y_n^{(3)}(u_2,M)}]=\sigma^2\left(\sum_{j=1}^{m}|\alpha_j|^2\right)\ee^{u_1+\overline{u_2}}(1-\ee^{-M}).
$$
The expressions on the right-hand sides are the covariances of the process
$$
u\mapsto \ee^{u}(\alpha_1 N_{1,M}+\cdots+\alpha_m N_{m,M}),
$$
where $N_{1,M},\ldots,N_{m,M}$ are independent identically distributed complex Gaussian random variables with mean zero and the covariance structure
\begin{equation}\label{eq:limit_gaussian_cov_1_trunc}
\mathbb{E}[N_{i,M}\overline{N_{j,M}}]=\sigma^2(1-\ee^{-M})\1_{\{i=j\}},\quad 1\leq i,j\leq m,
\end{equation}
and
\begin{equation}\label{eq:limit_gaussian_cov_2_trunc}
\mathbb{E}[N_{i,M}N_{j,M}]=(\sigma_1^2-\sigma_2^2)(1-\ee^{-M})\1_{\{\phi_i+\phi_j=2\pi\}},\quad 1\leq i,j\leq m.
\end{equation}

Thus, to finish the proof of the convergence of finite dimensional distributions
\begin{equation}\label{eq:Y_n_3_fdd_conv}
(Y_n^{(3)}(u,M))_{u\in\mathbb{C}}~\tofdd~(\ee^{u}(\alpha_1 N_{1,M}+\cdots+\alpha_m N_{m,M}))_{u\in\mathbb{C}},
\end{equation}
it remains to verify the Lindeberg-Feller condition: for every $\delta>0$,
\begin{equation}\label{eq:lind-feller-alpha<-1/2}
\lim_{n\to\infty}\sum_{k=n-M\lfloor n a_n^{-1}\rfloor}^{n}\mathbb{E}[|W_{n,k}|^2\1_{\{|W_{n,k}|>\delta\}}=0,
\end{equation}
where $W_{n,k}:=b(k)\xi_k\left(\sum_{j=1}^{m}\alpha_j z_n^k(\psi_j,u)\right)$. Without loss of generality assume that there is at least one non-zero $\alpha_j$ and put $\alpha^{\ast}:=\max_{1\leq j\leq m}|\alpha_j|>0$. Writing
$$
\sum_{k=n-M\lfloor n a_n^{-1}\rfloor}^{n}\mathbb{E}[|W_{n,k}|^2\1_{\{|W_{n,k}|>\delta\}}=\sigma^2\sum_{k=n-M\lfloor n a_n^{-1}\rfloor}^{n}b^2(k)\left|\sum_{j=1}^{m}\alpha_j z_n^k(\psi_j,u)\right|^2\mathbb{E}[|\xi|^2\1_{\{|W_{n,k}|>\delta\}}],
$$
ans using $|W_{n,k}|\leq b(k)|\xi_k|\cdot|\alpha^{\ast}|\cdot|z_n^k(\psi^{\ast},u)|$, where $\psi^{\ast}:=\psi_{\arg\max\{|\alpha_j|\}}$, we see that it suffices to check that, for every fixed $0\leq\psi<2\pi$ and $\delta>0$,
$$
\lim_{n\to\infty}\sum_{k=n-M\lfloor n a_n^{-1}\rfloor}^{n}b^2(k)|z_n^k(\psi,u)|^2\mathbb{E}[|\xi|^2\1_{\{b(k)|\xi||z_n^k(\psi,u)|>\delta\}}]=0,
$$
or, equivalently,
\begin{equation}\label{eq:lind-feller-alpha<-1/2_2}
\lim_{n\to\infty}\sum_{k=n-M\lfloor n a_n^{-1}\rfloor}^{n}b^2(k)z_n^{2k}(0,0)\ee^{k(u+\overline{u})/n}\ee^{\ii \psi k}\mathbb{E}[|\xi|^2\1_{\{b(k)|\xi||z_n^k(\psi,u)|>\delta\}}]=0.
\end{equation}
Taking into account Lemma~\ref{lem:Y_n_3_variance} we only need to verify that
\begin{equation}\label{eq:Lind_feller_sec_moment}
\lim_{n\to\infty}\sup_{n-M\lfloor n a_n^{-1}\rfloor\leq k\leq n}\mathbb{E}[|\xi|^2\1_{\{b(k)|\xi||z_n^k(\psi,u)|>\delta\}}]=0.
\end{equation}
Since $\mathbb{E}[|\xi|^2]<\infty$ this will follow from
\begin{equation}\label{eq:lind-feller-alpha<-1/2_3}
\lim_{n\to\infty}\sup_{n-M\lfloor n a_n^{-1}\rfloor\leq k\leq n}b(k)|z_n^k(\psi,u)|=0.
\end{equation}
By the uniform convergence theorem for regularly varying functions and the representation $z_n(\psi,u)=z_n(0,0)\ee^{u/n}\ee^{\ii\psi}$,~\eqref{eq:lind-feller-alpha<-1/2_3} is a consequence of
$$
\lim_{n\to\infty}b(n)\sup_{n-M\lfloor n a_n^{-1}\rfloor\leq k\leq n}z_n^k(0,0)=0.
$$
The latter follows from~\eqref{eq:z_n_main_relation}, the fact that $z_n(0,0)\geq 1$ and $a_n=o(n)$, see Lemma~\ref{lem:W_-1}(ii).

Combining,~\eqref{eq:Y_n_1_func},~\eqref{eq:Y2_to_zero} and~\eqref{eq:Y_n_3_fdd_conv}, the independence of $(Y_n^{(1)}(u))_{u\in\mathbb{C}}$ and $(Y_n^{(3)}(u,M))_{u\in\mathbb{C}}$ and the obvious fact that
$$
(\ee^{u}(\alpha_1 N_{1,M}+\cdots+\alpha_m N_{m,M}))_{u\in\mathbb{C}}~\tofdd(\ee^{u}(\alpha_1 N_{1}+\cdots+\alpha_m N_{m}))_{u\in\mathbb{C}},\quad M\to\infty,
$$
we conclude that~\eqref{eq:prop:boundary:alpha<-1/2_main_CW} holds true in the sense of finite-dimensional distributions by an appeal to~\cite[Theorem 3.2]{Billingsley}.

To finish the proof of the functional convergence~\eqref{eq:prop:boundary:alpha<-1/2_main_CW} it remains to verify the tightness. By~\cite[Lemma 4.2]{Kab+Klim:2014} it suffices to check that
$$
\sup_{n\in\mathbb{N}}\sup_{|u|\leq R}\mathbb{E}[|P_n(z_n(\psi,u))|^2]<\infty,
$$
for every fixed $R>0$. But this follows from the formula
$$
\sup_{|u|\leq R}\mathbb{E}[|P_n(z_n(\psi,u))|^2]=\sup_{|u|\leq R}\sum_{k=0}^{n}b^2(k)z_n^{2k}(0,0)\ee^{2\Re(u)k/n}\leq \ee^{2R}\sum_{k=0}^{n}b^2(k)z_n^{2k}(0,0)
$$
taking into account~\eqref{eq:Y2_to_zero_1},~\eqref{lem:claim1} and
$$
\sum_{k=0}^{\lfloor \log n\rfloor}b^2(k)z_n^{2k}(0,0)\leq C\sum_{k=0}^{\lfloor \log n\rfloor}b^2(k)\leq C\sum_{k=0}^{\infty}b^2(k)<\infty,
$$
for some $C>0$.

Finally, if $\alpha<-\nicefrac{1}{2}$, then $a_n$ can be replaced by $-2\log b(n)-\log n +\log \log n+\log (-2\alpha-1)$ according to part (iii) of Lemma~\ref{lem:W_-1} because in this case
$$
z_n(0,0)=\exp((-2\log b(n)-\log n +\log \log n+\log (-2\alpha-1))/2n)(1+o(n^{-1})).
$$
The proof of Theorem~\ref{thm:boundary:alpha<-1/2} is complete.
\end{proof}

\begin{rem}[Global convergence of zeros in  the strong crystalline phase]\label{rem:zeros_in_annulus_strong_cryst_phase}
As in Theorem~\ref{thm:boundary:alpha<-1/2}, assume~\eqref{eq:moment_assumptions_on_xi1},~\eqref{eq:moment_assumptions_on_xi2}, $\alpha\leq -\nicefrac{1}{2}$ and, if $\alpha=-\nicefrac{1}{2}$ also~\eqref{eq:convergence_b^2}. Corollary~\ref{cor:boundary:alpha<-1/2} suggests that,
for all real $s_1 < s_2$, the number of complex zeros of $P_n$ in the annulus $\mathbb{A}(r_n \ee^{s_1/n},r_n\ee^{s_2/n})=\{z\in \mathbb{C}: r_n\ee^{s_1/n} \leq |z| \leq r_n\ee^{s_2/n}\}$ satisfies
\begin{equation}\label{eq:annulus_strong_cryst}
\lim_{n\to\infty} \frac{1}{n} \zeros_{\mathbb{A}(r_n\ee^{s_1/n},r_n\ee^{s_2/n})} (P_n)   \to  \frac 1 {2\pi}\int_{0}^{2\pi} \left(\ee^{- |P_{\infty}(\ee^{\ii \psi})|^2\ee^{-2s_2}/\sigma^2} - \ee^{- |P_{\infty}(\ee^{\ii \psi})|^2\ee^{-2s_1}/\sigma^2}\right)  \,  {\rm d} \psi, \quad n\to\infty,
\end{equation}
in probability. Note that the limit on the right-hand side is \emph{random} since it contains $P_\infty$.
To give a non-rigorous justification of~\eqref{eq:annulus_strong_cryst},  fix some large $B>0$ and cover the annulus $\mathbb{A}(r_n \ee^{s_1/n},r_n\ee^{s_2/n})$ by $\lfloor 2\pi n / B \rfloor$ windows $W_{j;n} = \{r_n \ee^{u/n} \ee^{\ii B j/n}: \Re u \in [s_1,s_2], \Im u \in [0, B]\}$, where $j \in \{1, \ldots, \lfloor 2\pi n/B \rfloor\}$ and the boundary effects are ignored. For the rest of the argument, condition on a realization of $(P_\infty(\ee^{\ii \psi}))_{\psi\in [0,2\pi]}$.
Corollary~\ref{cor:boundary:alpha<-1/2} suggests that, if $n$ is large, for every window $W_{j;n}$ there are two possibilities: (a) if the vertical lattice $\log(-N_j^{-1}P_{\infty}(\ee^{\ii Bj/n}))+2\pi {\rm i}\mathbb Z$ intersects the rectangle $\{u\in \mathbb C: \Re u \in [s_1,s_2], \Im u \in [0, B]\}$, then the number of complex zeros of $P_n$ in the window $W_{j;n}$ is approximately $B/(2\pi)$ (if $B$ is large), or (b):  if the lattice does not intersect the rectangle, there are no zeros in $W_{j;n}$. Here, the random variables $N_j$ are isotropic complex normal with $\E [|N_j|^2] = \sigma^2$, and this continues to hold even conditionally on $(P_\infty(\ee^{\ii \psi}))_{\psi\in [0,2\pi]}$. (For the last point we note that the $N_j$'s are determined by the diffusive limit of the $\xi_j$'s with $j$ ``close'' to $n$ and  with the contribution of each individual variable being negligible, whereas $(P_\infty(\ee^{\ii \psi}))_{\psi\in [0,2\pi]}$ is mainly determined by the first few $\xi_j$'s.) The probability of (a) is given by
$$
\mathbb P \left[s_1 \leq  \log|N^{-1}_jP_{\infty}(\ee^{\ii Bj/n})| \leq s_2 \;|\; (P_\infty(\ee^{\ii \psi}))_{\psi\in [0,2\pi]}\right]
=
\ee^{- |P_{\infty}(\ee^{\ii Bj/n})|^2\ee^{-2s_2}/\sigma^2} - \ee^{- |P_{\infty}(\ee^{\ii Bj/n})|^2\ee^{-2s_1}/\sigma^2},
$$
since $|N_j|^2$ is exponential with expectation $\sigma^2$.  
Since the random variables $N_j$ corresponding to scaling windows that are far apart behave asymptotically independently, the law of large numbers, applied conditionally on a realization of $(P_\infty(\ee^{\ii \psi}))_{\psi\in [0,2\pi]}$,  suggests~\eqref{eq:annulus_strong_cryst}. For comparison, note that in the setting of Kac polynomials, which belong to the liquid phase, a similar argument suggests that the number of zeros in the annulus $\mathbb{A}(\ee^{s_1/n},\ee^{s_2/n})$, divided by $n$, converges in probability to a \emph{deterministic} limit which coincides with the limit of the expectation that has been determined in~\cite{Ibragimov+Zeitouni:1997}; see Example~\ref{ex:edelman_kostlan}.
\end{rem}

\subsection{Case \texorpdfstring{$\alpha=-\nicefrac{1}{2}$}{alpha=-1/2} and \texorpdfstring{$S(-1;\ell)=+\infty$}{S(-1;l)=infinity}: the weak crystalline phase}\label{sec:zeros_alpha=-1/2}

Throughout this section $\alpha=-\nicefrac{1}{2}$, that is,
$$
b(x)=x^{-1/2}\ell(x),\quad x>0,
$$
for a slowly varying $\ell$. Define the function
\begin{equation}\label{eq:L_definition}
L(x):=\int_0^{x}b^2(t){\rm d}t,\quad x\geq 0,
\end{equation}
and observe that $L$ is monotone increasing and $L(x)~\to+\infty$, as $x\to+\infty$, in view of $S(-1)=+\infty$. Moreover, it is known that $L$ is slowly varying at $+\infty$ and also
\begin{equation}\label{eq:L_fast_enough}
\lim_{n\to\infty}\frac{nb^2(n)}{L(n)}=\lim_{n\to\infty}\frac{\ell^{2}(n)}{L(n)}=0,
\end{equation}
see \cite[Eq.~(1.5.8)]{BGT}.

Next we define the radius $r_n$ of the circle near which the zeros of $P_n$ are concentrated, and ``local coordinates'' in a scaling window near $r_n\ee^{\ii \psi}$.
For a sufficiently large $n\in\mathbb{N}$, put
\begin{align}\label{eq:z_n_alpha=-1/2}
\hat{z}_n(\psi,u):=\exp\left\{-\frac{1}{2n}W_{-1}\left(-\frac{nb^2(n)}{L(n)}\right)+\frac{u}{n}+\ii\psi\right\}=\hat{z}_n(0,0)\ee^{u/n}\ee^{\ii\psi},
\end{align}
where, as before, $W_{-1}$ is the secondary branch of the Lambert $W$-function. Then, $r_n = \hat{z}_n(0,0)$. Denote also $\hat{a}_n:=-W_{-1}(-\ell^{2}(n)/L(n))$ and observe that
\begin{equation}\label{eq:a_n_hat_choice}
\frac{\ee^{\hat{a}_n}}{\hat{a}_n}=\frac{L(n)}{\ell^2(n)},
\end{equation}
for all sufficiently large $n$. Both $\hat{a}_n$ and $\hat{z}_n(\psi,u)$ are well-defined because $0<\ell^2(n)/L(n)<\ee^{-1}$ for all sufficiently large $n\in\mathbb{N}$ by~\eqref{eq:L_fast_enough}. For further use note that asymptotic expansion for $W_{-1}$ given by~\eqref{eq:W_-1_asymp} below implies that
\begin{equation}\label{eq:hat_a_n_asymp}
\hat{a}_n~\sim~\log \left(\frac{L(n)}{nb^2(n)}\right)~\sim~\log \left(\frac{L(n)}{\ell^2(n)}\right),\quad n\to\infty.
\end{equation}
In particular, the sequence $(\hat{a}_n)$ is slowly varying.

\begin{theorem}\label{thm:critical}
Assume~\eqref{eq:moment_assumptions_on_xi1} and~\eqref{eq:moment_assumptions_on_xi2}. Suppose also that $\alpha=-\nicefrac{1}{2}$ and $S(-1)=+\infty$. Then, for any $0\leq \psi_1<\cdots<\psi_m<2\pi$,
$$
\left(\frac{1}{\sqrt{L(n)}}P_n(\hat{z}_n(\psi_i,u))\right)_{u\in\mathbb{C},i=1,\ldots,m}~\Longrightarrow~(\hat{N}_i+\ee^{u}N_i)_{u\in\mathbb{C},i=1,\ldots,m},\quad n\to\infty.
$$
on the space $\mathcal{A}(\mathbb{C})^m$, where $(\hat{N}_1,\ldots,\hat{N}_m)$ and $(N_1,\ldots,N_m)$ are independent Gaussian vectors with the same distribution given by~\eqref{eq:limit_gaussian_cov_1} and~\eqref{eq:limit_gaussian_cov_2}.
\end{theorem}

\begin{corollary}[Local convergence of zeros near the unit circle, weak crystalline phase $\alpha\leq -\nicefrac{1}{2}$ and $\sum b^2(k)=+\infty$]
\label{cor:3.3}
Under the assumptions of Theorem~\ref{thm:critical}, the sequence
$$
(\zeros_{\mathbb{C}}(P_n(\hat{z}_n(\psi_1,\cdot))),\ldots,\zeros_{\mathbb{C}}(P_n(\hat{z}_n(\psi_m,\cdot))),\quad n\in\mathbb{N},
$$
regarded as a sequence of random elements of $(M_p(\mathbb{C}))^m$, converges in distribution to
\begin{multline}\label{eq:limit_zeros_critical}
\left(\zeros_{\mathbb{C}}(\hat{N}_1+\ee^{(\cdot)}N_1),\ldots,\zeros_{\mathbb{C}}(\hat{N}_m+\ee^{(\cdot)}N_m)\right)=\\
\left(\log(-N_1^{-1}\hat{N}_1)+2\pi {\rm i}\mathbb{Z},\ldots,\log(-N_m^{-1}\hat{N}_m)+2\pi {\rm i}\mathbb{Z}\right),
\end{multline}
where $\log$ denotes any branch (say, principal) of the complex logarithm.
\end{corollary}

For the proof of Theorem~\ref{thm:critical} we need a lemma.
\begin{lemma}\label{lem:L_a_hat_asymp}
For every $\gamma>0$ and $A>0$, $\lim_{n\to\infty}L(n A \hat{a}_n^{-\gamma})/L(n)=1$.
\end{lemma}
\begin{proof}
In view of the slow variation of $L$, it suffices to check the statement for $A=1$. Put $T(x):=1/\log(\ell^{-2}(x)L(x))$ for sufficiently large $x>0$. The function $x\mapsto T(x)$ is slowly varying, thus, by the representation theorem for slowly varying functions,
$$
T(x)=c(x)\exp\left(\int_{a}^{x}\frac{\varepsilon(s)}{s}{\rm d}s\right)=:c(x)T_1(x),
$$
where $\lim_{x\to+\infty}c(x)=c>0$ and $\lim_{x\to+\infty}\varepsilon(x)=0$, see~\cite[Theorem 1.3.1]{BGT}. According to~\eqref{eq:hat_a_n_asymp} and the slow variation of $L$ it suffices to check that, for every $\gamma>0$,
\begin{equation}\label{eq:L_a_hat_asymp}
\lim_{n\to\infty}\frac{L(n T_1^{\gamma}(n))}{L(n)}=1.
\end{equation}
Observe that, for every $\delta>0$, $x\mapsto x^{\delta}T_1(x)$ is eventually increasing. According to~\cite[Theorem 2]{Bojanic+Seneta:1971}, limit relation~\eqref{eq:L_a_hat_asymp} holds true provided that
\begin{equation}\label{eq:L_a_hat_asymp_proof1}
\lim_{x\to+\infty}\frac{L(\lambda x)-L(x)}{L(x)}\log T_1(x)=0,
\end{equation}
for every $\lambda>1$. It is known, see~\cite[p.~127]{BGT} that, for every fixed $\lambda>0$,
$$
\lim_{x\to\infty}\frac{L(\lambda x)-L(x)}{\ell^{2}(x)}=\log \lambda,
$$
which means that $L$ is a slowly varying function of the de Haan class $\Pi$. Thus,
\begin{multline*}
\frac{L(\lambda x)-L(x)}{L(x)}\log T_1(x)~=~\frac{L(\lambda x)-L(x)}{\ell^{2}(x)}\frac{\ell_2(x)}{L(x)}\log T_1(x)~\sim~(-1)\frac{L(\lambda x)-L(x)}{\ell^{2}(x)}\frac{\ell_2(x)}{L(x)}\log \log \frac{L(x)}{\ell_2(x)}\\
~\sim~(-\log \lambda)\frac{\ell_2(x)}{L(x)}\log \log \frac{L(x)}{\ell_2(x)}\to 0,\quad x\to+\infty,
\end{multline*}
which proves~\eqref{eq:L_a_hat_asymp_proof1}.
\end{proof}

\begin{proof}[Proof of Theorem~\ref{thm:critical}]
We need to verify that, for any $\alpha_1,\ldots,\alpha_m\in\mathbb{C}$,
\begin{equation}\label{eq:prop:boundary:alpha=-1/2_main_CW}
\left(\frac{1}{\sqrt{L(n)}}\sum_{i=1}^{m}\alpha_i P_n(\hat{z}_n(\psi_i,u))\right)_{u\in\mathbb{C}}~\Longrightarrow~\left(\sum_{i=1}^{m}\alpha_i \hat{N}_i+\ee^{u}\sum_{i=1}^{m}\alpha_i N_i\right)_{u\in\mathbb{C}},\quad n\to\infty,
\end{equation}
on the space $\mathcal{A}(\mathbb{C})$.

Fix $A\geq 1$, $M\in\mathbb{N}$ and decompose (for a sufficiently large $n$) the left-hand side of~\eqref{eq:prop:boundary:alpha=-1/2_main_CW} as follows
\begin{multline*}
\frac{1}{\sqrt{L(n)}}\sum_{k=0}^{\lfloor n \hat{a}_n^{-2} \rfloor}b(k)\xi_k \left(\sum_{j=1}^{m}\alpha_j \hat{z}_n^k(\psi_j,u)\right)+\frac{1}{\sqrt{L(n)}}\sum_{j=1}^{m}\alpha_j \sum_{k=\lfloor n \hat{a}_n^{-2} \rfloor+1}^{\lfloor n\ee^{-M \hat{a}_n^{-1}}\rfloor}b(k)\xi_k \hat{z}_n^k(\psi_j,u)\\
+\frac{1}{\sqrt{L(n)}}\sum_{k=\lfloor n\ee^{-M \hat{a}_n^{-1}}\rfloor+1}^{n}b(k)\xi_k \sum_{j=1}^{m}\alpha_j \hat{z}_n^k(\psi_j,u)=:\hat{Y}^{(1)}_n(u)+\sum_{j=1}^{m}\alpha_j \hat{Y}^{(2)}_{n,j}(u,M)+\hat{Y}^{(3)}_n(u,M).
\end{multline*}

We shall first prove that on the space $\mathcal{A}(\mathbb{C})$,
\begin{equation}\label{eq:prop:boundary:alpha=-1/2_Y_1}
(\hat{Y}^{(1)}_n(u))_{u\in\mathbb{C}}~\Longrightarrow~\left(\sum_{j=1}^{m}\alpha_j \hat{N}_j\right)_{u\in\mathbb{C}},\quad n\to\infty,
\end{equation}
and
\begin{equation}\label{eq:prop:boundary:alpha=-1/2_Y_3}
(\hat{Y}^{(3)}_n(u))_{u\in\mathbb{C}}~\Longrightarrow~\left(\ee^{u}\sum_{j=1}^{m}\alpha_j N_{j,M}\right)_{u\in\mathbb{C}},\quad n\to\infty,
\end{equation}
where $N_{1,M},\ldots,N_{m,M}$ are complex Gaussian variables with covariances~\eqref{eq:limit_gaussian_cov_1_trunc} and~\eqref{eq:limit_gaussian_cov_2_trunc}.

For the proof of~\eqref{eq:prop:boundary:alpha=-1/2_Y_1} observe that, for every compact set $K\subset\mathbb{C}$,
\begin{equation}\label{eq:critical_variance_asymp_1}
\lim_{n\to\infty}\sup_{u\in K}\sup_{1\leq k\leq \lfloor n \hat{a}_n^{-2}\rfloor }|\ee^{uk/n}\hat{z}_n^{k}(0,0)-1|=0
\end{equation}
because
$$
\hat{z}_n^{n \hat{a}_n^{-2}}(0,0)=\ee^{2^{-1}\hat{a}_n^{-1}}\to 1,\quad n\to\infty.
$$
Thus,~\eqref{eq:prop:boundary:alpha=-1/2_Y_1} follows from
\begin{equation}\label{eq:critical_clt_1}
\frac{1}{\sqrt{L(n)}}\sum_{k=0}^{\lfloor n \hat{a}_n^{-2} \rfloor}b(k)\xi_k \left(\sum_{j=1}^{m}\alpha_j\ee^{\ii\psi_j k}\right)~\overset{{\rm d}}{\to}~\sum_{j=1}^{m}\alpha_j \hat{N}_j,\quad n\to\infty.
\end{equation}
The latter is a consequence of the Lindeberg-Feller central limit theorem taking into account that
\begin{equation}\label{eq:critical_variance_asymp_2}
\frac{1}{L(n)}\sum_{k=0}^{\lfloor n \hat{a}_n^{-2} \rfloor}b^2(k)\ee^{\ii \theta k}
=\begin{cases}
1,&\theta\in 2\pi\mathbb{Z},\\
0,&\theta\not\in 2\pi\mathbb{Z}.
\end{cases}
\end{equation}
If $\theta\in 2\pi\mathbb{Z}$, then the above limit relation holds in view of Lemma~\ref{lem:L_a_hat_asymp} applied with $\gamma=2$ and $A=1$. If $\theta\not\in 2\pi\mathbb{Z}$, using summation by parts and setting
$$
U_n(\theta):=\sum_{k=n}^{\infty}k^{-1}\ee^{\ii\theta k},\quad n=1,2,\ldots,
$$
we obtain
\begin{multline}\label{eq:summation_by_parts1}
\sum_{k=1}^{n}b^2(k)\ee^{\ii \theta k}=\sum_{k=1}^{n}\ell^2(k)(U_k(\theta)-U_{k+1}(\theta))\\
=\ell^2(1) U_1(\theta)-\ell^{2}(n)U_{n+1}(\theta)+\sum_{k=2}^{n}U_k(\theta)(\ell^2(k)-\ell^2(k-1)).
\end{multline}
Using yet another summation by parts and the fact that $\theta\not\in 2\pi\mathbb{Z}$, we conclude
$$
|U_n(\theta)|\leq \frac{c(\theta)}{n},~\text{for some }c(\theta)>0~\text{and all}~n\in\mathbb{N}.
$$
Thus,
$$
\sum_{k=2}^{n}U_k(\theta)|\ell^2(k)-\ell^2(k-1)|\leq c(\theta)\sum_{k=2}^{n}\frac{|\ell^2(k)-\ell^2(k-1)|}{k}=c(\theta)\sum_{k=2}^{n}b^2(k)\left|1-\frac{\ell^2(k-1)}{\ell^2(k)}\right|.
$$
By the slow variation of $\ell^{2}$, the right-hand side is $o(L(n))$. Therefore,~\eqref{eq:summation_by_parts1} implies
$$
\lim_{n\to\infty}\frac{1}{L(n)}\sum_{k=0}^{n}b^2(k)\ee^{\ii \theta k}=0,\quad \theta\not\in 2\pi\mathbb{Z}.
$$
Verification of the Lindeberg-Feller condition in~\eqref{eq:critical_clt_1}, which goes along the same lines as in the proof of Theorem~\ref{thm:boundary:alpha<-1/2}, is left to the reader.

The proof of~\eqref{eq:prop:boundary:alpha=-1/2_Y_3} is similar. First, observe that
$$
\lim_{n\to\infty}\ee^{-M \hat{a}_n^{-1}}=1
$$
and thereupon
$$
\lim_{n\to\infty}\left|\sup_{\lfloor n\ee^{-M \hat{a}_n^{-1}}\rfloor+1\leq k\leq n}\ee^{uk/n}-\ee^{u}\right|=0,
$$
which means that it is sufficient to prove that
\begin{equation}\label{eq:hat_y_n3_clt_proof}
\lim_{n\to\infty}\frac{1}{\sqrt{L(n)}}\sum_{k=\lfloor n\ee^{-M \hat{a}_n^{-1}}\rfloor+1}^{n}b(k)\xi_k \sum_{j=1}^{m}\alpha_j \hat{z}_n^k(0,0)\ee^{\ii \psi_j k}~\overset{{\rm d}}{\to}~\sum_{j=1}^{m}\alpha_j N_{j,M},\quad n\to\infty.
\end{equation}
This again follows by the Lindeberg-Feller central limit theorem taking into account the next lemma

\begin{lemma}\label{lem:critical_variance_asymp_3}
Under the assumptions of Theorem~\ref{thm:critical}, for every fixed $\theta\in\mathbb{R}$ and $M>0$,
$$
\frac{1}{L(n)}\sum_{k=\lfloor n\ee^{-M \hat{a}_n^{-1}}\rfloor+1}^{n}b^2(k)\hat{z}_n^{2k}(0,0)\ee^{\ii \theta k}=\begin{cases}
(1-\ee^{-M}),&\theta\in 2\pi\mathbb{Z},\\
0,&\theta\not\in 2\pi\mathbb{Z}.
\end{cases}
$$
\end{lemma}
\begin{proof}
Assume first that $\theta\in 2\pi\mathbb{Z}$. By the uniform convergence theorem for regularly varying functions,
\begin{multline*}
\sum_{k=\lfloor n\ee^{-M \hat{a}_n^{-1}}\rfloor+1}^{n}b^2(k)\hat{z}_n^{2k}(0,0)~\sim~b^2(n)\sum_{k=\lfloor n\ee^{-M \hat{a}_n^{-1}}\rfloor+1}^{n}\hat{z}_n^{2k}(0,0)
=b^2(n)\ee^{\hat{a}_n}\sum_{k=0}^{n-\lfloor n\ee^{-M \hat{a}_n^{-1}}\rfloor-1}\ee^{-\hat{a}_nk/n}\\=b^2(n)\ee^{\hat{a}_n}\frac{1-\ee^{-\hat{a}_n(n-\lfloor n\ee^{-M \hat{a}_n^{-1}}\rfloor)/n}}{1-\ee^{-\hat{a}_n/n}}\overset{\eqref{eq:a_n_hat_choice}}{\sim}L(n)(1-\ee^{-\hat{a}_n(n-\lfloor n\ee^{-M \hat{a}_n^{-1}}\rfloor)/n}).
\end{multline*}
The expression in the last parentheses converges to $1-\ee^{-M}$, as $n\to\infty$, in view of $\ee^{-M\hat{a}_n^{-1}}~\sim~1-M\hat{a}_n^{-1}+o(\hat{a}_n^{-1})$. The case $\theta\not\in 2\pi\mathbb{Z}$ follows by the same reasoning as we have used in the proof if~\eqref{lem:claim2}. We omit the details.
\end{proof}

Similarly to~\eqref{eq:Lind_feller_sec_moment}, the following is sufficient for the validity of the Lindeberg-Feller condition in~\eqref{eq:hat_y_n3_clt_proof}:
$$
\lim_{n\to\infty}\sup_{\lfloor n\ee^{-M \hat{a}_n^{-1}}\rfloor\leq k\leq n}\mathbb{E}[\xi^2 \1_{\{b(k)|\xi|\hat{z}_n^k(0,0)>\delta\}}]=0.
$$
This relation holds true by the uniform convergence theorem for regularly varying functions in conjunction with
$$
b(n)\hat{z}_n^n(0,0)=b(n)\ee^{\hat{a}_n/2}~\sim~\sqrt{\frac{L(n)\hat{a}_n}{n}}~\to 0,\quad n\to\infty.
$$

To finish the proof of~\eqref{eq:prop:boundary:alpha=-1/2_main_CW} it remains to verify that, for every fixed $j=1,\ldots,m$, $\delta>0$ and a compact set $K\subset\mathbb{C}$,
$$
\lim_{M\to+\infty}\limsup_{n\to\infty}\mathbb{P}\{\sup_{u\in K}|\hat{Y}_{n,j}^{(2)}(u,M)|>\delta\}=0.
$$
By Markov's inequality and analyticity of $u\mapsto \hat{Y}_{n,j}^{(2)}(u,M)$ it suffices to check that
\begin{equation}\label{eq:billingsley_critical}
\lim_{M\to+\infty}\limsup_{n\to\infty}\sup_{u\in K}\mathbb{E}|\hat{Y}_{n,j}^{(2)}(u,M)|^2=0,\quad 1\leq j\leq m,
\end{equation}
see Lemma 4.2 in~\cite{Kab+Klim:2014}. Observe that
\begin{multline*}
\sup_{u\in K}\mathbb{E}|\hat{Y}_{n,j}^{(2)}(u,M)|^2=
\frac{\sigma^2}{L(n)}\sup_{u\in K}\sum_{k=\lfloor n \hat{a}_n^{-2} \rfloor+1}^{\lfloor n\ee^{-M \hat{a}_n^{-1}}\rfloor}b^2(k) \hat{z}_n^{2k}(0,0)\ee^{k(u+\overline{u})/n}\\
\leq \frac{\sigma^2}{L(n)}\left(\sup_{u\in K}\left(\ee^{2\Re(u)}\right)\right)\sum_{k=\lfloor n \hat{a}_n^{-2} \rfloor+1}^{\lfloor n\ee^{-M \hat{a}_n^{-1}}\rfloor}b^2(k) \hat{z}_n^{2k}(0,0).
\end{multline*}
Fix some $A>1$ and decompose the last sum as follows
$$
\sum_{k=\lfloor n \hat{a}_n^{-2} \rfloor+1}^{\lfloor n\ee^{-M \hat{a}_n^{-1}}\rfloor}b^2(k) \hat{z}_n^{2k}(0,0)=\sum_{k=\lfloor n \hat{a}_n^{-2} \rfloor+1}^{\lfloor n A\hat{a}_n^{-1} \rfloor}b^2(k) \hat{z}_n^{2k}(0,0)+\sum_{k=\lfloor n A\hat{a}_n^{-1} \rfloor+1}^{\lfloor n\ee^{-M \hat{a}_n^{-1}}\rfloor}b^2(k) \hat{z}_n^{2k}(0,0)
$$
and observe that the first sum is bounded by
$$
\ee^{A}\sum_{k=\lfloor n \hat{a}_n^{-2} \rfloor+1}^{\lfloor n \hat{a}_n^{-1} \rfloor}b^2(k)\leq \ee^{A}(L(n)-L(n \hat{a}_n^{-2}))=o(L(n)),
$$
where the last equality is a consequence of Lemma~\ref{lem:L_a_hat_asymp}.

To treat the second sum we shall use the Potter's bound for regularly varying functions, see~\cite[Theorem 1.5.6]{BGT}. Pick $0<\varepsilon<A-1$ and choose $n$ such that
$$
\frac{b^2(k)}{b^2(n)}\leq 2\left(\frac{n}{k}\right)^{1+\varepsilon},\quad n \hat{a}_n^{-1}\leq k\leq n.
$$
Then
$$
\sum_{k=\lfloor n A \hat{a}_n^{-1} \rfloor+1}^{\lfloor n\ee^{-M \hat{a}_n^{-1}}\rfloor}b^2(k) \hat{z}_n^{2k}(0,0)\leq 2b^2(n)n^{1+\varepsilon} \sum_{k=\lfloor n A \hat{a}_n^{-1} \rfloor+1}^{\lfloor n\ee^{-M \hat{a}_n^{-1}}\rfloor}\left(\frac{1}{k}\right)^{1+\varepsilon} \ee^{\hat{a}_n k/n}
$$
The function $x\mapsto x^{-(1+\varepsilon)}\ee^{\hat{a}_n x/n}$ is increasing for $x> n\hat{a}_n^{-1}(1+\varepsilon)$ and, by the choice of $\varepsilon$, for $x> n\hat{a}_n^{-1}A$. Using this fact and estimating the sum by the corresponding integral, we conclude
$$
\sum_{k=\lfloor nA\hat{a}_n^{-1}\rfloor+1}^{\lfloor n\ee^{-M \hat{a}_n^{-1}}\rfloor}k^{-(1+\varepsilon)}\ee^{\hat{a}_n k/n}\leq \int_{n A\hat{a}_n^{-1}}^{n\ee^{-M \hat{a}_n^{-1}}+1}\frac{\ee^{\hat{a}_n t/n}}{t^{1+\varepsilon}}{\rm d}t=\frac{a_n^{\varepsilon}}{n^{\varepsilon}}\int_{A}^{\hat{a}_n(\ee^{-M \hat{a}_n^{-1}} + n^{-1})}\frac{\ee^t}{t^{1+\varepsilon}}{\rm d}t.
$$
Applying the asymptotic relation $\int_{1}^{x}t^{-(1+\varepsilon)}\ee^{t}{\rm d}t~\sim~x^{-(1+\varepsilon)}\ee^{x}$, $x\to+\infty$, which can be verified by L'H\^{o}pital's rule, we conclude that
$$
\int_{A}^{\hat{a}_n(\ee^{-M \hat{a}_n^{-1}} + n^{-1})}\frac{\ee^t}{t^{1+\varepsilon}}{\rm d}t~\sim~\frac{\ee^{\hat{a}_n(\ee^{-M \hat{a}_n^{-1}} + n^{-1})}}{\hat{a}^{1+\varepsilon}_n(\ee^{-M \hat{a}_n^{-1}} + n^{-1})^{1+\varepsilon}}~\sim~\frac{\ee^{\hat{a}_n}}{\hat{a}^{1+\varepsilon}_n}\ee^{-M}.
$$
Combining pieces together we conclude that, for all sufficiently large $n\in\mathbb{N}$,
\begin{align*}
\sum_{k=\lfloor n A \hat{a}_n^{-1} \rfloor+1}^{\lfloor n\ee^{-M \hat{a}_n^{-1}}\rfloor}b^2(k) \hat{z}_n^{2k}(0,0)\leq 3 b^2(n)n^{1+\varepsilon}\frac{a_n^{\varepsilon}}{n^{\varepsilon}}\frac{\ee^{\hat{a}_n}}{\hat{a}^{1+\varepsilon}_n}\ee^{-M}~\sim~3\ell^{2}(n)\frac{\ee^{\hat{a}_n}}{\hat{a}_n}\ee^{-M}~\sim~3L(n)\ee^{-M}.
\end{align*}
Thus,~\eqref{eq:billingsley_critical} is proved. The proof of Theorem~\ref{thm:critical} is complete.
\end{proof}

\begin{ex}
\label{example:3.2}
Assume~\eqref{eq:moment_assumptions_on_xi1} and~\eqref{eq:moment_assumptions_on_xi2} and recall the polynomials $P^{{\rm crit}}_{n}$ defined in~\eqref{eq:critical}. In this case $b^2(x)=1/x$ and $L(x)=\log x+O(1)$, for $x\geq 1$. According to Theorem~\ref{thm:critical}, for any $0\leq \psi_1<\cdots<\psi_m<2\pi$,
$$
\left(\frac{1}{\sqrt{\log n}}P^{{\rm crit}}_n(\hat{z}^{{\rm crit}}_n(\psi_i,u))\right)_{u\in\mathbb{C},i=1,\ldots,m}~\Longrightarrow~(\hat{N}_i+\ee^{u}N_i)_{u\in\mathbb{C},i=1,\ldots,m},\quad n\to\infty.
$$
on the space $\mathcal{A}(\mathbb{C})^m$, where
$$
\hat{z}^{{\rm crit}}_n(\psi,u)=\exp\left(\frac{\log \log n+\log\log\log n}{2n}\right)\ee^{u/n+\ii \psi}.
$$
The form of $\hat{z}^{{\rm crit}}_n(\psi_i,u)$ follows from~\eqref{eq:z_n_alpha=-1/2} in view of the asymptotic relation
$$
W_{-1}(x)=-\log(-1/x)-\log\log(-1/x)+o(1),\quad x\to 0-,
$$
see~\eqref{eq:W_-1_asymp} below. Consequently, the sequence $$
(\zeros_{\mathbb{C}}(P^{{\rm crit}}_n(\hat{z}^{{\rm crit}}_n(\psi_1,\cdot))),\ldots,\zeros_{\mathbb{C}}(P^{{\rm crit}}_n(\hat{z}^{{\rm crit}}_n(\psi_m,\cdot))),\quad n\in\mathbb{N},
$$
converges in distribution to the lattices defined by~\eqref{eq:limit_zeros_critical}.
\end{ex}

\begin{rem}[Global distribution of zeros in the weak crystalline phase]\label{rem:zeros_in_annulus_weak_cryst_phase}
As in Theorem~\ref{thm:critical}, assume~\eqref{eq:moment_assumptions_on_xi1}, \eqref{eq:moment_assumptions_on_xi2},  $\alpha=-\nicefrac{1}{2}$ and $S(-1)=+\infty$.
Corollary~\ref{cor:3.3} suggests that,
for all real $s_1 < s_2$, the number of complex zeros of $P_n$ in the annulus $\mathbb{A}(r_n \ee^{s_1/n},r_n\ee^{s_2/n})=\{z\in \mathbb{C}: r_n\ee^{s_1/n} \leq |z| \leq r_n\ee^{s_2/n}\}$ satisfies
\begin{equation}\label{eq:annulus_weak_cryst}
\lim_{n\to\infty} \frac{1}{n} \zeros_{\mathbb{A}(r_n\ee^{s_1/n},r_n\ee^{s_2/n})} (P_n)   \to  \frac 12 \int_{s_1}^{s_2} \frac {{\rm d} s} {\cosh^2 s}, \quad n\to\infty,
\end{equation}
in probability and in $L^1$. Note that the limit on the right-hand side is \emph{deterministic}, in contrast with the situation discussed in Remark~\ref{rem:zeros_in_annulus_strong_cryst_phase}. The non-rigorous argument leading to~\eqref{eq:annulus_weak_cryst}  is similar to that leading to~\eqref{eq:annulus_strong_cryst}, but now there is no need to condition on $(P_{\infty}(\ee^{\ii \psi}))_{\psi\in [0,2\pi]}$, and the probability that the window $W_{j;n}$ contains zeros is
$$
\mathbb P \left[s_1 \leq  \log|N^{-1}_j \hat N_j| \leq s_2\right]
=
\frac 12 \int_{s_1}^{s_2} \frac {{\rm d} s} {\cosh^2 s}.
$$
Since scaling windows that are far apart behave asymptotically independently, the law of large numbers suggests~\eqref{eq:annulus_weak_cryst}.
\end{rem}


\subsection{The limiting crossovers and intensities of zeros}\label{sec:crossover}
As shown in the previous sections, the limiting Gaussian process of the appropriately scaled $P_n$ is $G_{\psi}$  when $\alpha > -\nicefrac{1}{2}$ (liquid phase), and a \emph{degenerate} Gaussian process
$(\hat{N}_1 + \ee^{u}N_1)_{u \in \mathbb{C}}$ when $\alpha = -\nicefrac{1}{2}$  and
$S(-1) = +\infty$ (weak crystalline phase). Our goal is to connect these two regimes by
describing the crossover of the limiting processes as $\alpha \to -\nicefrac{1}{2} + 0$. Throughout this section, we shall include two indices in the notation for
$G_{\psi}$, writing it as $G_{\psi,\alpha}$ to emphasize its dependence on the now-varying parameter $\alpha$.

\begin{theorem}\label{thm:crossover}
In the limit when $\alpha\to -\nicefrac{1}{2}+0$ it holds that
\begin{equation}\label{eq:crossover}
\left(\sqrt{2\alpha+1}G_{\psi,\alpha}\left(\frac{1}{2}\log \frac{1}{1+2\alpha} + \frac{1}{2}\log\log \frac{1}{1+2\alpha}+u\right)\right)_{u\in\mathbb{C}}~\Longrightarrow~(\hat{N}_1 + \ee^{u}N_1)_{u \in \mathbb{C}} 
\end{equation}
on the space $\mathcal{A}(\mathbb{C})$, where $\hat{N}_1$ and $N_1$ are independent centered complex normal variables with the same distribution determined by
\begin{equation}\label{eq:limit_gaussian_cov_alt}
\mathbb{E}[N_1\overline{N_1}]=\sigma^2,\quad \mathbb{E}[N_1^2]=(\sigma_1^2-\sigma_2^2)\1_{\{\psi\in \{0,\pi\}\}}.
\end{equation}
\end{theorem}
\begin{proof}
Put $\tau:=2\alpha+1$. Thus, $\alpha\to -\nicefrac{1}{2}+0$ if and only if $\tau\to 0+$. Since the processes on both sides of~\eqref{eq:crossover} are analytic Gaussian, the stated convergence follows once we show that the covariances of the processes on the left-hand side converge locally uniformly on $\mathbb{C}^2$ to the covariance of the process on the right-hand side. The tightness follows then immediately since the second moments are uniformly bounded on compact sets, see~\cite[Lemma 4.2]{Kab+Klim:2014}. Recalling the notation  for the function $\Phi_{\beta}$ in~\eqref{eq:Phi_def} and the relations~\eqref{eq:covariance:Ga} and~\eqref{eq:covariance:Gb}, we see that it suffices to show that locally uniformly in $u\in\mathbb{C}$.
\begin{equation}\label{eq:crossover_covariance}
\lim_{\tau\to 0+}\tau\Phi_{\tau-1}(\log \tau^{-1}+\log\log \tau^{-1}+u)=1+\ee^{u}.
\end{equation}
Put $s(\tau,u):=\log \tau^{-1}+\log\log \tau^{-1}+u$ and write
$$
\Phi_{\tau-1}(\log \tau^{-1}+\log\log \tau^{-1}+u)=\int_0^{1}x^{\tau-1}\ee^{x s(\tau,u)}{\rm d}x=\int_0^{\nicefrac{1}{2}}x^{\tau-1}\ee^{x s(\tau,u)}{\rm d}x+\int_{\nicefrac{1}{2}}^{1}x^{\tau-1}\ee^{x s(\tau,u)}{\rm d}x
$$
For the second integral observe that $x^{\tau-1}$ converges to $x^{-1}$ uniformly in $x\in [\nicefrac{1}{2},1]$, as $\tau\to 0+$. Therefore, locally uniformly in $u\in\mathbb{C}$,
\begin{multline*}
\int_{\nicefrac{1}{2}}^{1}x^{\tau-1}\ee^{x s(\tau,u)}{\rm d}x~\sim~\int_{\nicefrac{1}{2}}^{1}x^{-1}\ee^{x s(\tau,u)}{\rm d}x=\int_{s(\tau,u)/2}^{s(\tau,u)}x^{-1}\ee^{x}{\rm d}x\\
=\int_{1}^{s(\tau,u)}x^{-1}\ee^{x}{\rm d}x-\int_{1}^{s(\tau,u)/2}x^{-1}\ee^{x}{\rm d}x~\sim~\frac{\ee^{s(\tau,u)}}{s(\tau,u)}=\frac{\ee^{s(\tau,0)}\ee^{u}}{s(\tau,0)+u}~\sim~\frac{\ee^{u}}{\tau},\quad \tau\to 0+.
\end{multline*}
Here, the penultimate equivalence follows from the relation
$$
\int_{1}^{t}x^{-1}\ee^{x}{\rm d}x~\sim~\frac{\ee^{t}}{t},\quad t\to\infty,
$$
whereas the last equivalence is a consequence of $\lim_{\tau\to 0+}\frac{\tau\ee^{s(\tau,0)}}{s(\tau,0)}=1$.

For the first integral, note that
\begin{multline*}
\int_0^{\nicefrac{1}{2}}x^{\tau-1}\ee^{x s(\tau,u)}{\rm d}x=\int_0^{\nicefrac{1}{2}}x^{\tau-1}\ee^{x s(\tau,u)}{\rm d}x=(s(\tau,u))^{-\tau}\int_0^{s(\tau,u)/2}x^{\tau-1}\ee^{x}{\rm d}x\\
=(s(\tau,u))^{-\tau}\left(\int_0^{1}+\int_1^{s(\tau,u)/2}\right)x^{\tau-1}\ee^{x}{\rm d}x.
\end{multline*}
The second summand is bounded by $(s(\tau,u))^{-\tau}\ee^{s(\tau,u)/2}=o(\tau)$ locally uniformly in $u\in\mathbb{C}$. The first summand can be further decomposed as
$$
(s(\tau,u))^{-\tau}\int_0^{1}x^{\tau-1}(\ee^{x}-1){\rm d}x+\tau^{-1}(s(\tau,u))^{-\tau}~\sim~\tau^{-1},\quad \tau\to 0+,
$$
where we have used that the first integral converges to a finite constant $\int_0^{1}x^{-1}(\ee^{x}-1){\rm d}x$. Combining  pieces together we obtain~\eqref{eq:crossover_covariance}. The proof is complete.
\end{proof}

Theorem~\ref{thm:crossover} implies the convergence of the point processes. The zeros of $G_{\psi,\alpha}$, as a point process, converge after `centering' by $\frac{1}{2}\log \frac{1}{1+2\alpha}+ \frac{1}{2}\log\log \frac{1}{1+2\alpha}$ to a randomly shifted lattice. Formally, we have the following
\begin{corollary}\label{cor:crossover}
The following holds true
$$
\zeros_{\mathbb{C}}\left(G_{\psi,\alpha}\left(\frac{1}{2}\log \frac{1}{1+2\alpha}+ \frac{1}{2}\log\log \frac{1}{1+2\alpha}+(\cdot)\right)\right)~\Longrightarrow~\log(-N_1^{-1}\hat{N}_1)+2\pi {\rm i}\mathbb{Z},\quad \alpha\to -\nicefrac{1}{2}+0.
$$
\end{corollary}

Let us emphasize that the above results describe the crossover behavior of the limiting processes and their zeros. In general, this does not imply a corresponding crossover for the zero intensities of $P_n$, except in
the case of isotropic complex normal coefficients, as discussed in Example~\ref{ex:edelman_kostlan}.
Let us now return to this example. Equation~\eqref{eq:RadialIntensity} for the radial intensity $p_n(\alpha;r)$ holds for the random polynomials $P_n$ with the $\xi$ having isotropic complex normal distribution, for every value of $\alpha$ and degree $n$. In particular, it is true also in the strong crystalline phase $\alpha\leq -\nicefrac{1}{2}$ and $S(2\alpha)<+\infty$. Recall that $z_n(0,0)$ is defined by~\eqref{eq:z_n_alpha<=-1/2}. We have, for every $s\in\mathbb{R}$,
\begin{multline}\label{eq:radial_intensity_crystalline}
\frac{p_n (\alpha;z_n(0,0)\ee^{s/n})}{n^2} ~\sim~ \\ \frac{1}{\pi n^2} \frac{S_n(2\alpha+2;\ell;z^2_n(0,0)\ee^{2s/n}) S_n(2\alpha;\ell; z^2_n(0,0)\ee^{2s/n}) - S_n^2(2\alpha+1;\ell; z^2_n(0,0)\ee^{2s/n})}{S_n^2(2\alpha;\ell; z^2_n(0,0)\ee^{2s/n})}.
\end{multline}
Using the decomposition
\begin{multline*}
S_n(2\alpha;\ell; z^2_n(0,0)\ee^{2z/n})=\sum_{k=1}^{n}b^2(k)z_n^{2k}(0,0)\ee^{2zk/n}\\
=\left(\sum_{k=1}^{\lfloor \log n\rfloor }+\sum_{k=\lfloor \log n\rfloor}^{n-M\lfloor n a_n^{-1}\rfloor}+\sum_{k=n-M\lfloor n a_n^{-1}\rfloor+1}^{n}\right)b^2(k)z_n^{2k}(0,0)\ee^{2zk/n},\quad z\in\mathbb{C},
\end{multline*}
and sending first $n\to\infty$ and then $M\to\infty$, we obtain from Lemma~\ref{lem:Y_n_3_variance} and formula~\eqref{eq:Y2_to_zero_1} that
$$
\lim_{n\to\infty}S_n(2\alpha;\ell; z^2_n(0,0)\ee^{2z/n})=\sum_{k\geq 1}b^2(k)+\ee^{2z}=S(2\alpha)+\ee^{2z}
$$
locally uniformly in $z\in\mathbb{C}$. The locally uniform convergence of analytic functions implies the convergence of their derivatives, whence
$$
\lim_{n\to\infty}\frac{2}{n}S_n(2\alpha+1;\ell; z^2_n(0,0)\ee^{2z/n})=\lim_{n\to\infty}\frac{{\rm d}}{{\rm d}z}S_n(2\alpha;\ell; z^2_n(0,0)\ee^{2z/n})=2\ee^{2z}.
$$
and
$$
\lim_{n\to\infty}\frac{4}{n^2}S_n(2\alpha+2;\ell; z^2_n(0,0)\ee^{2z/n})=\lim_{n\to\infty}\frac{{\rm d}^2}{{\rm d}z^2}S_n(2\alpha;\ell; z^2_n(0,0)\ee^{2z/n})=4\ee^{2z}.
$$
Combining pieces together we conclude that
\begin{align}\label{eq:limit_radial_intensity_crystalline}
\lim_{n\to\infty}\frac{p_n (\alpha;z_n(0,0)\ee^{s/n})}{n^2} =
\frac{1}{4\pi\cosh^2(s-m_{\alpha})}, \quad s\in\mathbb{R},
\end{align}
where $m_{\alpha}=\frac{1}{2}\log S(2\alpha)$. Equation~\eqref{eq:limit_radial_intensity_crystalline} holds whenever $S(2\alpha)<\infty$ which is always true when $\alpha <-\nicefrac{1}{2}$ and is an additional condition on $\ell$ when $\alpha=-\nicefrac{1}{2}$. Interestingly, in the limit the dependence on $b$ is manifested only as a shift in the argument, as evident from~\eqref{eq:limit_radial_intensity_crystalline}.

The above derivation of~\eqref{eq:limit_radial_intensity_crystalline} can be amended to prove that when $\alpha=-\nicefrac{1}{2}$ and $S(-1)=+\infty$,
\begin{align}\label{eq:radial_density_alpha=-1/2}
\lim_{n\to\infty}\frac{p_n (\alpha;\hat{z}_n(0,0)\ee^{s/n})}{n^2} =
\frac{1}{4\pi\cosh^2(s)}, \quad s\in\mathbb{R}.
\end{align}
This follows from~\eqref{eq:critical_variance_asymp_1},~\eqref{eq:critical_variance_asymp_2}, Lemma~\ref{lem:critical_variance_asymp_3} and~\eqref{eq:billingsley_critical} which together imply
that locally uniformly in $z\in\mathbb{C}$ we have
$\lim_{n\to\infty}S_n(-1;\ell; \hat{z}^2_n(0,0)\ee^{2z/n})/L(n)=1+\ee^{2z}$.

The crossover as $\alpha\to -\nicefrac{1}{2}+0$ connecting the regimes~\eqref{eq:RadialIntensity1}  and~\eqref{eq:radial_density_alpha=-1/2} follows from~\eqref{eq:crossover_covariance}. Recall that $\tau=2\alpha+1$. Taking the second logarithmic derivative with respect to $u$ in~\eqref{eq:crossover_covariance}, which is possible due to the locally uniform convergence and analytitcity, we obtain
\begin{multline}\label{eq:intensity_crossover_derivation}
\lim_{\tau\to 0+}\frac{1}{\pi} \left(
\frac{\Phi_{2\alpha+2}(\log \tau^{-1}+\log\log \tau^{-1}+2s)}{\Phi_{2\alpha}(\log \tau^{-1}+\log\log \tau^{-1}+2s)}-\frac{\Phi^2_{2\alpha+1}(\log \tau^{-1}+\log\log \tau^{-1}+2s)}{\Phi^2_{2\alpha}(\log \tau^{-1}+\log\log \tau^{-1}+2s)}\right)\\
=\frac{1}{4\pi}\frac{{\rm d}}{{\rm d}s^2}\log(1+\ee^{2s})=\frac{1}{4\pi}\frac{1}{\cosh^2(s)},\quad s\in\mathbb{R}.
\end{multline}
Thus, we see that after the appropriate scaling the right-hand side of~\eqref{eq:RadialIntensity1} converges as $\alpha\to -\nicefrac{1}{2}+0$ to the right-hand side of~\eqref{eq:radial_density_alpha=-1/2}. This implies convergence of the first intensities in the case when real and imaginary parts of $\xi$ have equal variances. Then the marginal distributions of $G_{\psi}$ do not depend on $\psi$ and are isotropic Gaussian. By making use of the Edelman-Kostlan formula and according to formulae~\eqref{eq:rho},~\eqref{eq:Phi_def} and~\eqref{eq:intensity_crossover_derivation}, if $\rho_1 (\alpha, \cdot )$ denotes the intensity of zeros of $G_{0}$, then, for every $u\in \mathbb{C}$,
\begin{equation}\label{eq:convergence_of_intensities_right}
\lim_{\alpha \to -\nicefrac{1}{2} +0} \rho_1 \left(\alpha, u + \frac{1}{2}\log \big(1/(2\alpha+1)\big) + \frac{1}{2}\log \log \big(1/(2\alpha+1)\big) \right) = 1/(4\pi \cosh^2 (\Re\, u)).
\end{equation}
The function on the right-hand-side is the intensity of the zero set of $u\mapsto \hat N_1+\ee^{u}N_1$.

The crossover between strong and weak crystalline phases is given by the next
\begin{prop}\label{prop:left-hand-side-crossover}
Fix $\psi\in [0,2\pi)$ and assume that the sequence $b_m(x)=x^{\alpha_m}\ell_m(x)$, $m\geq 1$,  is chosen such that~\eqref{eq:b_m_crossover_strong_weak} holds true. Consider the sequence of random variables
$$
P_{m,\infty}(\ee^{\ii \psi}) := \sum_{k\ge 0} b_m(k) \xi_k \ee^{\ii k \psi}, \quad m\ge 1,
$$
and, furthermore, assume that
\begin{equation}\label{eq:left_crossover_assump}
\lim_{m\to\infty}\frac{\sup_{k\geq 0}b_m^2(k)}{S_m(2\alpha_m;\ell_m)}=\lim_{m\to\infty}\frac{\sup_{k\geq 0}b_m^2(k)}{\sum_{k\geq 0}b_m^2(k)}=0.
\end{equation}
Finally, if $\psi\notin\{0,\pi\}$ and $\sigma_1^2\neq \sigma_2^2$, assume that $x\mapsto b_m(x)$ is unimodal on $[A,+\infty)$, for every fixed $m\geq 1$ and some $A\geq 0$. Then
\begin{equation}\label{eq:left_crossover_claim1}
\frac{P_{m,\infty}(\ee^{\ii\psi})}{\sqrt{S_m(2\alpha_m;\ell_m)}}=\frac{\sum_{k\geq 0}b_m(k)\xi_k\ee^{\ii\psi k}}{\sqrt{S_m(2\alpha_m;\ell_m)}}~\todistr~\hat{N}_1,\quad m\to\infty.
\end{equation}
Therefore,
$$
\left(\frac{P_{m,\infty}(\ee^{\ii \psi})+\ee^{u+\nicefrac{1}{2}\log S_m(2\alpha_m;\ell_m)}N_1}{\sqrt{S_m(2\alpha_m;\ell_m)}}\right)_{u\in\mathbb{C}}~\Longrightarrow~\left(\hat{N}_1+\ee^{u}N_1\right)_{u\in\mathbb{C}},\quad m\to\infty,
$$
on the space $\mathcal{A}(\mathbb{C})$ and, therefore,
\begin{equation}\label{eq:left_crossover_claim2}
\zeros_{\mathbb{C}}(P_{m,\infty}(\ee^{\ii \psi})+\ee^{(\cdot)+\nicefrac{1}{2}\log S_m(2\alpha_m;\ell_m)}N_1)~\Longrightarrow~\zeros_{\mathbb{C}}(\hat{N}_1+\ee^{(\cdot)}N_1),\quad m\to\infty,
\end{equation}
on the space $M_p(\mathbb{C})$. Here, as before, $N_1$ and $\hat{N}_1$ are independent Gaussian random variables with covariance given by~\eqref{eq:limit_gaussian_cov_alt}.
\end{prop}
\begin{proof}
We only need to prove~\eqref{eq:left_crossover_claim1}. Observe that $P_{m,\infty}(\ee^{\ii\psi})$ is an infinite sum of independent random variables with finite second moments. Thus, the Lindeberg-Feller theorem applies and amounts to checking that
$$
\lim_{m\to\infty}\mathbb{E}\left[\frac{P_{m,\infty}(\ee^{\ii \psi})\overline{P_{m,\infty}(\ee^{\ii \psi})}}{S_m(2\alpha_m;\ell_m)}\right]=\sigma^2,\quad\lim_{m\to\infty}\mathbb{E}\left[\frac{P^2_{m,\infty}(\ee^{\ii \psi})}{S_m(2\alpha_m;\ell_m)}\right]=(\sigma_1^2-\sigma_2^2)\1_{\{\psi\in \{0,\pi\}\}},
$$
and the Lindeberg-Feller condition. The first relation in the last display is trivial. The second is trivial if $\sigma_1^2=\sigma_2^2$ and, otherwise, is equivalent to
\begin{equation}\label{eq:left_crossover_covariances_claim}
\lim_{m\to\infty}\frac{1}{S_m(2\alpha_m;\ell_m)}\sum_{k\geq 0}b_m^2(k)\ee^{2\ii \psi k}
=\begin{cases}
1,&\psi\in \{0,\pi\},\\
0,&\psi\in (0,\pi)\cup (\pi,2\pi).
\end{cases}
\end{equation}
If $\psi\not\in \{0,\pi\}$, this can be checked using unimodality. Indeed,
$$
\sum_{k\geq 0}b_m^2(k)\ee^{2\ii \psi k}=b_m^2(0)+\sum_{k\geq 1}(b_m^2(k)-b_m^2(k+1))\sum_{j=1}^{k}\ee^{2\ii j\psi}.
$$
The absolute value of the inner sum is bounded by an absolute constant $c(2\psi)>0$, whence
$$
\left|\sum_{k\geq 0}b_m^2(k)\ee^{2\ii \psi k}\right|\leq b_m^2(0)+c(2\psi)\sum_{k\geq 1}|b_m^2(k)-b_m^2(k+1)|\leq (1+Bc(2\psi)) \sup_{k\geq 0}b_m^2(k),
$$
with the last inequality holding for some $B>0$ and being a consequence of the unimodality.

Finally, the Lindeberg-Feller condition reads
$$
\lim_{m\to\infty}\frac{1}{S_m(2\alpha_m;\ell_m)}\sum_{k\geq 0}b_m^2(k)\mathbb{E}[|\xi|^2\1_{\{b_m(k)|\xi|\geq \delta S_m^{1/2}(2\alpha_m;\ell_m)\}}]=0,\quad \delta>0,
$$
and follows from
$$
\lim_{m\to\infty}\sup_{k\geq 0}\mathbb{E}[|\xi|^2\1_{\{b_m(k)|\xi|\geq \delta S_m^{1/2}(2\alpha_m;\ell_m)\}}]\leq \lim_{m\to\infty}\mathbb{E}[|\xi|^2\1_{\{\sup_{k\geq 0}b_m(k)|\xi|\geq \delta S_m^{1/2}(2\alpha_m;\ell_m)\}}]=0,
$$
which is a consequence of~\eqref{eq:left_crossover_assump}.
\end{proof}

\begin{rem}
Condition~\eqref{eq:left_crossover_assump} ensures that the individual summands are asymptotically negligible and is necessary for the central limit theorem in~\eqref{eq:left_crossover_claim1}. Although the unimodality assumption could be weakened to any condition that guarantees~\eqref{eq:left_crossover_covariances_claim}, identifying the optimal set of assumptions lies beyond the scope of this paper. If $\ell_m=\ell$, for all $m\geq 1$ and $\ell$ such that $S(-1;\ell)=+\infty$, then condition~\eqref{eq:b_m_crossover_strong_weak} is equivalent to $\alpha_m\to -\nicefrac{1}{2}-0$ and~\eqref{eq:left_crossover_assump} is always true since the numerator is uniformly bounded in $m$.
\end{rem}

\section{Self-inversive polynomials}\label{sec:si_polys}
Motivated by the results on self-inversive polynomials obtained in~\cite{Bogomolny+Bohigas+Leboeuf:1992} in this section we apply our findings to the analysis of the self-inversive polynomials $K_m$ defined by~\eqref{eq:K_m_polys_def}. As $m\to\infty$, the zeros of $K_m$ in the unit disk converge to  a nontrivial point process. This is a straightforward corollary of Proposition~\ref{prop:1}.
\begin{prop}[Convergence inside the unit disk: zeros of $K_m(z)$ ]
Assume~\eqref{eq:xi_log_moment}. Then
$$
\zeros_{\mathbb{D}_1}(K_m(z))~\Longrightarrow~\zeros_{\mathbb{D}_1}\big(1+\sum\limits_{k\ge 1} b(k)\xi_kz^k \big),\quad m\to\infty,
$$
on the space $\mathcal{M}_p(\mathbb{D}_1)$.
\end{prop}
Recalling that the zeros of $K_m$ occur in pairs $(z,1/\overline{z})$, the above proposition also yields a limit theorem for zeros outside $\mathbb{D}_1$. As before, we now turn to the analysis of zeros near the unit circle.

\subsection{Case \texorpdfstring{$\alpha>-\nicefrac{1}{2}$}{alpha>-1/2}.}
Theorem~\ref{thm:boundary:alpha>-1/2} implies that under the assumptions~\eqref{eq:moment_assumptions_on_xi1},~\eqref{eq:moment_assumptions_on_xi2} and $\alpha>-\nicefrac{1}{2}$,
\begin{equation}\label{eq:si_poly0}
\left(\frac{K_m(\ee^{u/m})}{b(m)\sqrt{m}}\right)_{u\in\mathbb{C}}~\Longrightarrow~\left(\mathfrak{G}(u)\right)_{u\in\mathbb{C}},\quad m\to\infty,
\end{equation}
on the space $\mathcal{A}(\mathbb{C})$, where
$$
\mathfrak{G}(u):=G_0(u)+\ee^{2u}\overline{G_0(-\overline{u})},\quad u\in\mathbb{C}.
$$
Hurwitz's theorem, of course, implies convergence of the point process of complex zeros of $K_m(\ee^{(\cdot)/m})$ to the point process of zeros of $\mathfrak{G}$. However, more can be said about zeros of $K_m$ that lie {\bf exactly} on the unit circle $\partial \mathbb{D}_1$. Put $\widetilde{K}_m(t):=\ee^{-\ii t(m+\nicefrac{1}{2})}K_m(\ee^{\ii t})$ and observe that
\begin{equation}\label{eq:si_poly1}
\overline{\widetilde{K}_m(t)}=\widetilde{K}_m(t),\quad t\in\mathbb{R}.
\end{equation}
Let $\mathcal{A}_{\mathbb{R}}(\mathbb{C})$ be a subspace of $\mathcal{A}(\mathbb{C})$ consisting of all functions in $\mathcal{A}(\mathbb{C})$ which take real values on $\mathbb{R}$. Equality~\eqref{eq:si_poly1} implies that $\mathbb{P}\{\widetilde{K}_m\in \mathcal{A}_{\mathbb{R}}(\mathbb{C})\}=1$. Since $\mathcal{A}_{\mathbb{R}}(\mathbb{C})$ is a closed subspace of $\mathcal{A}(\mathbb{C})$,~\eqref{eq:si_poly0} yields
\begin{equation}\label{eq:si_poly2}
\left(\frac{\widetilde{K}_m(t/m)}{b(m)\sqrt{m}}\right)_{t\in\mathbb{C}}~\Longrightarrow~\left(\ee^{-\ii t}\mathfrak{G}(\ii t)\right)_{t\in\mathbb{C}},\quad m\to\infty,
\end{equation}
on the space $\mathcal{A}_{\mathbb{R}}(\mathbb{C})$. The analytic process $\widetilde{\mathfrak{G}}(t):=\ee^{-\ii t}\mathfrak{G}(\ii t)$ satisfies
$$
\widetilde{\mathfrak{G}}(t)=\ee^{-\ii t}G_0(\ii t)+\overline{\ee^{-\ii t}G_0(\ii t)}=2\Re(\ee^{-\ii t}G_0(\ii t)),\quad t\in\mathbb{R}.
$$
Using the integral representation in Remark~\ref{rem:integral_representation} we can also write
\begin{multline}\label{eq:s_tilde_integral_rep}
\widetilde{\mathfrak{G}}(t)=2\Re\left(\int_0^1 x^{\alpha}\ee^{\ii t(x-1)}{\rm d}B(x)\right)\\
=2\sigma_1\int_0^{1}x^{\alpha}\cos(t(x-1)){\rm d}B_1(x)-2\sigma_2\int_0^{1}x^{\alpha}\sin(t(x-1)){\rm d}B_2(x),\quad t\in\mathbb{R},
\end{multline}
where $B_1$ and $B_2$ are independent real standard Brownian motions. In Lemma~\ref{lem:variance} in the Appendix we check that with probability one the process $(\widetilde{\mathfrak{G}}(t))_{t\in\mathbb{R}}$ does not have multiple zeros on the real line. Thus, Lemma 4.2 in~\cite{Iksanov+Kabluchko+Marynych:2016} together with~\eqref{eq:si_poly2} imply the following
\begin{prop}
Assume~\eqref{eq:moment_assumptions_on_xi1},~\eqref{eq:moment_assumptions_on_xi2} and $\alpha>-\nicefrac{1}{2}$. Then
\begin{equation}\label{eq:conv_to_s_tilde}
\zeros_{\mathbb{R}}(K_m(\ee^{\ii (\cdot)/m}))~\Longrightarrow~\zeros_{\mathbb{R}}(\widetilde{\mathfrak{G}}(\cdot)),\quad m\to\infty.
\end{equation}
on the space $M_p(\mathbb{R})$, where $\widetilde{\mathfrak{G}}$ is defined by~\eqref{eq:s_tilde_integral_rep}.
\end{prop}
\begin{rem}
If $\alpha=0$, then~\eqref{eq:s_tilde_integral_rep} can be simplified to
\begin{equation}\label{eq:s_tilde_integral_rep_alpha=0}
\widetilde{\mathfrak{G}}(t)\od 2\Re\left(\int_0^1 \ee^{\ii tx}{\rm d}B(x)\right)\od 2\sigma_1\int_0^{1}\cos(tx){\rm d}B_1(x)+2\sigma_2\int_0^{1}\sin(tx){\rm d}B_2(x),\quad t\in\mathbb{R},
\end{equation}
using the time-reversal property of the Brownian motion: $(B(x))_{x\in [0,1]}\od (B(1)-B(1-x))_{x\in [0,1]}$. In this case the convergence~\eqref{eq:conv_to_s_tilde} has been proved in~\cite[Theorem 2.3]{Iksanov+Kabluchko+Marynych:2016}, see also Section 3.1.2 of that paper. Further results on the roots of random trigonometric polynomials can be found in the recent works~\cite{Angst+Pautrel+Poly:2022,Do+Nguyen+Nguyen:2022}, along with references to earlier research.
\end{rem}

\subsection{Case \texorpdfstring{$\alpha\leq-\nicefrac{1}{2}$}{alpha<=-1/2} and \texorpdfstring{$S(2\alpha;\ell)<+\infty$}{S(2alpha;l)<+infty}.}
Using the assumption $S(2\alpha)<+\infty$ it is easy to see that
\begin{equation}\label{eq:si_b^2_k_summable}
(P_m(\ee^{u/m}))_{u\in\mathbb{C}}~\Longrightarrow~P_{\infty}(1):=\sum_{k\geq 1}b(k)\xi_k,\quad m\to\infty,
\end{equation}
on the space $\mathcal{A}(\mathbb{C})$. Indeed, for every compact set $K\subset\mathbb{C}$,
$$
\sup_{u\in K}\left|\sum_{k=1}^{n}b(k)\xi_k(\ee^{uk/m}-1)\right|\leq \sum_{k=1}^{m}b(k)|\xi_k|\sup_{u\in K}\left|\ee^{uk/m}-1\right|.
$$
The right-hand side converges a.s.~to zero, as $m\to\infty$, by the dominated convergence theorem. Formula~\eqref{eq:si_b^2_k_summable} implies that
\begin{equation}\label{eq:si_b^2_k_summable_flt}
(\widetilde{K}_m(t/m))_{t\in\mathbb{C}}~\Longrightarrow~(\ee^{-\ii t}+\ee^{-\ii t}P_{\infty}(1)+\ee^{\ii t}\overline{P_{\infty}(1)}+\ee^{\ii t})_{t\in\mathbb{C}},\quad m\to\infty
\end{equation}
on the space $\mathcal{A}_{\mathbb{R}}(\mathbb{C})$. For $t\in\mathbb{R}$, the right-hand can be written as
$$
\ee^{-\ii t}+\ee^{-\ii t}P_{\infty}(1)+\ee^{\ii t}\overline{P_{\infty}(1)}+\ee^{\ii t}=2\cos (t)+2\Re(\ee^{-\ii t}P_{\infty}(1))=(2+2\Re P_{\infty}(1))\cos t+(2\Im P_{\infty}(1))\sin t.
$$
Therefore, Lemma 4.2 in~\cite{Iksanov+Kabluchko+Marynych:2016} implies that the point process  $\zeros_{\mathbb{R}}(K_m(\ee^{\ii (\cdot)/m}))$ of the real zeros of $K_m(\ee^{\ii (\cdot)/m})$ converges in distribution to a lattice defined by the set of solutions to
$$
\tan t=-\frac{1+\Re P_{\infty}(1)}{\Im P_{\infty}(1)},
$$
which coincides with the set of real zeros of $t\mapsto (2+2\Re P_{\infty}(1))\cos t+(2\Im P_{\infty}(1))\sin t$.
\subsection{Case \texorpdfstring{$\alpha=-\nicefrac{1}{2}$}{alpha=-1/2} and \texorpdfstring{$S(2\alpha;\ell)=+\infty$}{S(2alpha,l)=+infty}.} Here, we only state the result and outline the proof, leaving the details to the reader. By showing that the variance tends to zero, it can be checked that, for every compact set $K\subset\mathbb{C}$,
$$
\sup_{u\in K}\frac{\left|\sum_{k=1}^{m}b(k)\xi_k(\ee^{uk/m}-1)\right|}{\sqrt{L(n)}}~\overset{\mathbb{P}}{\to}~0,\quad m\to\infty,
$$
whence
$$
\left(\frac{P_m(\ee^{u/m})}{\sqrt{L(n)}}\right)_{u\in\mathbb{C}}~\Longrightarrow~\mathcal{N},\quad m\to\infty,
$$
where $\mathcal{N}$ is the complex Gaussian random variable with the same covariance structure as $\xi$. Thus,
$$
\left(\frac{\widetilde{K}_m(t/m)}{\sqrt{L(m)}}\right)_{t\in\mathbb{C}}~\Longrightarrow~\left(\ee^{-\ii t}\mathcal{N}+\ee^{\ii t}\overline{\mathcal{N}}\right)_{t\in\mathbb{C}},\quad m\to\infty,
$$
on the space $\mathcal{A}_{\mathbb{R}}(\mathbb{C})$. Thus, in the regime in focus, $\zeros_{\mathbb{R}}(K_m(\ee^{\ii (\cdot)/m})$ also converge to a lattice defined by the set of real solutions to
$$
\tan t=-\frac{\Re \mathcal{N}}{\Im \mathcal{N}}.
$$

\subsection{Isotropic Gaussian coefficients}
\label{section:5}
Assume that the coefficients $(\xi_k)_{0\leq k\leq m}$ are independent identically distributed complex Gaussian variables with isotropic distribution and the variance $\sigma^2$. Put
\begin{equation}\label{eq:g_1_2_def}
g_1(m):=\sigma^2\sum_{k=1}^{m}b^2(k),\quad g_2(m):=\sigma^2\sum_{k=1}^{m}(m+\nicefrac{1}{2}-k)^2 b^2(k).
\end{equation}
Let $\nu_m$ be the number of zeros of $K_m$ defined by~\eqref{eq:K_m_polys_def} that lie on the unit circle. According to \cite[Eq.~(A.20)]{Bogomolny+Bohigas+Leboeuf:1996}\footnote{There is a typo in the cited formula (A.20): the very last term $N$ should be replaced by $N^2$ (in the notation of the cited paper).}
\begin{multline}\label{eq:Bogo}
\mathbb{E}[\nu_m/(2m+1)]=\frac{1}{\sqrt{2\pi}}
\int_{-(g_1(m))^{-\nicefrac{1}{2}}}^{(g_1(m))^{-\nicefrac{1}{2}}}\ee^{-y^2/2}{\rm d}y\\
+\frac{1}{\pi(m+\nicefrac{1}{2})}\left(\frac{g_2(m)}{g_1(m)}\right)^{\nicefrac{1}{2}}\int_0^1\int_0^{\pi} \ee^{-\left(\frac{1}{2g_1(m)}\cos^2\phi+\frac{(m+\nicefrac{1}{2})^2}{2g_2(m)}\frac{\sin^2\phi}{x^2}\right)}{\rm d}\, \phi{\rm d}x.
\end{multline}
This identity holds for arbitrary polynomial degree and arbitrary variance profile of the coefficients. Passing in this identity to the limit $m\to\infty$ we find that the average proportion of zeros of $K_m$ on the unit circle is asymptotically non-zero.

\begin{theorem}
Assume that $b(x)=x^{\alpha}\ell(x)$ and that the function $\ell$ is slowly varying at $+\infty$.
\begin{itemize}
\item[(i)] If $\alpha\le -\nicefrac{1}{2}$, then
\begin{equation*}
\lim_{m\to\infty}\mathbb{E}[\nu_m/(2m+1)] = 1.
\end{equation*}
\item[(ii)] If $\alpha > -\nicefrac{1}{2}$, then
\begin{equation*}
\lim_{m\to\infty}\mathbb{E}[\nu_m/(2m+1)] = \frac{1}{\sqrt{(1+\alpha)(3+2\alpha)}} <1.
\end{equation*}

\end{itemize}
\end{theorem}
\begin{proof}

Under the assumption $b(x)=x^{\alpha}\ell(x)$ the asymptotics of the term $\frac{1}{m}\left(\frac{g_2(m)}{g_1(m)}\right)^{\nicefrac{1}{2}}$ is calculated in Lemma~\ref{lem:si_polys_mean_asymp} in the Appendix. In particular, if $\alpha <  -\nicefrac{1}{2}$, or if $\alpha =  -\nicefrac{1}{2}$ and
\begin{align*}
S(2\alpha)=S(2\alpha;\ell)=\sum_{k\geq 1}k^{2\alpha}\ell^2(k)=\sum_{k\geq 1}b^2(k)=\sigma^{-2}\lim_{m\to\infty}g_1(m)<\infty,
\end{align*}
the aforementioned lemma entails
\begin{align}\label{eq:g2tog1}
\lim_{m\to\infty}\frac{1}{m}\left(\frac{g_2(m)}{g_1(m)}\right)^{\nicefrac{1}{2}}=1,
\end{align}
and
\begin{align*}
\lim_{m\to\infty}\frac{g_2(m)}{m^2}=S(2\alpha).
\end{align*}
Therefore, in this case
\begin{align}\label{eq:e_nu_m_1}
\lim_{m\to\infty}\mathbb{E}[\nu_m/(2m+1)] &= 
\erf \left(1/\sqrt{2S(2\alpha)} \right)  +\frac{1}{\pi}\int_0^1\int_0^{\pi}\ee^{-\frac{\cos^2 \phi}{2S(2\alpha)}-\frac{\sin^2 \phi}{2x^2S(2\alpha)}}{\rm d}\phi{\rm d}x.
\end{align}
The repeated integral above is equal to the complementary error function $\erfc\left({1}/{\sqrt{2S(2\alpha)}}\right) $, see Lemma~\ref{lem:integrals} in the Appendix, so that the two terms on the right add up to 1. If $\alpha=-\nicefrac{1}{2}$ and $\lim_{m\to\infty}g_1(m)=\infty$, equation~\eqref{eq:g2tog1} still holds but now $\lim_{m\to\infty}\frac{g_2(m)}{m^2}=+\infty$. Passing to the limit $m\to\infty$ in \eqref{eq:Bogo}, one obtains $\lim_{m\to\infty}\mathbb{E}[\nu_m/(2m+1)] = 1$. This proves part (i).

If $\alpha>-\nicefrac{1}{2}$, then necessarily $\lim_{m\to\infty}g_1(m)=\infty$ and also
\begin{equation}\label{eq:g_2_grows_fast}
\frac{g_2(m)}{m^2}\geq \frac{\sigma^2}{m^2}\sum_{1\leq k\leq m/2}\left(\frac{2m+1}{2}-k\right)^2 b^2(k)\geq \frac{\sigma^2}{m^2} \left(\frac{m+1}{2}\right)^2\sum_{1\leq k\leq m/2}b^2(k)\to+\infty,\quad m\to\infty.
\end{equation}
Lemma~\ref{lem:si_polys_mean_asymp}(i) yields
\begin{equation}\label{eq:e_nu_m_2}
\lim_{m\to\infty}\mathbb{E}[\nu_m/(2m+1)] = \frac{1}{\sqrt{(1+\alpha)(3+2\alpha)}}.
\end{equation}
This proves part (ii).
\end{proof}

\smallskip

The limit of $\mathbb{E}[\nu_m/(2m+1)]$ as $m\to\infty$ was previously found in~\cite{Bogomolny+Bohigas+Leboeuf:1996} for the particular case when $b(x)=x^{\alpha}$ and $\alpha \geq -\nicefrac{1}{2}$ along with the finite-$m$ corrections to the limit in the case when $b(x)=x^{-\nicefrac{1}{2}}$ which turned out to be of order $1/\log m$. To obtain the speed of convergence of $\mathbb{E}[\nu_m/(2m+1)]$ to the limit in the general case of Gaussian coefficients with regularly varying variance profile, it is useful to recast identity~\eqref{eq:Bogo} in a slightly different form. By introducing
\begin{align}\label{eq:u_mv_m_def}
u_m=\frac{1}{\sqrt{2g_1(m)}}, \quad v_m=\frac{m+\nicefrac{1}{2}}{\sqrt{2g_2(m)}},
\end{align}
and making use of the integrals
\begin{align}\label{eq:aux_int3}
\int_0^1 \ee^{ - \frac{1}{x^2} v_m^2\sin^2\phi}{\rm d}x =
\ee^{-v_m^2 \sin^2 \phi }-\sqrt{\pi } v_m \erfc (v_m \sin \phi ) \sin \phi
\end{align}
and
\begin{align}\label{eq:aux_int2}
\frac{1}{\pi} \int_0^{\pi}  \ee^{-u_m^2\cos^2\phi-v_m^2\sin^2 \phi} {\rm d}\phi = \ee^{-\frac{1}{2} (u_m^2+v_m^2)} I_0((u_m^2-v_m^2)/2 ),
\end{align}\label{eq:aux_int1}
identity \eqref{eq:Bogo} transforms to
\begin{multline}\label{eq:Bogomo}
\mathbb{E}[\nu_m/(2m+1)]=\erf (u_m) + \frac{u_m}{v_m}\, \ee^{-\frac{1}{2} (u_m^2+v_m^2)} I_0 ((u_m^2-v_m^2)/2 ) \\
- \frac{u_m}{\sqrt{\pi}} \int_0^{\pi} \ee^{-u_m^2\cos^2\phi} \erfc \big(v_m \sin \phi \big) \sin \phi \, \d \phi.
\end{multline}
The integral in \eqref{eq:aux_int3} is evaluated in Lemma \ref{lem:integrals} and the integral in \eqref{eq:aux_int2} is derived from the standard integral representation for the modified Bessel function $I_0$ of zero order using also that this function is even.

Identity \eqref{eq:Bogomo} is convenient for asymptotic analysis. We restrict ourselves to the most interesting case when $\alpha \le -\nicefrac{1}{2}$. In this case $v_m\sim u_m$ in the limit $m\to\infty$, and to leading order the integral in \eqref{eq:Bogomo} is given by (see Lemma \ref{lem:integrals})
\begin{align}\label{eq:integral_evaluation}
\frac{u_m}{\sqrt{\pi}} \int_0^{\pi} \ee^{-u_m^2\cos^2\phi} \erfc \big(u_m \sin \phi \big) \sin \phi \, \d \phi = e^{-u_m^2}-\erfc(u_m),
\end{align}
so that the tree terms on the right in  \eqref{eq:Bogomo} to leading order add up to 1.

To obtain sub-leading orders, observe that if $\alpha \le -\nicefrac{1}{2}$, then Lemma~\ref{lem:si_polys_mean_asymp} and elementary manipulations yield
\begin{align}\label{eq:u/v}
{v_m}/{u_m}=1+\epsilon_{m,\alpha} +o(\epsilon_{m,\alpha}), \quad m\to\infty,
\end{align}
where
\begin{align} \label{eq:eps_m}
\epsilon_{m,\alpha} = \begin{cases}
\,\,\displaystyle{\frac{3}{4}\frac{\ell^2(m)}{\sum_{k=1}^{m}\frac{\ell^2(k)}{k}}} ,
& \quad \text{if } \alpha =-\nicefrac{1}{2};\\[3ex]
\,\, \displaystyle{
\frac{(2+\alpha)}{2(1+\alpha)(3+2\alpha)S(2\alpha)}\,m^{1+2\alpha}\ell^2(m)},
& \quad \text{if } -1< \alpha < -\nicefrac{1}{2};\\[2ex]
\,\,\displaystyle{\frac{1}{S(2\alpha)} \frac{\sum_{k=1}^{m}k^{1+2\alpha}\, \ell^2 (k) }{m} } ,
&\quad \text{if } \alpha\leq -1 .
\end{cases}
\end{align}
Obviously, $\epsilon_{m,\alpha}$ is positive and $\epsilon_{m,\alpha} \to 0$, as $m\to\infty$.
\smallskip

To proceed with the calculation of $\mathbb{E}[\nu_m/(2m+1)]$, we expand  $\erfc(v_m  \sin \phi)= \erfc((1+\delta_m) \,u_m \sin \phi)$  in powers of $\delta_m=\epsilon_{m,\alpha}(1+o(1))$,
\begin{multline}\label{eq:T_expansion}
\erfc(u_m(1+\delta_m) \sin \phi)= \\
\erfc(u_m \sin \phi )- \frac{2}{\sqrt{\pi}}\,  (u_m \sin \phi)\, \ee^{-(u_m \sin  \phi )^2}\,  \delta_m
+ R_1(u_m \sin \phi, \delta_m )\delta_m^2 .
\end{multline}
Since the second derivative of $\erfc(c u)$ in $u$ is $\pi^{-\nicefrac{1}{2}} 4c^3 u \ee^{-c^2 u^2}$ and $\delta_m\in [0,1]$ for all $m$ sufficiently large, it holds that
\begin{align*}
0\le R_1(u_m \sin \phi, \delta_m ) \le 4\pi^{-\nicefrac{1}{2}} (u_m \sin \phi)^3  \ee^{-u_m^2 \sin^2 \phi}\quad \text{for  $\phi\in [0,\pi]$}.
\end{align*}
Consequently, using that $u_m$ is bounded,
\begin{align*}
0\le \frac{u_m}{\sqrt{\pi}} \int_{0}^{\pi} R_1(u_m \sin \phi, \delta_m )\,   \ee^{-u_m^2\cos^2\phi}\,  \sin \phi \, \d \phi \le \frac{3}{2 } u_m^4 \ee^{-u_m^2} \le C u_m^2 \ee^{-u_m^2},
\end{align*}
for some $C>0$. Hence
\begin{multline}
\frac{u_m}{\sqrt{\pi}} \int_0^{\pi} \ee^{-u_m^2\cos^2\phi} \erfc \big(v_m \sin \phi \big) \sin \phi \, \d \phi = \\
\erf(u_m)+\ee^{-u_m^2}-1 - u_m^2\ee^{-u_m^2}\delta_m + O(u_m^2\ee^{-u_m^2}\delta_m^2)=\erf(u_m)+\ee^{-u_m^2}-1 - u_m^2\ee^{-u_m^2}\epsilon_{m,\alpha}(1+o(1)).
\end{multline}
It now remains to expand in powers of $\delta_m$ the second term on the right-hand side in~\eqref{eq:Bogomo}. Using $I_0(x)=1+o(x)$, as $x\to 0$, we infer
\begin{align*}
 \frac{u_m}{v_m}\, \ee^{-\frac{1}{2} (u_m^2+v_m^2)} I_0 ((u_m^2-v_m^2)/2 )=\ee^{-u_m^2}-\ee^{-u_m^2} (u_m^2+1) \epsilon_{m,\alpha} (1+o(1)).
\end{align*}
 Thus, finally, for $\alpha \le -\nicefrac{1}{2}$ we obtain
\begin{align}\label{eq:a_frac}
\mathbb{E}[\nu_m/(2m+1)]=
1-\ee^{-u_m^2} \epsilon_m (1+o(1)) = 1-\ee^{ -\frac{1}{2g_1(m)}} \epsilon_{m,\alpha}  (1+o(1)).
\end{align}
The proposition below summarises our findings.
\begin{prop}
Assume that $b(x)=x^{\alpha}\ell(x)$ and that the function $\ell$ is slowly varying at $+\infty$.
\begin{itemize}
\item[(i)] If $\alpha=-\nicefrac{1}{2}$, then, as $m\to\infty$,
\begin{equation}\label{eq:alpha=1/2}
1-\mathbb{E}[\nu_m/(2m+1)]~\sim~\frac{3 \ee^{-1/(2g_1(m))}}{4}\frac{ \ell^2(m)}{\sum_{k=1}^{m}k^{-1}\ell^2(k)}.
\end{equation}
\item[(ii)] If $-1<\alpha < -\nicefrac{1}{2}$, then
\begin{equation}
1-\mathbb{E}[\nu_m/(2m+1)]~\sim~\frac{\ee^{-1/(2\sigma^2 S(2\alpha))}(2+\alpha)}{2(1+\alpha)(3+2\alpha)S(2\alpha)}\,m^{1+2\alpha}\ell^2(m).
\end{equation}
\item[(iii)] If $\alpha \le  -1$,  then
\begin{equation}\label{eq:alpha<=-1}
1-\mathbb{E}[\nu_m/(2m+1)]~\sim~\frac{\ee^{-1/(2\sigma^2 S(2\alpha))}}{S(2\alpha)} \frac{\sum_{k=1}^{m}k^{1+2\alpha}\ell^2(k) }{m}.
\end{equation}
\end{itemize}
\end{prop}
\begin{proof}
The assertion of the Proposition follows from~\eqref{eq:a_frac}.
\end{proof}

\begin{ex}
For $\alpha=-\nicefrac{1}{2}$ and $\ell^2(x)=\log^{\beta} x$, $\beta>-1$, we obtain
$$
\mathbb{E}[\nu_m/(2m+1)] =
1- \frac{3(\beta+1)}{4\log m}\, (1+ o( 1) ).
$$
For $\beta=0$ this result has been derived in~\cite{Bogomolny+Bohigas+Leboeuf:1996}, see Eq.~(3.24) therein. If $\alpha=-\nicefrac{1}{2}$ and $\beta<-1$, then
$$
\mathbb{E}[\nu_m/(2m+1)] =
1- \frac{3\ee^{-1/(2\sigma^2 S(-1))}}{4S(-1)}(\log^\beta m)(1+o(1)).
$$
Finally, if $\alpha=-\nicefrac{1}{2}$ and $\beta=-1$, then
$$
\mathbb{E}[\nu_m/(2m+1)] =
1- \frac{3}{4\log m} \frac{1}{ \log \log m}(1+o(1)).
$$
\end{ex}

\begin{rem}
If $\alpha=-\nicefrac{1}{2}$ and $S(-1)=+\infty$, then the exponential in \eqref{eq:alpha=1/2} can be replaced by $1$. In this case the rate of convergence $\mathbb{E}[\nu_m/(2m+1)]$ to $1$ does not depend on $\sigma^2=\mathbb{E}[|\xi|^2]$. Depending on the choice of the slowly varying function the rate can be very slow. For example, if $b(x)= x^{-1/2} \ee^{\log^{\gamma} x}$ with $\gamma\in (0,1)$, then
\begin{equation*}
1-\mathbb{E}[\nu_m/(2m+1)]~\sim~\frac{3}{4}\frac{ 2\gamma}{\log^{1-\gamma} m}, \quad m\to \infty.
\end{equation*}
At the other critical point, if $S(1+2\alpha)=\sum_{k=1}^{\infty}k^{1+2\alpha}\ell^2(k)<\infty$, this sum on the right in~\eqref{eq:alpha<=-1} can be replaced by $S(1+2\alpha)$. In particular, this is possible whenever $\alpha<-1$ and we find that
\begin{equation*}
1-\mathbb{E}[\nu_m/(2m+1)]~\sim~\frac{\ee^{-{1}/{(2\sigma^2 S(2\alpha))}} \, S(1+2\alpha)}{S(2\alpha)} \frac{1 }{m},\quad m\to\infty.
\end{equation*}
Observe that  the rate of convergence of $\mathbb{E}[\nu_m/(2m+1)]$ to $1$, for fixed $\sigma>0$, cannot be faster than $O(m^{-1})$.

\end{rem}

The average number of zeros of $K_m$ on the arc $A(\phi):=\{z\in\mathbb{D}_1\,:\,0\leq \arg(z)\leq \phi\}$ is given by the function
$$
\phi\mapsto \mathbb{E}[(\zeros_{\mathbb{C}}(K_m))(A(\phi))],\quad \phi\in [0,2\pi].
$$
In case of isotropic Gaussian coefficients the above function is differentiable with respect to $\phi$ and its derivative $p_m(\phi)$, equal to the intensity of the point process of zeros of $K_m$ on the unit circle, was calculated in~\cite[Eq. (A.16)]{Bogomolny+Bohigas+Leboeuf:1996},
\begin{align*}
\frac{p_m(\phi) }{m+\nicefrac{1}{2}}=
\frac{u_m\ee^{-u_m^2\cos^2 ((m+\nicefrac{1}{2}) \phi)}}
{\sqrt{\pi}} \left\{ \left| \sin ((m+\nicefrac{1}{2}) \phi) \right| +\frac{1}{\sqrt{\pi} v_m} \int_0^{1} \ee^{-\frac{1}{x^2} v_m^2\sin^{2}((m+\nicefrac{1}{2}) \phi) }\, \d x\right\},\;\; \phi \in [0,2\pi].
\end{align*}
The integral on the right is evaluated in~\eqref{lem:integrals_eq1}. Using it,
\begin{equation}\label{eq:expected_density}
\frac{p_m(\phi) }{2m+1}=
\frac{1}{2\pi}
\frac{u_m}{v_m}
 \ee^{
 -u_m^2\cos^2\phi_m -v_m^2\sin^2\phi_m
 }
  +
\ee^{-u_m^2\cos^2\phi_m}\frac{u_m}{2\sqrt{\pi}}\left(|\sin\phi_m|\erf ( v_m  |\sin\phi_m|)\right)
\end{equation}
where $\phi_m=(m+\nicefrac{1}{2})\phi$. For fixed $m$,  $p_m(\phi)$ is periodic with period $\pi/(m+\nicefrac{1}{2})$.
\begin{theorem}
Assume that $b(x)=x^{\alpha}\ell(x)$ and that the function $\ell$ is slowly varying at $+\infty$.
\begin{itemize}
\item[(i)] If $\alpha >-\nicefrac{1}{2}$, or $\alpha =-\nicefrac{1}{2}$ and $S(-1)=+\infty$, then
\begin{equation*}
\lim_{m\to\infty} \frac{p_m(\phi) }{(2m+1)} = \frac{1}{2\pi} \frac{1}{\sqrt{(1+\alpha)(3+2\alpha)}} , \quad \text{for every } \phi \in[0,2\pi];
\end{equation*}
\item[(ii)] If $\alpha < -\nicefrac{1}{2}$, or $\alpha = -\nicefrac{1}{2}$ and $S(-1) < +\infty$,
then the intensities $p_m(\phi)/(2m+1)$ do not converge but
\begin{equation}\label{eq:expected_counting_measure_to_lebesgue}
\lim_{m\to\infty}\frac{\mathbb{E}[(\zeros_{\mathbb{C}}(K_m))(A(x))]}{2m+1}=\frac{x}{2\pi},\quad 0\leq x\leq 2\pi.
\end{equation}
That is, the expected normalized counting measure of the zeros on the unit circle converges weakly to the uniform measure on $[0,2\pi]$.
\end{itemize}

\end{theorem}
\begin{proof}
If $\alpha> -\nicefrac{1}{2}$, or $\alpha = -\nicefrac{1}{2}$ and $S(-1)=+\infty$, then $\lim_{m\to\infty}u_m=\lim_{m\to\infty}v_m=0$ and Lemma~\ref{lem:si_polys_mean_asymp} (parts (i-ii)) and~\eqref{eq:expected_density} imply that
$$
\lim_{m\to\infty} \frac{p_m(\phi) }{2m+1} = \frac{1}{2\pi} \frac{1}{\sqrt{(1+\alpha)(3+2\alpha)}}, \quad \phi \in[0,2\pi].
$$
If $\alpha<-\nicefrac{1}{2}$, or $\alpha=-\nicefrac{1}{2}$ and $S(-1)<\infty$, then
\begin{equation}\label{eq:u_m_v_m_limit}
\lim_{m\to\infty}u_m=\lim_{m\to\infty}v_m=(2\sigma^{2}S(2\alpha))^{-1/2}=:C\in (0,+\infty).
\end{equation}
Therefore, the right-hand side of~\eqref{eq:expected_density} does not converge due to the rapid oscillations of $\sin \phi_m$. The convergence~\eqref{eq:expected_counting_measure_to_lebesgue} follows upon integrating~\eqref{eq:expected_density} over $[0,x]$ and changing the variable $t=(m+\nicefrac{1}{2})\phi$:
\begin{align*}
\int_{0}^{x}\frac{p_m(\phi) }{2m+1}{\rm d}\phi&=
\frac{1}{2\pi}
\frac{u_m}{v_m}
 \int_0^x \ee^{
 -u_m^2\cos^2\phi_m -v_m^2\sin^2\phi_m
 }{\rm d}\phi\\
 &\hspace{2cm}+\frac{u_m}{2\sqrt{\pi}}\int_0^x
\ee^{-u_m^2\cos^2\phi_m}\erf ( v_m  |\sin\phi_m|)|\sin\phi_m|{\rm d}\phi\\
&=\frac{1}{2\pi(m+\nicefrac{1}{2})}
\frac{u_m}{v_m}
 \int_0^{(m+\nicefrac{1}{2})x} \ee^{
 -u_m^2\cos^2 t -v_m^2\sin^2 t
 }{\rm d}t\\
 &\hspace{2cm}+\frac{u_m}{2\sqrt{\pi}(m+\nicefrac{1}{2})}\int_0^{(m+\nicefrac{1}{2})x}
\ee^{-u_m^2\cos^2 t}\erf ( v_m  |\sin t|)|\sin t|{\rm d}t.
\end{align*}
Using that the integrands are $\pi$-periodic functions we conclude that, as $m\to\infty$,
$$
\int_{0}^{x}\frac{p_m(\phi) }{2m+1}{\rm d}\phi=\frac{x}{2\pi^2}
\frac{u_m}{v_m}
 \int_0^{\pi} \ee^{
 -u_m^2\cos^2 t -v_m^2\sin^2 t
 }{\rm d}t+\frac{x u_m}{2\pi^{3/2}}\int_0^{\pi}
\ee^{-u_m^2\cos^2 t}\erf ( v_m \sin t )\sin t{\rm d}t+o(1).
$$
Applying~\eqref{eq:u_m_v_m_limit} we obtain
$$
\int_{0}^{x}\frac{p_m(\phi) }{2m+1}{\rm d}\phi=\frac{x}{2\pi^2}
 \int_0^{\pi} \ee^{-C^2}{\rm d}t+\frac{Cx}{2\pi^{3/2}}\int_0^{\pi}
\ee^{-C^2\cos^2 t}\erf ( C\sin t )\sin t{\rm d}t+o(1).
$$
From~\eqref{lem:integrals_eq2} after elementary manipulations we finally deduce~\eqref{eq:expected_counting_measure_to_lebesgue}.
\end{proof}

\appendix
\section{Auxiliary results}
\begin{prop}\label{prop:riemann_lebesgue}
Let $f:[a,b]\to\mathbb{C}$ be a continuous function, $-\infty<a<b<+\infty$, and $\psi\in\mathbb{R}$. Then
$$
\lim_{n\to\infty}\frac{1}{n}\sum_{k=\lfloor na\rfloor}^{\lfloor nb\rfloor-1}f(k/n)\ee^{\ii \psi k}=
\begin{cases}
\int_a^{b}f(x){\rm d}x,& \psi\in 2\pi\mathbb{Z},\\
0,&\psi\not\in 2\pi\mathbb{Z}.
\end{cases}.
$$
\end{prop}
\begin{proof}
The case $\psi\in 2\pi\mathbb{Z}$ is trivial. Assume that $\psi\not\in 2\pi\mathbb{Z}$ and, without loss of generality, $a=0$ and $b=1$. Let $U$ be a random variable with the uniform distribution on $[0,1]$. Then
$$
\frac{1}{n}\sum_{k=0}^{n-1}f(k/n)\ee^{\ii \psi k}=\mathbb{E}\left[f\left(\frac{\lfloor nU\rfloor}{n}\right)\ee^{\ii \psi \lfloor nU\rfloor}\right]=\mathbb{E}\left[\left(f\left(\frac{\lfloor nU\rfloor}{n}\right)-f(U)\right)\ee^{\ii \psi \lfloor nU\rfloor}\right]+\mathbb{E}\left[f(U)\ee^{\ii \psi \lfloor nU\rfloor}\right].
$$
By the dominated convergence theorem and continuity of $f$,
$$
\lim_{n\to\infty}\mathbb{E}\left[\left(f\left(\frac{\lfloor nU\rfloor}{n}\right)-f(U)\right)\ee^{\ii \psi \lfloor nU\rfloor}\right]=0.
$$
Now observe that the change of variables $x=y+1/n$ implies
\begin{align*}
\mathbb{E}\left[f(U)\ee^{\ii \psi \lfloor nU\rfloor}\right]&=\int_0^{1}f(x)\ee^{\ii \psi \lfloor nx\rfloor}{\rm d}x=\ee^{\ii \psi}\int_{-1/n}^{1-1/n}f(y+1/n)\ee^{\ii \psi \lfloor ny\rfloor}{\rm d}y\\
&=\ee^{\ii \psi}\int_{0}^{1-1/n}f(y+1/n)\ee^{\ii \psi \lfloor ny\rfloor}{\rm d}y+O(n^{-1})\\
&=\ee^{\ii \psi}\int_{0}^{1-1/n}(f(y+1/n)-f(y))\ee^{\ii \psi \lfloor ny\rfloor}{\rm d}y+\ee^{\ii \psi}\int_{0}^{1}f(y)\ee^{\ii \psi \lfloor ny\rfloor}{\rm d}y+O(n^{-1}).
\end{align*}
Rearranging, we obtain
$$
(1-\ee^{\ii\psi})\int_0^{1}f(x)\ee^{\ii \psi \lfloor nx\rfloor}{\rm d}x=\ee^{\ii \psi}\int_{0}^{1-1/n}(f(y+1/n)-f(y))\ee^{\ii \psi \lfloor ny\rfloor}{\rm d}y+O(n^{-1}).
$$
The left-hand side converges to $0$ as $n\to\infty$ by the Lebesgue dominated convergence theorem and continuity of $f$. Since $\psi\not\in 2\pi\mathbb{Z}$ implies $(1-\ee^{\ii\psi})\neq 0$, the proof is complete.
\end{proof}

\begin{lemma}\label{lem:W_-1}
Assume that $b$ is regularly varying with index $\alpha\leq -\nicefrac{1}{2}$ and, if $\alpha=-\nicefrac{1}{2}$, also that~\eqref{eq:convergence_b^2} holds. The sequence $a_n=-W_{-1}(-nb^2(n))$ is well-defined for all sufficiently large $n\in\mathbb{N}$ and satisfies:
\begin{enumerate}
\item[(i)] $\lim_{n\to\infty }a_n\to+\infty$;
\item[(ii)] $a_n=O(\log n)$;
\item[(iii)] if $\alpha<-\nicefrac{1}{2}$, $a_n=-2\log b(n)-\log n + \log \log n + \log (-2\alpha-1) + o(1)$, as $n\to\infty$.
\end{enumerate}
\end{lemma}
\begin{proof}
The regular variation of $b$ and $\sum_{k\geq 0}b^2(k)<\infty$ imply that $nb^2(n)\to 0$, as $n\to\infty$. Indeed,
$$
0\leftarrow \sum_{k=n+1}^{2n}b^2(k)\geq n\inf_{n<k\leq 2n}b^2(k)\geq nb^2(n)\inf_{1\leq \lambda\leq 2}\frac{b^2(\lambda n)}{b^2(n)}.
$$
By the uniform convergence theorem for regularly varying functions, the last infimum converges to $\inf_{1\leq \lambda \leq 2} \lambda^{2\alpha}=2^{2\alpha}$. Thus, the sequence $a_n$ is well-defined for all $n\in\mathbb{N}$ such that $nb^2(n)<\ee^{-1}$. All three claims follow now from the asymptotic expansion
\begin{equation}\label{eq:W_-1_asymp}
W_{-1}(x)=-\log(-1/x)-\log\log(-1/x)-\frac{\log \log(-1/x)}{\log(-1/x)}+O\left(\left(\frac{\log \log(-1/x)}{\log(-1/x)}\right)^2\right),\quad x\to 0-,
\end{equation}
in conjunction with the regular variation of the function $x\mapsto 1/(xb^2(x))$.

If $\alpha<-\nicefrac{1}{2}$, then $-\log (n b^2(n))~\sim~(-2\alpha-1)\log n$. Therefore,~\eqref{eq:W_-1_asymp} implies
\begin{multline*}
a_n=-\log (nb^2(n))+\log \log (1/(nb^2(n)))+o(1)\\
=-2\log b(n)-\log n +\log \log n+\log (-2\alpha-1)+o(1).
\end{multline*}
\end{proof}

\begin{lemma}\label{lem:variance}
With probability one the process $(\widetilde{\mathfrak{G}}(t))_{t\in\mathbb{R}}$ defined by~\eqref{eq:s_tilde_integral_rep} does not have multiple real zeros.
\end{lemma}
\begin{proof}
According to a result by E.~V.~Bulinskaya~\cite{Bulinskaya:1961}, see also Lemma 4.3 in~\cite{Iksanov+Kabluchko+Marynych:2016}, it suffices to verify that the density of $\widetilde{\mathfrak{G}}(t)$ is  bounded uniformly in $t\in\mathbb{R}$. Since the process $\widetilde{\mathfrak{G}}$ is centered Gaussian this amounts to checking that the variance ${\rm Var}\,[\widetilde{\mathfrak{G}}(t)]$ is bounded away from zero. Observe that~\eqref{eq:s_tilde_integral_rep} yields
$$
{\rm Var}\,[\widetilde{\mathfrak{G}}(t)]=4\sigma_1^2\int_0^1 x^{2\alpha}\cos^2(t(x-1)){\rm d}x+4\sigma_2^2\int_0^1 x^{2\alpha}\sin^2(t(x-1)){\rm d}x,\quad t\in\mathbb{R}.
$$
If $c:=\min(\sigma_1^2,\sigma_2^2)>0$, then
$$
{\rm Var}\,[\widetilde{\mathfrak{G}}(t)]\geq 4c\int_0^1 x^{2\alpha}{\rm d}x=\frac{4c}{1+2\alpha}>0,\quad t\in\mathbb{R}.
$$
If $\sigma_2^2=0$ and $\sigma_1^2>0$, then
$$
{\rm Var}\,[\widetilde{\mathfrak{G}}(t)]\geq 4\sigma_1^2\min(2^{-2\alpha},1)\int_{\nicefrac{1}{2}}^1 \cos^2(t(x-1)){\rm d}x=\sigma_1^2\min(2^{-2\alpha},1)(1+\sin(t)/t).
$$
The right-hand side is bounded away from zero. Finally, if $\sigma_1^2=0$ and $\sigma_2^2>0$,
$$
{\rm Var}\,[\widetilde{\mathfrak{G}}(t)]\geq 4\sigma_2^2\min(2^{-2\alpha},1)\int_{\nicefrac{1}{2}}^1 \sin^2(t(x-1)){\rm d}x=\sigma_2^2\min(2^{-2\alpha},1)(1-\sin(t)/t).
$$
The right-hand side is bounded away from zero for all $|t|>a$, where $a>0$ is any fixed positive constant. Thus, Bulinskaya's lemma implies that there are no multiple zeros on $\mathbb{R}\setminus\{0\}$. Although, if  $\sigma_1^2=0$ and $\sigma_2^2>0$, then $\widetilde{\mathfrak{G}}(0)=0$, but we have
$$
\widetilde{\mathfrak{G}}^{\prime}(0)=2\sigma_2\int_0^1 x^{\alpha}(1-x){\rm d} B_2(x),
$$
which is a centered Gaussian random variable with variance $4\sigma_2^2 \int_0^1 x^{2\alpha}(1-x)^2{\rm d}x>0$. Thus, $\mathbb{P}\{\widetilde{\mathfrak{G}}^{\prime}(0)=0\}=0$. The proof is complete.
\end{proof}

\begin{lemma}\label{lem:si_polys_mean_asymp}
Assume that $b^2(x)=x^{2\alpha}\ell^2(x)$. Let $g_1$ and $g_2$ be defined by~\eqref{eq:g_1_2_def}.
\begin{itemize}
\item[(i)] If $\alpha>-\nicefrac{1}{2}$, then
$$
\lim_{m\to\infty}\frac{1}{m}\left(\frac{g_2(m)}{g_1(m)}\right)^{1/2}=\frac{1}{\sqrt{(1+\alpha)(3+2\alpha)}}.
$$
\item[(ii)] If $\alpha=-\nicefrac{1}{2}$, then
$$
1-\frac{1}{m}\left(\frac{g_2(m)}{g_1(m)}\right)^{1/2}~\sim~\frac{3}{4}\frac{\ell^2(m)}{\sum_{k=1}^{m}\frac{\ell^2(k)}{k}}~\to~0,\quad m\to\infty.
$$
\item[(iii)] If $-1<\alpha<-\nicefrac{1}{2}$, then
$$
1-\frac{1}{m}\left(\frac{g_2(m)}{g_1(m)}\right)^{1/2}~\sim~\frac{2+\alpha}{(1+\alpha)(3+2\alpha)}\frac{m^{1+2\alpha}\ell^{2}(m)}{2S(2\alpha)}~\to~0,\quad m\to\infty.
$$
\item[(iv)] If $\alpha\leq -1$, then
$$
1-\frac{1}{m}\left(\frac{g_2(m)}{g_1(m)}\right)^{1/2}~\sim~-\frac{1}{2m}+\frac{1}{S(2\alpha)}\frac{\sum_{k=1}^{m}k^{1+2\alpha}\ell^2(k)}{m}~\to~0,\quad m\to\infty.
$$
\end{itemize}
\end{lemma}
\begin{proof}
The proof follows by elementary manipulations from the direct half of Karamata's theorem, see~\cite[Proposition 1.5.8]{BGT} and the relation
\begin{equation}\label{eq:app_lemma6_eq1}
\frac{\ell^2(m)}{\sum_{k=1}^{m}\frac{\ell^2(k)}{k}}~\to~0,\quad m\to\infty,
\end{equation}
which holds for any slowly varying function $\ell$, see formula (1.5.8) in the same reference. In parts (ii)-(iv) we also used that $1-(1-x)^{1/2}~\sim x/2$, as $x\to 0$. For example, let us check (iv). Without loss of generality assume that $\sigma^2=1$. Write
\begin{equation}\label{eq:app_lemma6_eq2}
1-\frac{1}{m}\left(\frac{g_2(m)}{g_1(m)}\right)^{1/2}=1-\frac{1}{m}\left((m+\nicefrac{1}{2})^2-\frac{2m+1}{g_1(m)}\sum_{k=1}^{m}kb^2(k)+\frac{1}{g_1(m)}\sum_{k=1}^{m}k^2 b^2(k)\right)^{1/2}.
\end{equation}
Observe that
$$
\frac{1}{m}\sum_{k=1}^{m}k^2 b^2(k)=\frac{1}{m}\sum_{k=1}^{m}k^{2+2\alpha} \ell^2(k)=o\left(\sum_{k=1}^{m}k^{1+2\alpha} \ell^2(k)\right)=o\left(\sum_{k=1}^{m}k b^2(k)\right),\quad m\to\infty.
$$
Indeed, if $\alpha=-1$, this follows from~\eqref{eq:app_lemma6_eq1}, whereas if $\alpha<-1$ this follows from
$$
\lim_{m\to\infty}\frac{1}{m}\sum_{k=1}^{m}k^{2+2\alpha} \ell^2(k)=0,
$$
which is true by the Stolz-Ces\'{a}ro theorem. Thus,~\eqref{eq:app_lemma6_eq2} and $1-(1-x)^{1/2}~\sim x/2$ yields
\begin{align*}
1-\frac{1}{m}\left(\frac{g_2(m)}{g_1(m)}\right)^{1/2}&~\sim~\frac{1}{2}\left(1-(1+1/(2m))^2+\frac{2m+1}{m^2 g_1(m)}\sum_{k=1}^{m}kb^2(k)-\frac{1}{m^2 g_1(m)}\sum_{k=1}^{m}k^2 b^2(k)\right)\\
&~\sim~\frac{1}{m g_1(m)}\sum_{k=1}^{m}kb^2(k)-\frac{1}{2m}~\sim~\frac{1}{m S(2\alpha)}\sum_{k=1}^{m}k^{1+2\alpha}\ell^2(k)-\frac{1}{2m}.
\end{align*}
The proof of part (iv) is complete.
\end{proof}

\begin{lemma}\label{lem:integrals}
Assume $u>0$ and $\phi\in [0,2\pi]$. Then
\begin{align}
&\int_0^1 \ee^{-\frac{1}{x^2}u^2\sin^2 \phi} \, \d x = \ee^{-u^2\sin^2 \phi} - \sqrt{\pi}\, \erfc(u |\sin \phi|) \, u|\sin \phi|;\label{lem:integrals_eq1}\\
&\frac{1}{\sqrt{\pi}} \int_0^{\pi} \ee^{-u^2\cos^2 \phi} \erfc(u \sin \phi)  (u \sin \phi) \, \d \phi = \ee^{-u^2} - \erfc (u);\label{lem:integrals_eq2}\\
&\frac{1}{\pi}\int_0^1\int_0^{\pi}\ee^{-u^2(\cos^2 \phi +x^{-2}\sin^2 \phi)}\,  \d \phi \d x = \erfc(u).\label{lem:integrals_eq3}
\end{align}
\end{lemma}

\begin{proof}
Setting $q:=u|\sin \phi|$, changing variable $y=1/x$ and integrating by parts, we obtain
\begin{align*}
\int_0^1 \ee^{-q^2/x^2} \, \d x = \int_1^{+\infty} \ee^{-q^2 y^2}\, \frac{\d y}{y^2} = \ee^{-q^2} - 2 \int_1^{+\infty} q^2 \ee^{-q^2 y^2}\, \d y = \ee^{-q^2} -\sqrt{\pi} q \erfc (q).
\end{align*}
This proves equality~\eqref{lem:integrals_eq1}.

The substitution $x=\cos \phi$ transforms the integral in~\eqref{lem:integrals_eq2} into
\begin{align*}
\frac{2u}{\sqrt{\pi}} \int_0^{1} \ee^{-u^2 x^2} \erfc(u \sqrt{1-x^2}) \, \d x =\frac{4u^2}{\pi} \iint_{D} \ee^{-u^2 (x^2+y^2)}  \, \d x\d y,
\end{align*}
where the integral on the right is over the domain $D=\{ (x,y): \,  y \geq \sqrt{1-x^2}, \, x\in [0,1]
\}$. Using that $D$ is the difference of $D_c=\{ (x,y): \,  x^2+y^2 \geq 1,  x\geq 0, y\geq 0 \}$ and $\{ (x,y): \, x\geq 1, y\geq 0 \}$, one can easily evaluate the double integral using the polar coordinates:
\begin{multline*}
\frac{4u^2}{\pi} \iint_{D}\!\! \ee^{-u^2 (x^2+y^2)}  \, \d x\d y \\= \frac{4u^2}{\pi} \iint_{D_c} \!\! \ee^{-u^2 (x^2+y^2)}  \, \d x\d y - \frac{4u^2}{\pi} \int_1^{+\infty} \!\!\!\!\ee^{-u^2 x^2} \int_0^{+\infty} \!\!\!\! \ee^{-u^2 y^2}  \, \d y \, \d x
 = \ee^{-u^2} - \erfc(u).
\end{multline*}
This proves equality~\eqref{lem:integrals_eq2}.

Equality~\eqref{lem:integrals_eq3} follows upon multiplying both sides of~\eqref{lem:integrals_eq1} by $\frac{1}{\pi}\ee^{-u^2 \cos^2 \phi}$ and integrating over $\phi\in [0,\pi]$ using~\eqref{lem:integrals_eq2}.
\end{proof}

\section{The characteristic polynomial of a random Haar unitary matrix}

The following result appears on p.~2245 of~\cite{Chhaibi+Madaule+Najnudel:2018}. For completeness, we provide a relatively straightforward proof that relies on a well-known theorem of Diaconis and Shahshahani~\cite{Diaconis+Shahshahani:1994}.

\begin{prop}\label{prop:char_poly_unitary}
Let $U_n$ be a random Haar unitary matrix and ${\bf Id}_n$ the identity matrix, both of dimension $n$. Then
$$
(\log {\rm det}\,({\bf Id}_n-z U_n))_{z\in \mathbb{D}_1}~\Longrightarrow~(P^{{\rm crit}}_{\infty}(z))_{z\in \mathbb{D}_1},\quad n\to\infty,
$$
on the space $\mathcal{A}(\mathbb{D}_1)$, where
$P^{{\rm crit}}_{\infty}(z)=\sum_{k\geq 1}\frac{\xi_k}{\sqrt{k}}z^k$ and $(\xi_k)_{k\geq 1}$ are i.i.d.\ standard complex normal.
\end{prop}
\begin{proof}
Fix $M\in\mathbb{N}$ and write
$$
\log {\rm det}\,({\bf Id}_n-z U_n)=-\sum_{k\geq 1}\frac{z^k}{k}{\rm Tr}(U_n^k)=-\sum_{k=1}^{M}\frac{z^k}{k}{\rm Tr}\,(U_n^k)-\sum_{k>M}\frac{z^k}{k}{\rm Tr}\,(U_n^k),\quad z\in\mathbb{D}_1.
$$
According to~\cite[Theorem 1]{Diaconis+Shahshahani:1994}
$$
\left(\sum_{k=1}^{M}\frac{z^k}{k}{\rm Tr}\,(U_n^k)\right)_{z\in\mathbb{D}_1}~\Longrightarrow~\left(\sum_{k=1}^{M}\frac{z^k}{\sqrt{k}}\xi_k\right)_{z\in\mathbb{D}_1},\quad n\to\infty
$$
on the space $\mathcal{A}(\mathbb{D}_1)$. It is clear that the right-hand side converges a.s. on $\mathcal{A}(\mathbb{D}_1)$ to $P^{{\rm crit}}_{\infty}$. According to~\cite[Theorem 3.2]{Billingsley} it remains to check that, for every $\varepsilon>0$ and a compact set $K\subseteq
\mathbb{D}_1$,
$$
\lim_{m\to\infty}\limsup_{n\to\infty}\mathbb{P}\left\{\sup_{z\in K}\left|\sum_{k>M}\frac{z^k}{k}{\rm Tr}\,(U_n^k)\right|>\varepsilon\right\}=0.
$$
By using Markov's inequality this follows from the fact that $\rho:=\sup_{z\in K}|z|<1$ and
$$
\mathbb{E}[|{{\rm Tr}\,(U_n^k)}|^2]=k,\quad n\geq 2,
$$
which is a consequence of~\cite[Theorem 2]{Diaconis+Shahshahani:1994} applied with $a_1=\cdots=a_{k-1}=b_1=\cdots=b_{k-1}=0$ and $a_k=b_k=1$. The proof is complete.
\end{proof}

\bibliographystyle{plain}
\bibliography{Biblio}

@article{Do+Nguyen:2025,
author = {Yen Q. Do and Nhan D. V. Nguyen},
title = {{Real roots of random polynomials: Asymptotics of the variance}},
volume = {30},
journal = {Electronic Journal of Probability},
number = {none},
publisher = {Institute of Mathematical Statistics and Bernoulli Society},
pages = {1 -- 37},
year = {2025},
doi = {10.1214/25-EJP1375},
URL = {https://doi.org/10.1214/25-EJP1375}
}

@article {Shiffman+Zelditch:2003,
    AUTHOR = {Shiffman, Bernard and Zelditch, Steve},
     TITLE = {Equilibrium distribution of zeros of random polynomials},
   JOURNAL = {Int. Math. Res. Not.},
  FJOURNAL = {International Mathematics Research Notices},
      YEAR = {2003},
    NUMBER = {1},
     PAGES = {25--49},
      ISSN = {1073-7928,1687-0247},
   MRCLASS = {60G99 (30C10 30C15)},
  MRNUMBER = {1935565},
MRREVIEWER = {Kambiz\ Farahmand},
       DOI = {10.1155/S1073792803206073},
       URL = {https://doi-org.ezproxy.library.qmul.ac.uk/10.1155/S1073792803206073},
}

@article {Bleher+Ridzal:2002,
    AUTHOR = {Bleher, Pavel and Ridzal, Denis},
     TITLE = {{${\rm SU}(1,1)$} random polynomials},
   JOURNAL = {J. Statist. Phys.},
  FJOURNAL = {Journal of Statistical Physics},
    VOLUME = {106},
      YEAR = {2002},
    NUMBER = {1-2},
     PAGES = {147--171},
      ISSN = {0022-4715,1572-9613},
   MRCLASS = {82B44 (60G99 81Q50)},
  MRNUMBER = {1881723},
MRREVIEWER = {Aernout\ C. D. van Enter},
       DOI = {10.1023/A:1013124213145},
       URL = {https://doi-org.ezproxy.library.qmul.ac.uk/10.1023/A:1013124213145},
}

@article{Fyodorov+Bouchaud:2008,
doi = {10.1088/1751-8113/41/37/372001},
url = {https://doi.org/10.1088/1751-8113/41/37/372001},
year = {2008},
month = {aug},
publisher = {},
volume = {41},
number = {37},
pages = {372001},
author = {Fyodorov, Yan V and Bouchaud, Jean-Philippe},
title = {Freezing and extreme-value statistics in a random energy model with logarithmically correlated potential},
journal = {Journal of Physics A: Mathematical and Theoretical}
}

@misc{Kab+Klim:GREM:2014,
      title={Generalized Random Energy Model at Complex Temperatures},
      author={Zakhar Kabluchko and Anton Klimovsky},
      year={2014},
      eprint={1402.2142},
      archivePrefix={arXiv},
      primaryClass={math.PR},
      note={ArXiV Preprint: https://arxiv.org/abs/1402.2142},
}

@article {Michelen:2021,
    AUTHOR = {Michelen, Marcus},
     TITLE = {Real roots near the unit circle of random polynomials},
   JOURNAL = {Trans. Amer. Math. Soc.},
  FJOURNAL = {Transactions of the American Mathematical Society},
    VOLUME = {374},
      YEAR = {2021},
    NUMBER = {6},
     PAGES = {4359--4374},
   MRCLASS = {60G15 (26C10 30C15 42A32 60F05)},
  MRNUMBER = {4251232},
       DOI = {10.1090/tran/8379},
       URL = {https://doi.org/10.1090/tran/8379},
}

@article {Dembo+Mukherjee:2015,
    AUTHOR = {Dembo, Amir and Mukherjee, Sumit},
     TITLE = {No zero-crossings for random polynomials and the heat
              equation},
   JOURNAL = {Ann. Probab.},
  FJOURNAL = {The Annals of Probability},
    VOLUME = {43},
      YEAR = {2015},
    NUMBER = {1},
     PAGES = {85--118},
      ISSN = {0091-1798},
   MRCLASS = {60G15 (26A12 26C10 35K05)},
  MRNUMBER = {3298469},
MRREVIEWER = {Frank Aurzada},
       DOI = {10.1214/13-AOP852},
       URL = {https://doi.org/10.1214/13-AOP852},
}

@incollection {Ibragimov+Zaporozhets:2013,
    AUTHOR = {Ibragimov, Ildar and Zaporozhets, Dmitry},
     TITLE = {On distribution of zeros of random polynomials in complex
              plane},
 BOOKTITLE = {Prokhorov and contemporary probability theory},
    SERIES = {Springer Proc. Math. Stat.},
    VOLUME = {33},
     PAGES = {303--323},
 PUBLISHER = {Springer, Heidelberg},
      YEAR = {2013},
   MRCLASS = {60G99 (30C15 60F15)},
  MRNUMBER = {3070481},
MRREVIEWER = {Zakhar Kabluchko},
       DOI = {10.1007/978-3-642-33549-5\_18},
       URL = {https://doi.org/10.1007/978-3-642-33549-5_18},
}

@article {Peres+Virag:2005,
    AUTHOR = {Peres, Yuval and Vir\'{a}g, B\'{a}lint},
     TITLE = {Zeros of the i.i.d. {G}aussian power series: a conformally
              invariant determinantal process},
   JOURNAL = {Acta Math.},
  FJOURNAL = {Acta Mathematica},
    VOLUME = {194},
      YEAR = {2005},
    NUMBER = {1},
     PAGES = {1--35},
      ISSN = {0001-5962},
   MRCLASS = {60G99 (60G15)},
  MRNUMBER = {2231337},
MRREVIEWER = {Tomohiro Sasamoto},
       DOI = {10.1007/BF02392515},
       URL = {https://doi-org.ezproxy.library.qmul.ac.uk/10.1007/BF02392515},
}

@article{Tao+Vu:2014,
    author = {Tao, Terence and Vu, Van},
    title = {Local Universality of Zeroes of Random Polynomials},
    journal = {International Mathematics Research Notices},
    volume = {2015},
    number = {13},
    pages = {5053-5139},
    year = {2014},
    month = {06},
    issn = {1073-7928},
    doi = {10.1093/imrn/rnu084},
    url = {https://doi.org/10.1093/imrn/rnu084},
    eprint = {https://academic.oup.com/imrn/article-pdf/2015/13/5053/9277230/rnu084.pdf},
}

@article {Shepp+Vanderbei:1995,
    AUTHOR = {Shepp, Larry A. and Vanderbei, Robert J.},
     TITLE = {The complex zeros of random polynomials},
   JOURNAL = {Trans. Amer. Math. Soc.},
  FJOURNAL = {Transactions of the American Mathematical Society},
    VOLUME = {347},
      YEAR = {1995},
    NUMBER = {11},
     PAGES = {4365--4384},
      ISSN = {0002-9947,1088-6850},
   MRCLASS = {30C15 (60G99)},
  MRNUMBER = {1308023},
MRREVIEWER = {Kambiz\ Farahmand},
       DOI = {10.2307/2155041},
       URL = {https://doi.org/10.2307/2155041},
}

@article {Bogomolny+Bohigas+Leboeuf:1996,
    AUTHOR = {Bogomolny, E. and Bohigas, O. and Leboeuf, P.},
     TITLE = {Quantum chaotic dynamics and random polynomials},
   JOURNAL = {J. Statist. Phys.},
  FJOURNAL = {Journal of Statistical Physics},
    VOLUME = {85},
      YEAR = {1996},
    NUMBER = {5-6},
     PAGES = {639--679},
      ISSN = {0022-4715,1572-9613},
   MRCLASS = {81Q50 (82B44)},
  MRNUMBER = {1418808},
MRREVIEWER = {Christian\ Grosche},
       DOI = {10.1007/BF02199359},
       URL = {https://doi.org/10.1007/BF02199359},
}

@article{Bogomolny+Bohigas+Leboeuf:1992,
  title = {Distribution of roots of random polynomials},
  author = {Bogomolny, E. and Bohigas, O. and Leboeuf, P.},
  journal = {Phys. Rev. Lett.},
  volume = {68},
  issue = {18},
  pages = {2726--2729},
  numpages = {0},
  year = {1992},
  month = {May},
  publisher = {American Physical Society},
  doi = {10.1103/PhysRevLett.68.2726},
  url = {https://link.aps.org/doi/10.1103/PhysRevLett.68.2726}
}

@article{Aldous+Fyodorov:2004,
doi = {10.1088/0305-4470/37/4/011},
url = {https://dx.doi.org/10.1088/0305-4470/37/4/011},
year = {2004},
month = {jan},
publisher = {},
volume = {37},
number = {4},
pages = {1231},
author = {Anthony P Aldous and Yan V Fyodorov},
title = {Real roots of random polynomials: universality close to accumulation points},
journal = {Journal of Physics A: Mathematical and General}
}

@article {Do:2021,
    AUTHOR = {Do, Yen Q.},
     TITLE = {Real roots of random polynomials with coefficients of
              polynomial growth: a comparison principle and applications},
   JOURNAL = {Electron. J. Probab.},
  FJOURNAL = {Electronic Journal of Probability},
    VOLUME = {26},
      YEAR = {2021},
     PAGES = {Paper No. 144, 45},
      ISSN = {1083-6489},
   MRCLASS = {60G99 (60F05)},
  MRNUMBER = {4346676},
MRREVIEWER = {Przemys\l aw\ Matu\l a},
       DOI = {10.1214/21-ejp719},
       URL = {https://doi.org/10.1214/21-ejp719},
}

@article {Flasche+Kabluhcko:2020,
    AUTHOR = {Flasche, Hendrik and Kabluchko, Zakhar},
     TITLE = {Expected number of real zeroes of random {T}aylor series},
   JOURNAL = {Commun. Contemp. Math.},
  FJOURNAL = {Communications in Contemporary Mathematics},
    VOLUME = {22},
      YEAR = {2020},
    NUMBER = {7},
     PAGES = {1950059, 38},
      ISSN = {0219-1997,1793-6683},
   MRCLASS = {30B20 (26C10 30C15 60F05 60F17 60F99 60G15)},
  MRNUMBER = {4135008},
MRREVIEWER = {Christian\ Lavault},
       DOI = {10.1142/S0219199719500597},
       URL = {https://doi.org/10.1142/S0219199719500597},
}

@article{Do+Nguen+Vu:2018,
    AUTHOR = {Do, Yen and Nguyen, Oanh and Vu, Van},
     TITLE = {Roots of random polynomials with coefficients of polynomial
              growth},
   JOURNAL = {Ann. Probab.},
  FJOURNAL = {The Annals of Probability},
    VOLUME = {46},
      YEAR = {2018},
    NUMBER = {5},
     PAGES = {2407--2494},
      ISSN = {0091-1798,2168-894X},
   MRCLASS = {60G99},
  MRNUMBER = {3846831},
MRREVIEWER = {Maxim\ L.\ Yattselev},
       DOI = {10.1214/17-AOP1219},
       URL = {https://doi.org/10.1214/17-AOP1219},
}

@misc{Krishnapur+Lundberg+Nguyen:2022,
      title={The number of limit cycles bifurcating from a randomly perturbed center},
      author={Manjunath Krishnapur and Erik Lundberg and Oanh Nguyen},
      year={2022},
      eprint={2112.05672},
      archivePrefix={arXiv},
      primaryClass={math.PR},
      url={https://arxiv.org/abs/2112.05672},
}

@article {Nguyen+Vu:2022,
    AUTHOR = {Nguyen, Oanh and Vu, Van},
     TITLE = {Roots of random functions: a framework for local universality},
   JOURNAL = {Amer. J. Math.},
  FJOURNAL = {American Journal of Mathematics},
    VOLUME = {144},
      YEAR = {2022},
    NUMBER = {1},
     PAGES = {1--74},
      ISSN = {0002-9327,1080-6377},
   MRCLASS = {60E05 (30C15 42A61)},
  MRNUMBER = {4367414},
MRREVIEWER = {Benjamin\ Arras},
       DOI = {10.1353/ajm.2022.0000},
       URL = {https://doi.org/10.1353/ajm.2022.0000},
}

@article {Ibragimov+Zeitouni:1997,
    AUTHOR = {Ibragimov, Ildar and Zeitouni, Ofer},
     TITLE = {On roots of random polynomials},
   JOURNAL = {Trans. Amer. Math. Soc.},
  FJOURNAL = {Transactions of the American Mathematical Society},
    VOLUME = {349},
      YEAR = {1997},
    NUMBER = {6},
     PAGES = {2427--2441},
      ISSN = {0002-9947,1088-6850},
   MRCLASS = {60G99 (26C10 30B20 60F05)},
  MRNUMBER = {1390040},
MRREVIEWER = {Kambiz\ Farahmand},
       DOI = {10.1090/S0002-9947-97-01766-2},
       URL = {https://doi.org/10.1090/S0002-9947-97-01766-2},
}

@article {Hughes+Nikeghbali:2008,
    AUTHOR = {Hughes, C. P. and Nikeghbali, A.},
     TITLE = {The zeros of random polynomials cluster uniformly near the
              unit circle},
   JOURNAL = {Compos. Math.},
  FJOURNAL = {Compositio Mathematica},
    VOLUME = {144},
      YEAR = {2008},
    NUMBER = {3},
     PAGES = {734--746},
      ISSN = {0010-437X,1570-5846},
   MRCLASS = {30C10 (30B20 30C15)},
  MRNUMBER = {2422348},
MRREVIEWER = {Scott\ E.\ Zrebiec},
       DOI = {10.1112/S0010437X07003302},
       URL = {https://doi.org/10.1112/S0010437X07003302},
}

@article {Kabluchko+Zaporozhets:2014,
    AUTHOR = {Kabluchko, Zakhar and Zaporozhets, Dmitry},
     TITLE = {Asymptotic distribution of complex zeros of random analytic
              functions},
   JOURNAL = {Ann. Probab.},
  FJOURNAL = {The Annals of Probability},
    VOLUME = {42},
      YEAR = {2014},
    NUMBER = {4},
     PAGES = {1374--1395},
      ISSN = {0091-1798,2168-894X},
   MRCLASS = {30C15 (26C10 30B20 60B10 60B20 60G57)},
  MRNUMBER = {3262481},
MRREVIEWER = {Matthew\ M.\ Jones},
       DOI = {10.1214/13-AOP847},
       URL = {https://doi.org/10.1214/13-AOP847},
}

@article {Iksanov+Kabluchko+Marynych:2016,
    AUTHOR = {Iksanov, Alexander and Kabluchko, Zakhar and Marynych,
              Alexander},
     TITLE = {Local universality for real roots of random trigonometric
              polynomials},
   JOURNAL = {Electron. J. Probab.},
  FJOURNAL = {Electronic Journal of Probability},
    VOLUME = {21},
      YEAR = {2016},
     PAGES = {Paper No. 63, 19},
      ISSN = {1083-6489},
   MRCLASS = {60G50 (26C99 30C15 60F17 60G55)},
  MRNUMBER = {3563891},
       DOI = {10.1214/16-EJP9},
       URL = {https://doi.org/10.1214/16-EJP9},
}

@book {Billingsley,
    AUTHOR = {Billingsley, Patrick},
     TITLE = {Convergence of probability measures},
    SERIES = {Wiley Series in Probability and Statistics: Probability and
              Statistics},
   EDITION = {Second},
      NOTE = {A Wiley-Interscience Publication},
 PUBLISHER = {John Wiley \& Sons, Inc., New York},
      YEAR = {1999},
     PAGES = {x+277},
      ISBN = {0-471-19745-9},
   MRCLASS = {60B10 (28A33 60F17)},
  MRNUMBER = {1700749},
       DOI = {10.1002/9780470316962},
       URL = {https://doi.org/10.1002/9780470316962},
}

@book {BGT,
    AUTHOR = {Bingham, N. H. and Goldie, C. M. and Teugels, J. L.},
     TITLE = {Regular variation},
    SERIES = {Encyclopedia of Mathematics and its Applications},
    VOLUME = {27},
 PUBLISHER = {Cambridge University Press, Cambridge},
      YEAR = {1987},
     PAGES = {xx+491},
      ISBN = {0-521-30787-2},
   MRCLASS = {26A12 (11K65 11N60 30-02 40E05 60-02 60Fxx)},
  MRNUMBER = {898871},
MRREVIEWER = {R.\ A.\ Maller},
       DOI = {10.1017/CBO9780511721434},
       URL = {https://doi.org/10.1017/CBO9780511721434},
}

@book {HKPV,
    AUTHOR = {Hough, J. Ben and Krishnapur, Manjunath and Peres, Yuval and
              Vir\'ag, B\'alint},
     TITLE = {Zeros of {G}aussian analytic functions and determinantal point
              processes},
    SERIES = {University Lecture Series},
    VOLUME = {51},
 PUBLISHER = {American Mathematical Society, Providence, RI},
      YEAR = {2009},
     PAGES = {x+154},
      ISBN = {978-0-8218-4373-4},
   MRCLASS = {60G55 (30B20 30C15 60B20 60F10 60G15 65H04 82B31)},
  MRNUMBER = {2552864},
MRREVIEWER = {Dmitry\ Beliaev},
       DOI = {10.1090/ulect/051},
       URL = {https://doi.org/10.1090/ulect/051},
}

@article {Kab+Klim:2014,
    AUTHOR = {Kabluchko, Zakhar and Klimovsky, Anton},
     TITLE = {Complex random energy model: zeros and fluctuations},
   JOURNAL = {Probab. Theory Related Fields},
  FJOURNAL = {Probability Theory and Related Fields},
    VOLUME = {158},
      YEAR = {2014},
    NUMBER = {1-2},
     PAGES = {159--196},
      ISSN = {0178-8051,1432-2064},
   MRCLASS = {60G50 (30C15 60E07 60F05 60F17 60G15 82B44)},
  MRNUMBER = {3152783},
MRREVIEWER = {Lee-Peng\ Teo},
       DOI = {10.1007/s00440-013-0480-5},
       URL = {https://doi.org/10.1007/s00440-013-0480-5},
}

@book {Shiryaev,
    AUTHOR = {Shiryaev, A. N.},
     TITLE = {Probability},
   EDITION = {Second},
 PUBLISHER = {``Nauka'', Moscow},
      YEAR = {1989},
     PAGES = {640},
      ISBN = {5-02-013955-6},
   MRCLASS = {60-01},
  MRNUMBER = {1024077},
}

@article {Pritsker+Yeager:2015,
    AUTHOR = {Pritsker, Igor E. and Yeager, Aaron M.},
     TITLE = {Zeros of polynomials with random coefficients},
   JOURNAL = {J. Approx. Theory},
  FJOURNAL = {Journal of Approximation Theory},
    VOLUME = {189},
      YEAR = {2015},
     PAGES = {88--100},
      ISSN = {0021-9045,1096-0430},
   MRCLASS = {30C10 (12D05 12D10 30C15 60F99 60G99)},
  MRNUMBER = {3280673},
MRREVIEWER = {E.\ J.\ Barbeau},
       DOI = {10.1016/j.jat.2014.09.003},
       URL = {https://doi.org/10.1016/j.jat.2014.09.003},
}

@incollection {Shirai:2012,
    AUTHOR = {Shirai, Tomoyuki},
     TITLE = {Limit theorems for random analytic functions and their zeros},
 BOOKTITLE = {Functions in number theory and their probabilistic aspects},
    SERIES = {RIMS K\^oky\^uroku Bessatsu},
    VOLUME = {B34},
     PAGES = {335--359},
 PUBLISHER = {Res. Inst. Math. Sci. (RIMS), Kyoto},
      YEAR = {2012},
   MRCLASS = {60G55 (30B20 30C10 60F05 60F17)},
  MRNUMBER = {3014854},
MRREVIEWER = {Zakhar\ Kabluchko},
}

@article {Bojanic+Seneta:1971,
    AUTHOR = {Bojani\'c, R. and Seneta, E.},
     TITLE = {Slowly varying functions and asymptotic relations},
   JOURNAL = {J. Math. Anal. Appl.},
  FJOURNAL = {Journal of Mathematical Analysis and Applications},
    VOLUME = {34},
      YEAR = {1971},
     PAGES = {302--315},
      ISSN = {0022-247X},
   MRCLASS = {26.49},
  MRNUMBER = {274676},
MRREVIEWER = {Y.\ L.\ Luke},
       DOI = {10.1016/0022-247X(71)90114-4},
       URL = {https://doi.org/10.1016/0022-247X(71)90114-4},
}

@article {Hunt:1951,
    AUTHOR = {Hunt, G. A.},
     TITLE = {Random {F}ourier transforms},
   JOURNAL = {Trans. Amer. Math. Soc.},
  FJOURNAL = {Transactions of the American Mathematical Society},
    VOLUME = {71},
      YEAR = {1951},
     PAGES = {38--69},
      ISSN = {0002-9947,1088-6850},
   MRCLASS = {42.4X},
  MRNUMBER = {51340},
MRREVIEWER = {M.\ Kac},
       DOI = {10.2307/1990858},
       URL = {https://doi.org/10.2307/1990858},
}

@article{Bulinskaya:1961,
author = {Bulinskaya, E. V.},
title = {On the Mean Number of Crossings of a Level by a Stationary {G}aussian Process},
journal = {Theory of Probability \& Its Applications},
volume = {6},
number = {4},
pages = {435-438},
year = {1961},
doi = {10.1137/1106059},
}

@article{Diaconis+Shahshahani:1994,
    AUTHOR = {Diaconis, P. and Shahshahani, M.},
     TITLE = {On the eigenvalues of random matrices},
   JOURNAL = {J. Appl. Probab.},
  FJOURNAL = {Journal of Applied Probability},
    VOLUME = {31A},
      YEAR = {1994},
     PAGES = {49--62},
      ISSN = {0021-9002},
   MRCLASS = {60B15 (15A18 15A52 60F05)},
  MRNUMBER = {1274717},
MRREVIEWER = {Daniel Rockmore},
       DOI = {10.2307/3214948},
       URL = {https://doi.org/10.2307/3214948},
}

@article {Hughes+Keating+OConnell:2001,
    AUTHOR = {Hughes, C. P. and Keating, J. P. and O'Connell, Neil},
     TITLE = {On the characteristic polynomial of a random unitary matrix},
   JOURNAL = {Comm. Math. Phys.},
  FJOURNAL = {Communications in Mathematical Physics},
    VOLUME = {220},
      YEAR = {2001},
    NUMBER = {2},
     PAGES = {429--451},
      ISSN = {0010-3616},
   MRCLASS = {82B41 (11Z05 33C90 60F10)},
  MRNUMBER = {1844632},
MRREVIEWER = {Estelle L. Basor},
       DOI = {10.1007/s002200100453},
       URL = {https://doi.org/10.1007/s002200100453},
}

@article{Chhaibi+Madaule+Najnudel:2018,
    AUTHOR = {Chhaibi, Reda and Madaule, Thomas and Najnudel, Joseph},
     TITLE = {On the maximum of the {${\rm C}\beta {\rm E}$} field},
   JOURNAL = {Duke Math. J.},
  FJOURNAL = {Duke Mathematical Journal},
    VOLUME = {167},
      YEAR = {2018},
    NUMBER = {12},
     PAGES = {2243--2345},
      ISSN = {0012-7094},
   MRCLASS = {60B20 (60G70)},
  MRNUMBER = {3848391},
MRREVIEWER = {Nizar Demni},
       DOI = {10.1215/00127094-2018-0016},
       URL = {https://doi.org/10.1215/00127094-2018-0016},
}

@ARTICLE{Do+Nguyen+Nguyen:2022,
	author = {Do, Yen and Nguyen, Hoi H. and Nguyen, Oanh},
	title = {Random trigonometric polynomials: Universality and non-universality of the variance for the number of real roots},
	year = {2022},
	journal = {Annales de l'institut Henri Poincare (B) Probability and Statistics},
	volume = {58},
	number = {3},
	pages = {1460--1504},
}

@ARTICLE{Angst+Pautrel+Poly:2022,
	author = {Angst, J\"{u}rgen and Pautrel, Thibault and Poly, Guillaume},
	title = {Real Zeros of Random Trigonometric Polynomials with Dependent Coefficients},
	year = {2022},
	journal = {Transactions of the American Mathematical Society},
	volume = {375},
	number = {10},
	pages = {7209--7260},
}

@article{Buraczewski+Dong+Iksanov+Marynych:2023,
title = {Limit theorems for random {D}irichlet series},
journal = {Stochastic Processes and their Applications},
volume = {165},
pages = {246-274},
year = {2023},
author = {Buraczewski, Dariusz  and Dong, Congzao and Iksanov, Alexander and Marynych, Alexander},
}

\end{document}